\documentclass[a4paper]{amsart} 
\usepackage{amsmath,amssymb,amscd} 
\usepackage{times}
\usepackage{mathrsfs}
\usepackage[square,comma,numbers,sort&compress]{natbib}
\usepackage{stmaryrd}
\usepackage{graphicx}
\allowdisplaybreaks

\theoremstyle{plain}
\newtheorem{theorem}{Theorem}[section]
\newtheorem{proposition}[theorem]{Proposition}
\newtheorem{lemma}[theorem]{Lemma}

\theoremstyle{definition}
\newtheorem{definition}[theorem]{Definition}

\theoremstyle{remark}
\newtheorem{remark}[theorem]{Remark}

\numberwithin{equation}{section}

\newcommand{\p}{\partial}
\newcommand{\ve}{\varepsilon}
\newcommand{\f}{\frac}
\newcommand{\na}{\nabla}

\newcommand{\al}{\alpha}
\renewcommand{\t}{\tilde}
\newcommand{\q}{\quad}
\newcommand{\vp}{\varphi}
\renewcommand{\O}{\Omega}
\renewcommand{\th}{\theta}
\newcommand{\g}{\gamma}
\newcommand{\G}{\Gamma}
\newcommand{\R}{\mathbb R}

\newcommand{\dl}{\delta}

\newcommand{\ds}{\displaystyle}

\title[Low regularity solutions of degenerate hyperbolic
  equations]{The existence and singularity structure of low
  regularity \\ solutions of higher-order degenerate hyperbolic
  equations}

\author[Z.-P.~Ruan]{Zhuoping Ruan}
\address{Department of Mathematics and IMS, Nanjing University, Nanjing 210093,
  P.R.~of China}
\email{XXXX\@nju.edu.cn}

\author[I.~Witt]{Ingo Witt}
\address{Mathematical Institute, University of G\"ottingen, D-37073
  G\"ottingen, Germany} 
\email{iwitt@uni-math.gwdg.de}

\author[H.-C.~Yin]{Huicheng Yin}
\address{Department of Mathematics and IMS, Nanjing University, Nanjing 210093,
  P.R.~of China}
\email{huicheng@nju.edu.cn}

\thanks{Ruan Zhuoping and Yin Huicheng were supported by the NSFC
(No.~10931007, No.~11025105), by the Priority Academic Program
Development of Jiangsu Higher Education Institutions, and by the DFG
via the Sino-German research project ``Analysis of PDEs and
Applications.'' Ingo Witt was partially supported by the DFG via the
Sino-German research project ``Analysis of PDEs and Applications.''}



\begin{document}

\begin{abstract} 
This paper is a continuation of our previous work \cite{rwy12}, where
we have established that, for the second-order degenerate hyperbolic
equation $\left(\p_t^2-t^{m}\Delta_x\right)u=f(t,x,u)$, locally
bounded, piecewise smooth solutions $u(t,x)$ exist when the initial
data $\left(u,\p_t u\right)(0,x)$ belongs to suitable conormal
classes. In the present paper, we will study low regularity solutions
of higher-order degenerate hyperbolic equations in the category of
discontinuous and even unbounded functions. More specifically, we are
concerned with the local existence and singularity structure of low
regularity solutions of the higher-order degenerate hyperbolic
equations $\p_t\left(\p_t^2-t^{m}\Delta_x\right)u=f(t,x,u)$ and
$\left(\p_t^2-t^{m_1}\Delta_x\right)\left(\p_t^2-t^{m_2}\Delta_x\right)v=f(t,x,v)$
in $\mathbb R_+\times \mathbb R^n$ with discontinuous initial data
$\p_t^iu(0,x)=\vp_i(x)$ ($0\le i\le 2$) and $\p_t^jv(0,x)=\psi_j(x)$
($0\le j\le 3$), respectively; here $m, m_1, m_2\in\mathbb N$,
$m_1\not=m_2$, $x \in \R^n$, $n\ge 2$, and $f$ is $C^{\infty}$ smooth
in its arguments. When the $\vp_i$ and $\psi_j$ are piecewise smooth
with respect to the hyperplane $\{x_1=0\}$ at $t=0$, we show that
local solutions $u(t,x),\, v(t,x)\in L^{\infty}((0,T)\times\mathbb
R^n)$ exist which are $C^{\infty}$ away from $\G_0\cup \G_m^{\pm}$ and
$\G_{m_1}^{\pm}\cup\G_{m_2}^{\pm}$ in $[0,T]\times\mathbb R^n$,
respectively; here $\G_0=\{(t,x)\colon\, t\ge 0, \, x_1=0\}$ and the
$\Gamma_k^{\pm} = \Bigl\{(t,x)\colon\, t\ge 0,
x_1=\pm\ds\f{2t^{(k+2)/2}}{k+2}\Bigr\}$ are two characteristic
surfaces forming a cusp.  When the $\vp_i$ and $\psi_j$ belong to $
C_0^{\infty}(\mathbb R^n\setminus\{0\})$ and are homogeneous of degree
zero close to $x=0$, then there exist local solutions $u(t,x),\,
v(t,x)\in L^{\infty}_{\text{loc}}((0,T]\times\mathbb R^n)$ which are
  $C^{\infty}$ away from $\G_m\cup l_0$ and $\G_{m_1}\cup\G_{m_2} $ in
  $[0,T]\times\mathbb R^n$, respectively; here
  $\Gamma_k=\Bigl\{(t,x)\colon\, t\ge 0, \,
  |x|^2=\ds\f{4t^{k+2}}{(k+2)^2}\Bigr\}$ ($k=m, m_1, m_2$) is a
  cuspidal conic surface (``forward light cone'') and
  $l_0=\{(t,x)\colon\, t\ge 0,\, |x|=0\}$ is a ray.
\end{abstract}

\maketitle


\section{Introduction}\label{sec1}

In this paper, we shall study the local existence and singularity
structure of low regularity solutions of the higher-order degenerate
hyperbolic equations
\begin{equation}\label{1-1}
\left\{ \enspace
\begin{aligned} 
&\p_t\left(\p_t^2-t^{m}\Delta_x\right)u=f(t,x,u), \quad (t,x)\in (0,
+\infty)\times\mathbb R^n,\\
&\p_t^j u(0,x)=\vp_j(x),\quad  0 \le j \le 2,
\end{aligned}
\right.
\end{equation}
and
\begin{equation}\label{1-2}
\left\{ \enspace
\begin{aligned} 
&\left(\p_t^2-t^{m_1}\Delta_x\right)\left(\p_t^2-t^{m_2}\Delta_x\right)u
=f(t,x,u),\quad (t,x)\in (0,+\infty)\times\mathbb R^n,\\
&\p_t^k u(0,x)=\psi_k(x), \quad 0\le k \le 3,
\end{aligned}
\right.
\end{equation}
where $m, m_1, m_2\in\mathbb N$, $m_1\not=m_2$, $x \in \mathbb R^n$,
$n\ge 2$, $f$ is $C^{\infty}$ in its arguments and has compact support
with respect to the variable $x=(x_1, \dots, x_n)$. The discontinuous
initial data $\varphi_j(x)$ ($0\le j \le 2$) and $\psi_k(x)$ ($0 \le k
\le 3$) satisfy one of the following assumptions: 
\begin{equation}\label{a1}
\vp_j(x)=\begin{cases} \vp_{j1}(x)\quad\text{for
$x_1>0$},\\
\vp_{j2}(x)\quad\text{for
$x_1<0$},
\end{cases}   \qquad   
\psi_k(x)=\begin{cases} \psi_{k1}(x)
\quad\text{for
$x_1>0$},\\
\psi_{k2}(x)\quad\text{for
$x_1<0$},
\end{cases} \tag{$\text{A}_1$}
\end{equation}
where $\vp_{j1}, \vp_{j2}, \psi_{k1}, \psi_{k2}
\in C_0^{\infty}(\mathbb R^n)$ with $\vp_{j1}(0) \neq \vp_{j2}(0)$ and
$\psi_{k1}(0) \neq \psi_{k2}(0)$;
\begin{equation}\label{a2}
\vp_j(x)=g_j\left(x, \f{x}{|x|}\right), \qquad
\psi_k(x)=h_k\left(x, \f{x}{|x|}\right), \tag{$\text{A}_2$}
\end{equation}
where $g_j(x,y)$ and $h_k(x,y)\in C^{\infty}(\mathbb R^n\times\mathbb R^n)$
have compact support in $B(0,1)\times B(0, 2)$.


\pagebreak

Under assumptions \eqref{a1} and \eqref{a2}, we will prove the
following main results:

\begin{theorem}\label{thm1-1} 
Let assumption \eqref{a1} hold. Then there is a constant $T>0$
such that the following holds true\/\textup{:}

\textup{(i)} \ \textup{Eq.~\eqref{1-1}} admits a unique solution $u\in
L^{\infty}((0, T)\times\mathbb R^n)\cap C([0, T],$ $H^{1/2-}(\mathbb
R^n)) \cap C((0, T],$ \linebreak $H^{\f{m+1}{m+2}-}(\mathbb R^n))\cap C^1([0, T],
  H^{-\f{1}{m+2}-}(\mathbb R^n))$. Moreover, $u\in
  C^{\infty}\left(([0, T]\times\mathbb R^n)\setminus(\G_m^{\pm} \cup
  \G_0)\right)$, where $\Gamma_m^{\pm}=\Bigl\{(t,x)\colon\, t\ge 0,\,
  x_1=\pm\,\ds\f{2t^{(m+2)/2}}{m+2}\Bigr\}$ and $\G_0=\{(t,x)\colon\,
  t\ge 0, \, x_1=0\}$.

\textup{(ii)} \ \textup{Eq.~\eqref{1-2}} admits a unique solution $u\in
L^{\infty}((0, T)\times\mathbb R^n)\cap C([0, T],H^{1/2-}(\mathbb
R^n)) \cap C((0, T],$ \linebreak $H^{\f{m_2+1}{m_2+2}-}(\mathbb
  R^n))\cap C^1([0, T], H^{-\f{1}{m_2 +2}-}(\mathbb R^n))$. Moreover,
  $u\in C^{\infty}\left(([0, T]\times\mathbb
  R^n)\setminus(\G_{m_1}^{\pm}\cup\G_{m_2}^{\pm})\right)$, where
  $\Gamma_{m_i}^{\pm}=\Bigl\{(t,x)\colon\, t\ge 0,\,
  x_1=\pm\,\ds\f{2t^{(m_i+2)/2}}{m_i+2}\Bigl\}$ for $i=1, 2$.
\end{theorem}

\begin{theorem}\label{thm1-2} 
Let assumption \eqref{a2} hold. Further let $f$ satisfy
\[
\left|\p_{t,x}^{\al}\p_u^lf(t,x,u)\right|\le C_{T_0, \al, l}
\left(1+|u|\right)^{\max\{K-l, 0\}}
\] 
for $\al\in\mathbb N_0^{1+n}$, $l\in\mathbb N_0$, $0\le t\le T_0$,
where $K>0$ is fixed. Then there is a constant $0<T\le T_0$ such
that the following holds true\/\textup{:}

\textup{(i)} \ \textup{Eq.~\eqref{1-1}} admits a unique solution $u\in
L^{\infty}_{\textup{loc}}((0,T]\times\mathbb R^n)\cap
  C([0,T],H^{n/2-}(\mathbb R^n))\cap C((0, T],$ \linebreak
    $H^{n/2+\f{m}{2(m+2)}-}(\mathbb R^n))\cap C^1([0,T],
    H^{n/2-\f{m+4}{2(m+2)}-}(\mathbb R^n))$. Moreover, $u\in
    C^{\infty}\left(([0, T]\times\mathbb R^n)\setminus(\G_m \cup
    l_0)\right)$, where $\Gamma_m=\Bigl\{(t,x)\colon\, t\ge 0, \,
    |x|^2=\ds\f{4t^{m+2}}{(m+2)^2}\Bigr\}$ and $l_0=\{(t,x)\colon\,
    t\ge 0,\, |x|=0\}$.

\textup{(ii)} \ \textup{Eq.~\eqref{1-2}} admits a unique solution
$u\in L_{\textup{loc}}^{\infty}((0, T]\times\mathbb R^n)\cap C([0,T],
  H^{n/2-}(\mathbb R^n))\cap C((0, T],$ \linebreak $
    H^{n/2+\f{m_2}{2(m_2+2)}-}(\mathbb R^n))\cap
    C^1([0,T],H^{n/2-\f{m_2+4}{2(m_2+2)}-}(\mathbb
    R^n))$. Moreover, $u\in C^{\infty}\left(([0, T]\times\mathbb
    R^n)\setminus(\G_{m_1}\cup\G_{m_2})\right)$, where
    $\Gamma_{m_i}=\Bigl\{(t,x)\colon\, t\ge 0,\,
    |x|^2=\ds\f{4t^{m_i+2}}{(m_i+2)^2} \Bigr\}$ for $i=1, 2$.
\end{theorem}

\begin{remark}\label{rem1-1} 
Consult the following figures to see the singularity structure of the
solutions as descripted by Theorem~\ref{thm1-1} and
Theorem~\ref{thm1-2}, respectively.
\end{remark}

\begin{figure}[ht]
\centering
\includegraphics[height=75mm]{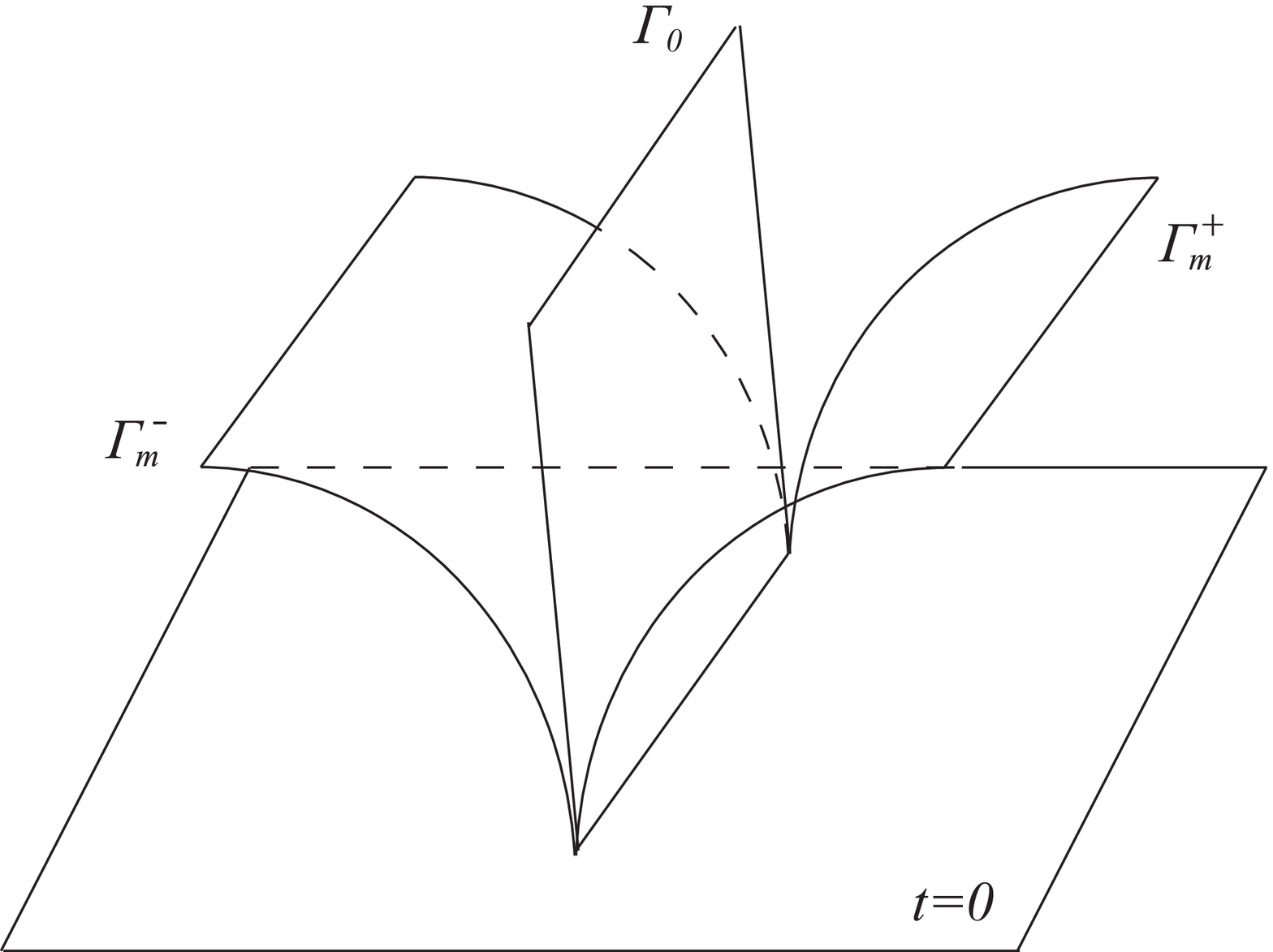}
\caption{The singularity set $\G_{m}^{\pm}\cup
\G_{0}$ of the solution $u(t,x)$ of \eqref{1-1} under assumption \eqref{a1}}
\end{figure}

\begin{figure}[ht]
\centering
\includegraphics[height=75mm]{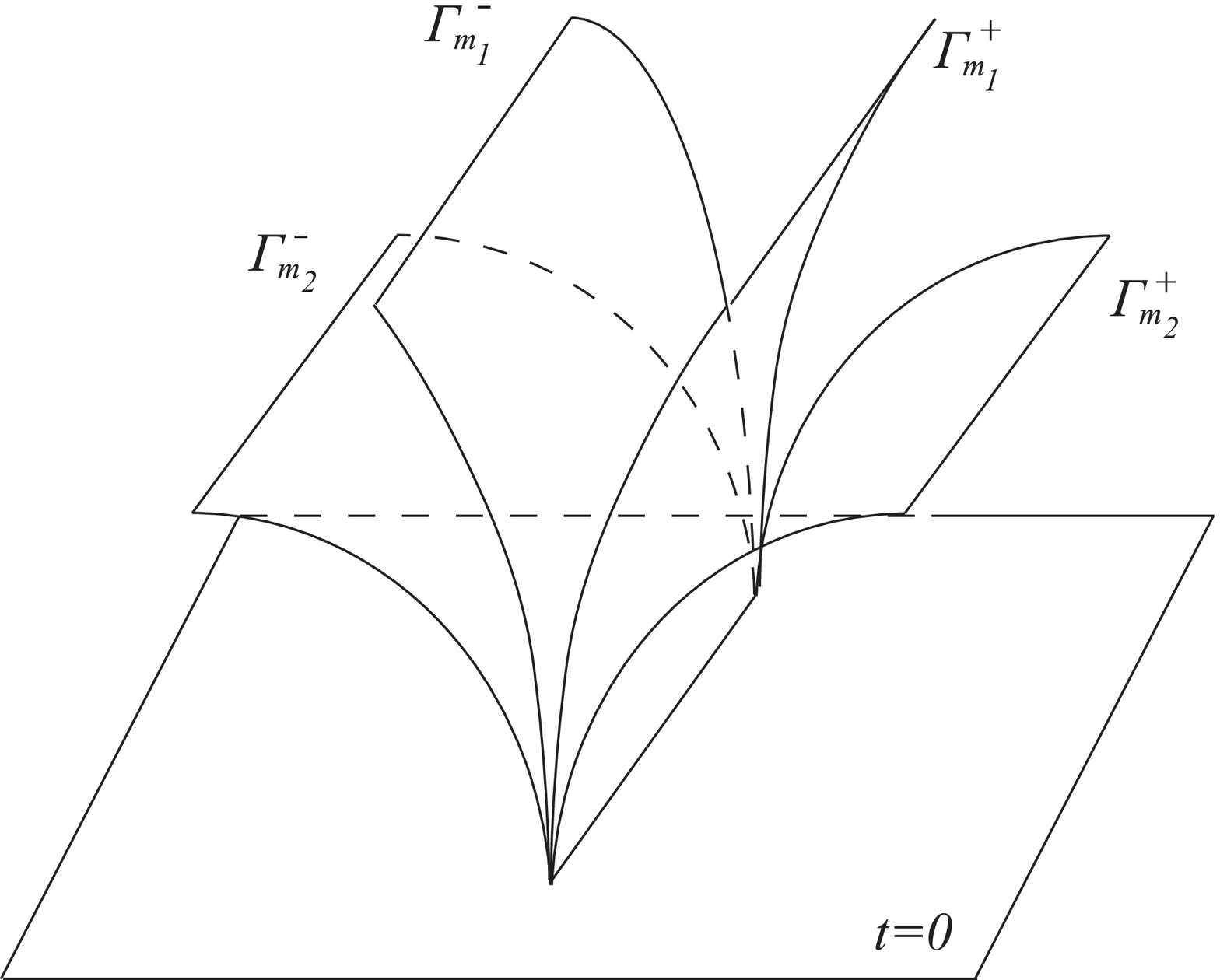}
\caption{The singularity set $\G_{m_1}^{\pm}\cup\G_{m_2}^{\pm}$ of
the solution $u(t,x)$ of \eqref{1-2} under assumption \eqref{a1}}
\end{figure}

\begin{figure}[ht]
\centering
\includegraphics[height=75mm]{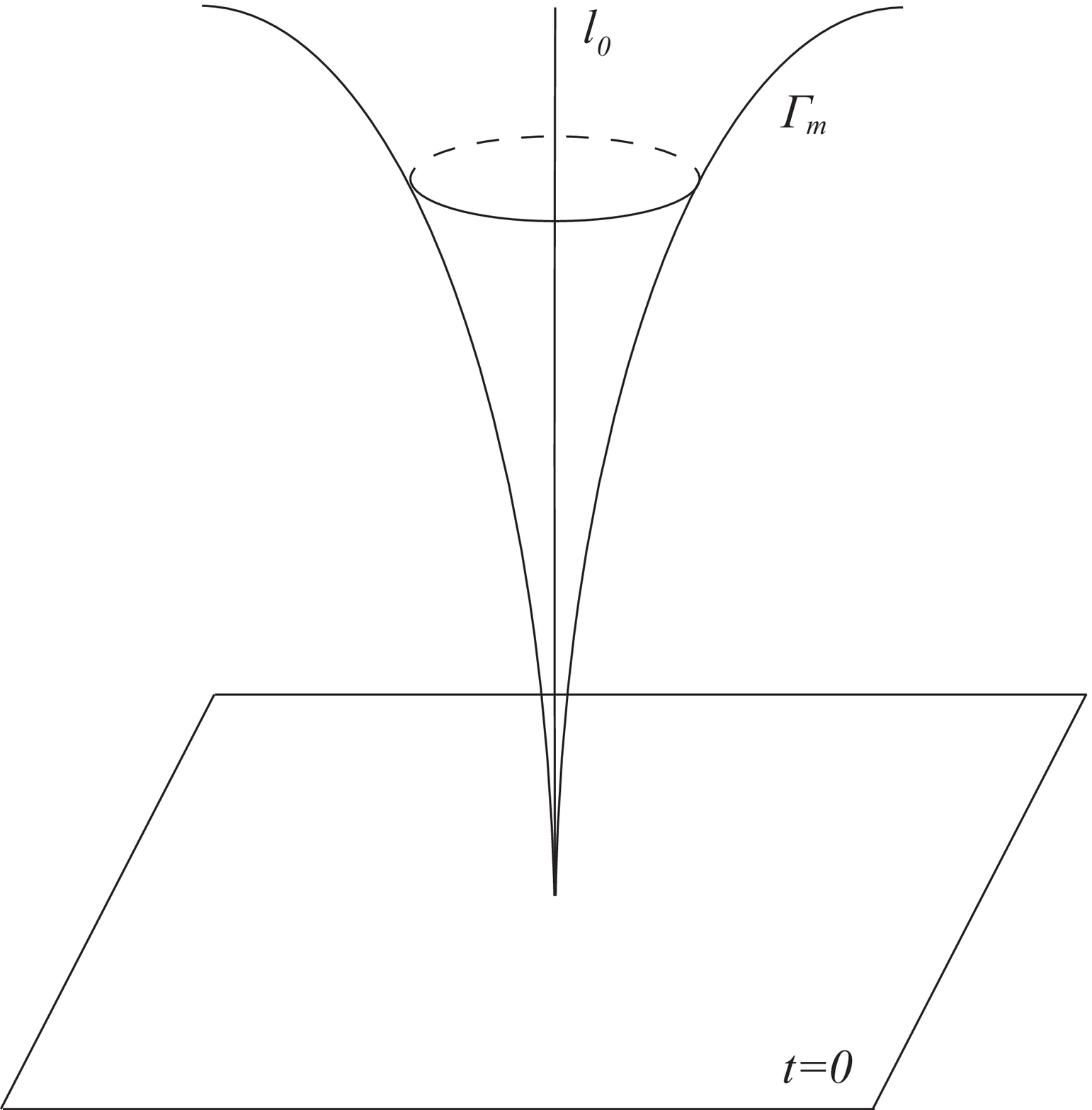}
\caption{The singularity set $\Gamma_m\cup l_0$ of
  the solution $u(t,x)$ of \eqref{1-1} under assumption \eqref{a2}}
\end{figure}

\begin{figure}[ht]
\centering
\includegraphics[height=75mm]{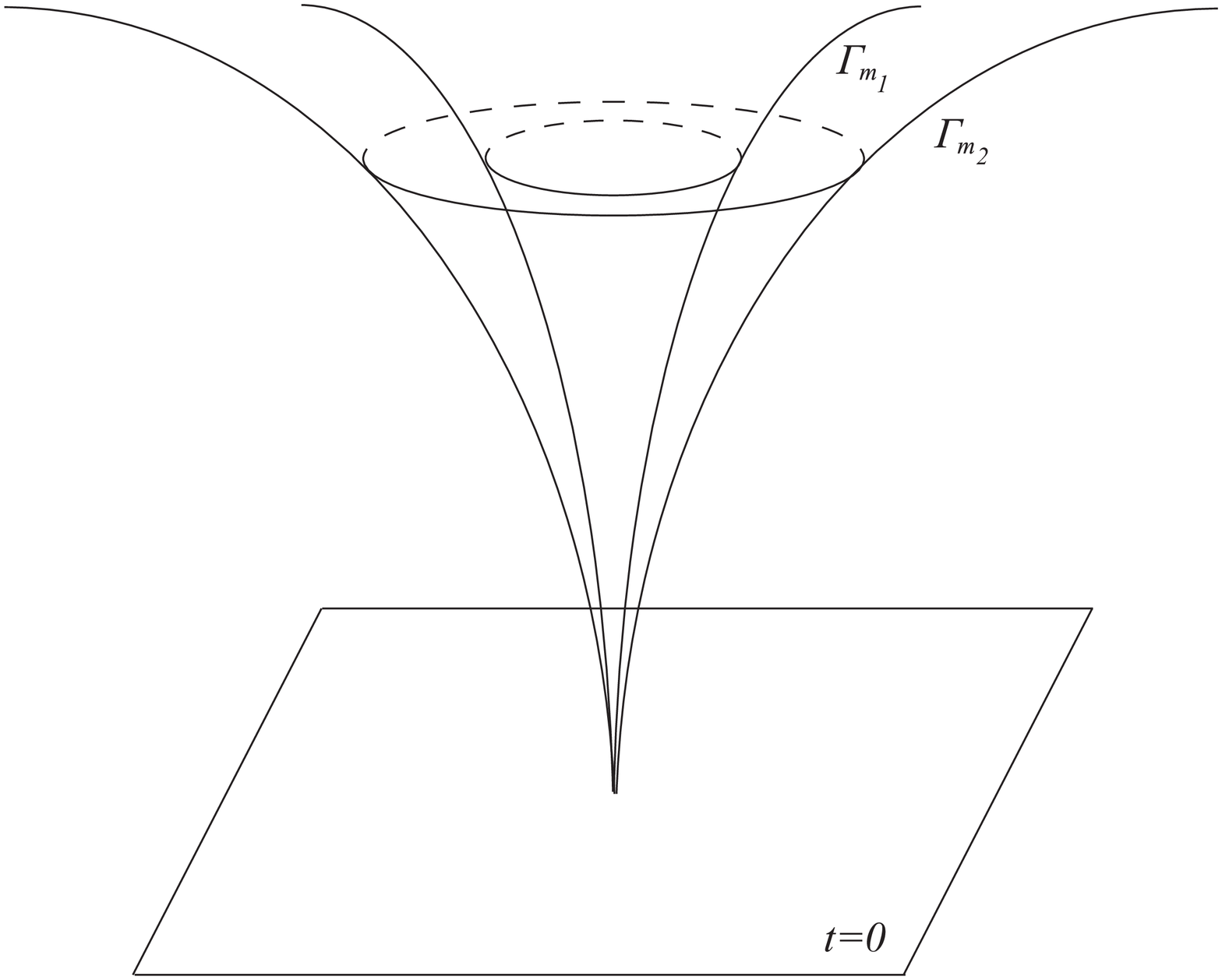}
\caption{The singularity set $\G_{m_1}\cup\G_{m_2}$ of the solution
  $u(t,x)$ of \eqref{1-2} under assumption \eqref{a2}}
\end{figure}

\begin{remark}\label{rem1-2} 
For $n=m=1$, the operator $\p_t^2-t\p_x^2$ is the classical Tricomi
operator that arises, e.g., in continuous transonic gas dynamics of
isentropic and irrotational flow. The principal symbol of the
third-order operator $\p_t\left(\p_t^2-t\p_x^2\right)$ in \eqref{1-1}
resembles the one of the 2-D steady compressible isentropic Euler
system in continuous transonic gas dynamics. Indeed, when introducing
the flux function $\psi(x)$ and the generalized potential $\vp(x)$ as
independent variables in place of the spatial variables $(x_1,x_2)$,
one arrives at a system the linearization of which has principal
symbol $\tau\left(\tau^2-\psi\xi^2\right)$ for $\psi\ge 0$. The latter
has three simple real eigenvalues for $\psi>0$, where all three of
them merge into one at the sonic line $\psi=0$. (See \cite[Chapter
  2]{kuz92} for details.)
\end{remark}

\begin{remark}\label{rem1-3} 
For the multi-dimensional compressible Euler system and initial data
which is $H^s$ ($s> n/2+5$) conormal with respect to the origin,
J.-Y.~Chemin \cite{che90b} has shown that the classical solution is
(weakly) singular only along the set $\Gamma\cup l$ (see Figure~5),
where $\Gamma$ is the characteristic conic surface and $l$ is the
stream curve both emanating from the origin.

For the quasilinear equation
$\Bigl(\p_t^2-\ds\sum_{i=1}^nc_i^2(t,x,\na_{t,x} u)\p_i^2\Bigr)
\Bigl(\p_t^2-\ds\sum_{i=1}^nd_i^2(t,x,\na_{t,x} u)\p_i^2\Bigr)u
=f\bigl(t,x,$ \linebreak $\{\na_{t,x}^{\al}u\}_{|\al|\le 3}\bigr)$,
which is strictly hyperbolic with respect to time $t$, and initial
data $\p_t^ju|_{t=0}\in H^{s-j, \infty}(0)$ $(0\leq j\leq 3)$ conormal
with respect to the origin, where $s>\ds(n+1)/2+9$, it has been shown
in \cite{che90a} that the local classical solution $u(t,x)\in
H^{s}_{\text{loc}}(\mathbb R_+^{n+1})$ is (weakly) singular only along
the two characteristic conic surfaces $\Gamma_1$ and $\Gamma_2$
emanating the origin (see Figure~6).

Compared to solutions of higher regularity studied in
\cite{che90a,che90b}, Theorem~\ref{thm1-2} deals with unbounded and
discontinuous solutions.
\end{remark}

\begin{figure}[ht]
\centering
\includegraphics[height=80mm]{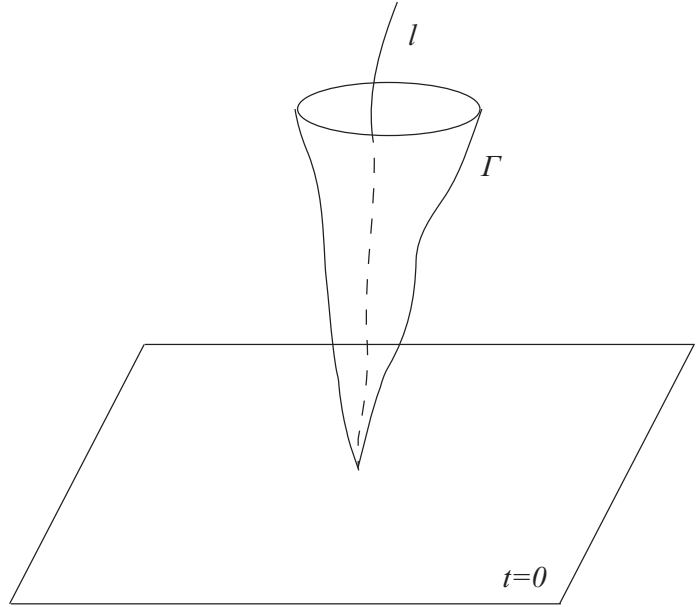}
\caption{The weak singularity set $\Gamma\cup l$} 
\end{figure}

\begin{figure}[ht]
\centering
\includegraphics[height=75mm]{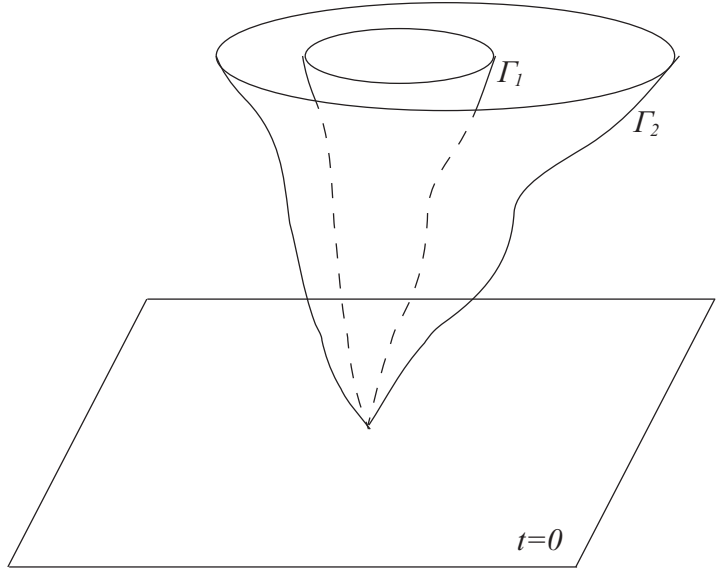}
\caption{The weak singularity set $\Gamma_1\cup\Gamma_2$} 
\end{figure}

\begin{remark}\label{rem1-4} 
Utilizing the technique of edge Sobolev space as in
\cite{drwi05a,drwi05b}, and microlocal analysis tools,
Theorems~\ref{thm1-1} and \ref{thm1-2} can be extended to more general
third-order and fourth-order semilinear degenerate hyperbolic
equations of the form
\[
\biggl(\p_t+t^l\ds\sum_{j=1}^na_j(t,x)\p_j\biggr)
\biggl(\p_t^2+2t^l\ds\sum_{j=1}^nb_{0j}(t,x)\p_t\p_j
-t^{2l}\ds\sum_{i,j=1}^nb_{ij}(t,x)\p_{ij}\biggr)u=f(t,x,u)
\] 
and
\begin{multline*}
\biggl(\p_t^2+2t^l\ds\sum_{j=1}^na_{0j}(t,x)\p_t\p_{j}
-t^{2l}\ds\sum_{i,j=1}^na_{ij}(t,x)\p_{ij}\biggr)\biggl(\p_t^2
+2t^l\ds\sum_{j=1}^nc_{0j}(t,x)\p_t\p_{j}
-t^{2l}\ds\sum_{i,j=1}^nc_{ij}(t,x)\p_{ij}\biggr)u \\
=f(t,x,u),
\end{multline*}
respectively; here $l\in\mathbb N$. We shall study even more general
nonlinear degenerate hyperbolic equations in a forthcoming paper.
\end{remark}

\begin{remark}\label{rem1-5} 
For the semilinear $N\times N$ strictly hyperbolic system
$\p_tU+\ds\sum_{j=1}^nA_j(t,x)\p_jU=F(t,x,U)$ with piecewise smooth
initial data or initial data which is $H^s$ conormal with respect to
some $C^{\infty}$ hypersurface $\Delta_0\subset\R^n$ (where $s>n/2$)
(in particular, this includes discontinuous Riemann initial data), the
local well-posedness of piecewise smooth solutions and solutions that
are $H^s$ conormal with respect to the $N$ pairwise transverse
characteristic surfaces $\Sigma_j$ passing through $\Delta_0$,
respectively, has been established in
\cite{bon80,bon82,met86,mera90}. In the present paper, in
Theorem~\ref{thm1-1}, we establish the corresponding result for
higher-order semilinear degenerate hyperbolic equations.
\end{remark}

\begin{remark}\label{rem1-6} 
For the second-order generalized Tricomi equation
\begin{equation}\label{1-3}
\left\{ \enspace
\begin{aligned}
&\p_t^2 u-t^{m}\Delta_x u=f(t,x,u), \qquad (t,x)\in (0,
+\infty)\times\mathbb R^n,\\ &\p_t^j u(0,x)=\vp_j(x),\quad j=0, 1,
\end{aligned}
\right.
\end{equation}
where $m\in\mathbb N$, $n\ge 2$, $f(t,x,u)$ and $\vp_j(x)$ ($j=0,1$)
satisfy all the assumptions of Theorems~\ref{thm1-1} and \ref{thm1-2},
by the same method one obtains conclusions analogous to those of
Theorems~\ref{thm1-1} and \ref{thm1-2}. Now we have weaken the
regularity assumptions of \cite{rwy12} on the initial data inasmuch as
there $u(0,x)=0$ holds.
\end{remark}

\begin{remark}\label{rem1-7} 
If we are only concerned with the local existence of solutions of
\eqref{1-1} and \eqref{1-2}, then
it is enough to assume the nonlinearity $f$ be of class $C^1$. For
instance, $f=\pm\, |u|^p$ or $f=\pm\, |u|^{p-1}u$ with $p>1$ will
do. By results of \cite{yag06}, one has that in general weak
solutions of \eqref{1-1} and \eqref{1-2} blow up in finite time.
\end{remark}

\begin{remark}\label{rem1-8} 
Because of the low regularity of initial data near the origin when
assumption~\eqref{a2} holds, it seems to be difficult to show $u\in
L^{\infty}((0, T)\times\mathbb R^n)$ in Theorem~\ref{thm1-2}. In fact,
even for the linear equation $\p_t^2w-t^m\Delta_x w=0$ with initial
data $(w(0,x), \p_tw(0,x))=(\vp_0(x), \vp_1(x))$, where $\vp_0(x)$ and
$\vp_1(x)$ satisfy \eqref{a2}, the solution $w(t,x)$ can be shown only
to satisfy $|w(t,x)|\le C_T(1+|\ln t|^2)$ for $0<t\leq T$ (see
Lemma~\ref{lem2-4}\,(ii)). In this case, the polynomial bound on
$f(t,x,u)$ with respect to the variable $u$ (a suitable exponential
bound would do as well) is necessary to guarantee that
$f(t,x,u(t,x))\in L^1((0, T)\times\mathbb R^n)$ in
Theorem~\ref{thm1-2} and then to obtain the local existence of
solutions in $L^{\infty}_{\textup{loc}}((0, T]\times\mathbb R^n)$ by a
  fixed-point argument.
\end{remark}

For the semilinear Tricomi equation $\p_t^2u-t\Delta_x u=f(t,x,u)$ and
initial data of regularity $H^s$ ($s>n/2$), M.~Beals \cite{bea92} has
proven the local existence of a classical solution $u\in C([0, T],
H^s(\mathbb R^n))\cap C^1([0, T],$ \linebreak $H^{s-5/6}(\mathbb
R^n))\cap C^2([0, T], H^{s-11/6}(\mathbb R^n))$ for some $T>0$ under
the assumption that the support of $f(t,x,u)$ with respect to the
variable $t$ lies in $\{t\ge 0\}$. Conormal regularity of the
classical $H^s$ solutions $u(t,x)$ with respect to the characteristic
cusp surfaces $x_1=\pm \, \ds 3t^{3/2}/2$ has also been established in
\cite{bea92}. For more general nonlinear degenerate hyperbolic
equations with data of higher regularity, the authors of
\cite{drre98,drre00} have studied the local existence and propagation
of weak singularity of classical solutions. For the Cauchy problem for
linear degenerate hyperbolic equations, there are rather complete
results on the well-posedness and the regularity of solutions (see
\cite{cosp82,han10,hhl06,hoe77,ivr76,ole70,shi91,tato80} and the
references therein). In \cite{rwy12}, we have established that bounded
and piecewise smooth solutions $u(t,x)$ exist locally for the
second-order semilinear degenerate hyperbolic equation
$\left(\p_t^2-t^{m}\Delta_x\right)u=f(t,x,u)$, where $u(0,x)$ is
continuous and piecewise smooth, while $\p_t u(0,x)$ is piecewise
smooth, but might be discontinuous. In the present paper, we will
focus on solutions (of even lower regularity) of higher-order
degenerate hyperbolic equation in the category of piecewise smooth and
possibly unbounded functions.

We now comment on the proofs of Theorems~\ref{thm1-1}
and~\ref{thm1-2}. In order to prove the local existence of solutions
of \eqref{1-1} and \eqref{1-2} with the low regularity as given, we
first will establish $L^{\infty}$ (or $L_{\textup{loc}}^{\infty}$)
bounds on the solutions $v(t,x)$ of the linear problem
$\p_t^2v-t^m\Delta_x v=F(t,x)$ with discontinuous initial data
$(v(0,x),\p_tv(0,x))=(\vp_0(x), \vp_1(x))$ so that nonlinear
superposition $v\mapsto f(t,x,v)$ be well-defined for $t>0$. When
doing this, we shall make full use of the special structure of the
piecewise smooth and single-point singular initial data, respectively,
as well as some tools from harmonic analysis such as the Hilbert
transformation and Fourier analysis methods. This is necessary as the
energy method and Sobolev embedding theorems cannot be applied
directly to obtain $v(t,x)\in L_{\textup{loc}}^{\infty}$ because of
its low $H^s$ regularity (with $s<n/2$).  (For instance, initial data
is in $H^{1/2-}(\mathbb R^n)$ in case of assumption \eqref{a1} and in
$H^{n/2-}(\mathbb R^n)$ in case of assumption \eqref{a2},
respectively.) Based on these $L^{\infty}$ (or
$L_{\textup{loc}}^{\infty}$) estimates and invoking the theory of
confluent hypergeometric functions, we construct suitable nonlinear
maps related to problems \eqref{1-1} and \eqref{1-2}, respectively,
and further show that these maps possess fixed points in the space
$L^{\infty}((0, T)\times\mathbb R^n)\cap C([0, T], H^{s_0}(\mathbb
R^n))$ for some $T>0$ and a suitable regularity $s_0>0$.  This then
establishes the local solvability of \eqref{1-1} and \eqref{1-2}.
After this, we start to deal with the singularity structure of the
solutions $u(t,x)$ of \eqref{1-1} and \eqref{1-2}. Note that the
initial data is conormal with respect to the hypersurface $\{x_1=0\}$
under assumptions \eqref{a1} and conormal with respect to the origin
$\{x=0\}$ under assumption \eqref{a2}, i.e., it holds
$(x_1\p_1)^{k_1}\ds\prod_{2\le i\le n}\p_i^{k_i}\vp_l(x)\in
H^{1/2-}(\mathbb R^n)$ for all $k_i\in\mathbb N_0$ ($1\le i\le n$) in
the first case and $\ds\prod_{1\le i,j\le
  n}(x_i\p_j)^{k_{ij}}\vp_l(x)\in H^{n/2-}(\mathbb R^n)$ for all
$k_{ij}\in\mathbb N_0$ in the second case.  We then intend to use
commutator arguments as in \cite{bea92,bon80,bon82} to prove
conormality of the solutions $u(t,x)$ of \eqref{1-1} and
\eqref{1-2}. The fact that the hypersurfaces $\G_l$, $\G_l^{\pm}$
($l=m,\, m_1,\ m_2$) form cusp singularities, however, makes it
difficult to use directly smooth vector fields $Z_1,\dots, Z_q$
tangent to $\G_l$ and to $\G_l^{\pm}$, respectively, to define the
conormal spaces and to perform the analysis of the commutators
$\bigl[\p_t\left(\p_t^2-t^m\Delta\right), Z_1^{l_1}\dots
  Z_q^{l_q}\bigr]$ and
$\bigl[\left(\p_t^2-t^{m_1}\Delta\right)\left(\p_t^2-t^{m_2}\Delta\right),
  Z_1^{l_1}\dots Z_q^{l_q}\bigr]$, since this leads to a violation of
the Levi condition on lower-order terms which results in a loss of
regularity for $Z_1^{l_1}\dots Z_q^{l_q}u$. (More detailed
explanations can be found in Remark~\ref{rem3-1}.) Motivated by
\cite{bea88,bea92,rwy12}, to overcome these difficulties we shall work
with nonsmooth vector fields instead and gain extra regularity by some
specific relations provided by the operator under study itself and
some parts of these vector fields (see Proposition~\ref{prop3-3}) to
obtain the full conormal regularity of the solutions $u(t,x)$.
This completes the proofs of Theorems~\ref{thm1-1}
and~\ref{thm1-2}. We point out that although some of the statements in
this paper are analogous to those of \cite{rwy12}, here due to the
lower regularity of the initial data and the higher order of the
degenerate hyperbolic equations under consideration, we have to
perform a more thorough analysis including a more technically involved
treatment of the linear problems. This is caused by the fact that some
commutator relations turn out not to be ``good'' in the sense that
some of the resulting coefficients are not admissible (see
Definition~\ref{def3-1} for the notion of an admissible function).
(Compare the expressions for $[P_1, \bar{V}_i^{(m)}]$ in
Lemma~\ref{lem3-1}. Likewise, the operators $\p_t^2-t^{m_1}\Delta$ and
$\p_t^2-t^{m_2}\Delta$ $(m_1\not=m_2)$ have different ``radial''
vector fields $2t\p_t+(m_1+2)(x_1\p_1+\cdots+x_n\p_n)$ and
$2t\p_t+(m_2+2)(x_1\p_1+\cdots+x_n\p_n)$ which needs a special
treatment in the proof of Theorem~\ref{thm1-2} in Section~\ref{sec6}.)

This paper is organized as follows: In Section~\ref{sec2}, we provide
preliminary results and prove $L^{\infty}$ (or
$L_{\text{loc}}^{\infty}$) bounds on the solutions of the linear
problems. In Section~\ref{sec3}, the conormal spaces related to the
equations under study are introduced and corresponding commutator
relations which are crucial for the following are established. Based
on the results of Section~\ref{sec2}, local solvability of
Eq.~\eqref{1-1} and Eq.~\eqref{1-2} is shown to hold in
Section~\ref{sec4}. In Sections~\ref{sec5} and \ref{sec6}, the proofs
of Theorems~\ref{thm1-1} and~\ref{thm1-2} are completed utilizing the
conormal spaces and commutator relations of Section~\ref{sec3}.


\section{Preliminaries}\label{sec2}

In this section, we recall some results of \cite{rwy12} and establish
the $L^{\infty}$ (or $L_{\text{loc}}^{\infty}$) property of low
regularity solutions of second-order linear degenerate hyperbolic
equations using rather delicate techniques.

\begin{lemma}[{\cite[Proposition 3.3]{rwy12}}]\label{lem2-1}
Let $\phi_1\in H^{s}(\mathbb R^n)$ and $\phi_2\in H^{s-\f{2}{m+2}}(\mathbb
R^n)$, where $s\in\mathbb R$. Then the homogeneous equation
\[
\left\{ \enspace
\begin{aligned}
&\p_t^2u-t^m\Delta_x u=0, \quad (t,x)\in (0,T)\times\mathbb R^n,\\
&u(0,x)=\phi_1(x),\quad \p_tu(0,x)=\phi_2(x),
\end{aligned}
\right.
\]
has a unique solution $u(t,x)\in C([0, T], H^s(\mathbb R^n ))\cap
C((0,T], H^{s+\f{m}{2(m+2)}}(\mathbb R^n))\cap C^1([0, T],
H^{s-\f{m+4}{2(m+2)}}(\mathbb R^n))$. Moreover, this solution
satisfies, for $0<t\leq T$,
\begin{multline*}
\|u(t,\cdot)\|_{H^{s}(\mathbb R^n)}
+t^{m/4}\|u(t,\cdot)\|_{H^{s+\f{m}{2(m+2)}}(\mathbb R^n)}
+\|\p_tu(t,\cdot)\|_{H^{s-\f{m+4}{2(m+2)}}(\mathbb R^n)} \\
\leq C\,\Bigl(\|\phi_1\|_{H^s(\mathbb
  R^n)}+\|\phi_2\|_{H^{s-\f{2}{m+2}}(\mathbb R^n)}\Bigr).
\end{multline*}
\end{lemma}

Let
\begin{equation}\label{2-1}
\left\{ \enspace
\begin{aligned} 
&V_1(t,
|\xi|)=  e^{-z/2}\Phi\left(\f{m}{2(m+2)},\f{m}{m+2};z\right),\\
&V_2(t,|\xi|)= 
te^{-z/2}\Phi\left(\f{m+4}{2(m+2)},\f{m+4}{m+2};z\right),
\end{aligned}
\right.
\end{equation}
where $z=\ds\f{4i}{m+2}t^{\f{m+2}{2}}|\xi|$;
$\ds\Phi(\f{m}{2(m+2)},\f{m}{m+2};z)$ and
$\ds\Phi(\f{m+4}{2(m+2)},\f{m+4}{m+2};z)$ are confluent hypergeometric
functions. (The definition of confluent hypergeometric function can be
found in \cite{emot53}.) These are analytic functions of $z$ that
satisfy, for large $|z|$,
\begin{equation}\label{2-2}
\left\{ \q
\begin{aligned}
&   \left|\Phi\left(\f{m}{2(m+2)},\f{m}{m+2};z\right)\right|
\le C\,|z|^{-\f{m}{2(m+2)}}\left(1+O\bigl(|z|^{-1}\bigr)\right),\\
&   \left|\Phi\left(\f{m+4}{2(m+2)},\f{m+4}{m+2};z\right)\right|
\le C\,|z|^{-\f{m+4}{2(m+2)}} \left(1+O\bigl(|z|^{-1}\bigr)\right).
\end{aligned}
\right.
\end{equation}

Then:

\begin{lemma}[{\cite[Lemma~3.2]{rwy12}}]\label{lem2-2}
Let $0\le s_1\ds\leq\f{m}{2(m+2)}$, $0\le s_2\ds\leq \f{m+4}{2(m+2)}$,
$T>0$, and $g(x)\in H^s(\mathbb R^n)$, where $s\in\mathbb R$. Then one
has, for $0<t\le T$,
\begin{equation}\label{2-3}
\left\{ \q
\begin{aligned} 
&\bigl\|\big(V_1(t,|\xi|)
\hat{g}(\xi)\big){}^{\!\vee}\bigr\|_{H^{s+s_1}}\leq C
t^{-\f{s_1(m+2)}{2}}\,\|g\|_{H^{s}},\\
&\bigl\|\big(V_2(t,|\xi|) \hat{g}(\xi)\big){}^{\!\vee}\bigr\|_{H^{s+s_2}}\leq C
t^{1-\f{s_2(m+2)}{2}}\,\|g\|_{H^{s}}
\end{aligned}
\right.
\end{equation}
and
\[
\begin{aligned} 
&\bigl\|\left(\p_t V_1(t,|\xi|)
\hat{g}(\xi)\right){}^{\!\vee}\bigr\|_{H^{s-\f{m+4}{2(m+2)}}}\leq
C\,\|g\|_{H^{s}},\\
&\bigl\|\left(\p_t V_2(t,|\xi|)
\hat{g}(\xi)\right){}^{\!\vee}\bigr\|_{H^{s-\f{m}{2(m+2)}}}\leq C\,\|g\|_{H^{s}},
\end{aligned}
\]
where ${}^\wedge$ and\/ ${}^\vee$ denote the Fourier transform and inverse
Fourier transform, respectively, with respect to $x\in\R^n$.
\end{lemma}

\begin{lemma}\label{lem2-3}  
Let $u(t,x)$ be a solution of the inhomogeneous problem
\[
\left\{ \enspace
\begin{aligned} 
& \p_t^2u-t^{m}\Delta_x u=f(t,x), 
\quad (t,x)\in (0,T)\times\mathbb R^n,\\
& u(0,x)=0,\quad \p_t u(0,x)=0.  
\end{aligned}
\right.
\]

\textup{(i)} \ If $f(t,x)\in L^p((0,T), H^{s}(\mathbb R^n))$, where
$s\in\mathbb R$ and $1 < p <\infty$, then, for any $t\in (0, T]$,
\begin{align}
&\|u(t,\cdot)\|_{H^{s+p_1}}\leq C t^{2-\f{
p_1(m+2)}{2}-1/p}\,\|f(t,x)\|_{L^{p}((0, T), H^s)}, \label{2-5}\\
& \|\p_tu(t,\cdot)\|_{H^{s-\f{m}{2(m+2)}+p_2}}\leq
Ct^{1-\f{p_2(m+2)}{2}-1/p}\,\|f(t,x)\|_{L^{p}((0, T), H^s)},
\label{2-6}
\end{align}
where $0\le p_1<p_1(m)=\ds\min\left\{\f{p(m+8)-4}{2p(m+2)},1\right\}$
and $0\le p_2<p_2(m)= \ds\min\left\{\frac{2(p-1)}{p(m+2)},
\f{m}{2(m+2)}\right\}$.

\smallskip

\textup{(ii) \ (See \cite[Lemma~3.4]{rwy12}.)} \ If $f(t,x)\in
C([0,T], H^s(\mathbb R^n))$, where $s\in\mathbb R$, then, for any
$t\in (0, T]$, 
\[
\begin{aligned} 
&\|u(t,\cdot)\|_{H^{s+p_3}}\leq Ct^{2-\f{
p_3(m+2)}{2}}\,\|f(t,x)\|_{L^{\infty}((0, T), H^s)},\\
&\|\p_tu(t,\cdot)\|_{H^{s-\f{m}{2(m+2)}+p_4}}\leq
C_{p_4}t^{1-\f{p_4(m+2)}{2}}\,\|f(t,x)\|_{L^{\infty}((0, T), H^s)},
\end{aligned}
\]
where $0\le p_3<p_3(m)=\ds\min\left\{\frac{m+8}{2(m+2)}, 1\right\}$ and
$p_4<p_4(m)=\ds\min\left\{\f{2}{m+2}, \f{m}{2(m+2)}\right\}$.
\end{lemma}

\begin{remark}\label{rem2-1} 
If $f(t,x)\in L^1((0,T), H^{s}(\mathbb R^n))$, then it follows from
  the proof of Lemma~\ref{lem2-3}\,(i), that, for $t\in(0,T]$,
\[
\left\{\q
\begin{aligned}
&\|u(t,\cdot)\|_{H^{s+p_1}}\leq C t^{1-\f{
p_1(m+2)}{2}}\|f(t,x)\|_{L^1((0, T], H^s)},\\
&\|\p_tu(t,\cdot)\|_{H^{s-\f{m}{2(m+2)}}}\leq
C\|f(t,x)\|_{L^1((0, T), H^s)},
\end{aligned}
\right.
\]
where $0\le p_1\le\ds\f{m}{2(m+2)}$.
\end{remark}

\begin{proof} 
We only need to prove (i). It is readily seen that the solution $u$ of
(2.4) can be expressed as
\begin{equation}\label{2-7}
\hat{u}(t,\xi)=\int_0^t
\Bigl(V_2(t,|\xi|)V_1(\tau,|\xi|) -V_1(t,|\xi|)V_2(\tau,|\xi|)
\Bigr)\hat{f}(\tau,\xi)d\tau,
\end{equation} 
where the definition of $V_1(t, |\xi|)$ and
$V_2(t,|\xi|)$ is given in \eqref{2-1}.
It follows from the Minkowski inequality and \eqref{2-7} that
\begin{multline}\label{2-8}
\|u(t,\cdot)\|_{H^{s+p_1}} \\
\begin{aligned} 
&\leq \int_0^t\biggl(\int_{\mathbb
R^n}\Big|(1+|\xi|^2)^{s/2+p_1/2}
(V_2(t,|\xi|)V_1(\tau,|\xi|)
-V_1(t,|\xi|)V_2(\tau,|\xi|))
\hat{f}(\tau,\xi)\Big|^2d\xi\biggr)^{1/2}d\tau\\
&\leq\int_0^t\|(1+|\xi|^2)^{s/2+p_1/2}
V_2(t,|\xi|)V_1(\tau,|\xi|)\hat{f}(\tau,\xi)\|_{L^2}\,d\tau\\
&\qquad\qquad +\int_0^t\|(1+|\xi|^2)^{s/2+p_1/2}
V_1(t,|\xi|)V_2(\tau,|\xi|)\hat{f}(\tau,\xi)\|_{L^2}\,d\tau\\
&\equiv I_1+I_2.
\end{aligned}
\end{multline}
Let $p_1=s_1+s_2$ with $0\le s_1<\ds\min\Bigl\{\f{m}{2(m+2)},
\f{2(p-1)}{p(m+2)}\Bigr\}$ and $0\le s_2\le\ds\f{m+4}{2(m+2)}$. Then
one has by Lemma~\ref{lem2-2} and the H\"older inequality that
\begin{equation}\label{2-9}
\left\{ \q
\begin{aligned} 
I_1&\le
Ct^{1-\f{s_2(m+2)}{2}}\int_0^t
\|(1+|\xi|^2)^{s_1/2}V_1(\tau,|\xi|)(1+|\xi|^2)^{s/2}\hat{f}(\tau,
\xi)\|_{L^2}\,d\tau\\
&\le Ct^{1-\f{s_2(m+2)}{2}}\int_0^t\tau^{-\f{s_1(m+2)}{2}}
\|f(\tau,\cdot)\|_{H^s}\,d\tau\\
& \le C t^{2-\f{p_1(m+2)}{2}-1/p}\|f(t,x)\|_{L^{p}((0, T),
H^s)}.
\end{aligned}
\right.
\end{equation}
If one sets $p_1=\t s_1+\t s_2$ with $0\le \t s_1\le\ds\f{m}{2(m+2)}$
and $0\le \t s_2<\ds\min\Bigl\{\f{m+4}{2(m+2)},
\frac{2(2p-1)}{p(m+2)}\Bigr\}$, then by Lemma~\ref{lem2-2} and
the H\"older inequality
\begin{equation}\label{2-10}
I_2\le Ct^{-\f{\t s_1(m+2)}{2}}\int_0^t\tau^{1-\f{\t s_2(m+2)}{2}}
\|f(\tau,\cdot)\|_{H^s}\,d\tau 
\le C t^{2-\f{p_1(m+2)}{2}-1/p}\|f(t,x)\|_{L^{p}((0, T),H^s)}.
\end{equation}
Substituting \eqref{2-9}--\eqref{2-10} into \eqref{2-8} yields
\eqref{2-5} for $0\le p_1<p_1(m)$.

Next we prove \eqref{2-6}. In view of
\[
\p_t\hat{u}(t,\xi)=\int_0^t \Bigl(\p_tV_2(t,|\xi|)V_1(\tau,|\xi|)
-\p_tV_1(t,|\xi|)V_2(\tau,|\xi|)\Bigr)\hat{f}(\tau,\xi)\,d\tau
\]
one has by the Minkowski inequality
\begin{multline*} 
\|\p_tu(t,\cdot)\|_{H^{s-\f{m}{2(m+2)}+p_2}}\\
\begin{aligned}
&\leq \int_0^t\biggl(\int_{\mathbb
R^n}\Bigl|(1+|\xi|^2)^{s/2-\f{m}{4(m+2)}+p_2/2}
\Bigl(\p_tV_2(t,|\xi|)V_1(\tau,|\xi|)
-\p_tV_1(t,|\xi|)V_2(\tau,|\xi|)\Bigr)\hat{f}(\tau,\xi)\Bigr|^2d\xi
\biggr)^{1/2}\,d\tau\\
&\leq\int_0^t\bigl\|(1+|\xi|^2)^{s/2-\f{m}{4(m+2)}+p_2/2}
\p_tV_2(t,|\xi|)V_1(\tau,|\xi|)\hat{f}(\tau,\xi)\bigr\|_{L^2}\,d\tau\\
&\qquad\qquad
+\int_0^t\bigl\|(1+|\xi|^2)^{s/2-\f{m+4}{4(m+2)}+\f12\min\bigl\{\f{m+4}{2(m+2)},
\f{4}{m+2}\bigr\}-\f{p_2(m)-p_2}{2}}
\p_tV_1(t,|\xi|)V_2(\tau,|\xi|)\hat{f}(\tau,\xi)\bigr\|_{L^2}\,d\tau\\
&\equiv I_3+I_4.
\end{aligned}
\end{multline*}
Applying for Lemma~\ref{lem2-2} and the H\"older inequality yields for
$0<t\leq T$
\begin{align*} 
I_3
& \leq C\int_0^t\bigl\|(1+|\xi|^2)^{p_2/2}V_1(\tau,|\xi|)
(1+|\xi|^2)^{s/2}\hat{f}(\tau,\xi)\bigr\|_{L^2}\,d\tau\\
&\le C\int_0^t\tau^{-\f{p_2(m+2)}{2}}\|f(\tau,\cdot)\|_{H^s}\,d\tau
\le C_{p_2}t^{1-\f{p_2(m+2)}{2}
-1/p}\|f\|_{L^{p}((0,T),H^s)}
\end{align*}
and
\begin{align*} 
I_4&\leq C\int_0^t\bigl\|(1+|\xi|^2)^{\f12\min\{\f{m+4}{2(m+2)},
    \f{4}{m+2}\}-\f{p_2(m)-p_2}{2}}V_2(\tau,|\xi|)
  (1+|\xi|^2)^{\f{s}{2}}\hat{f}(\tau,\xi)\bigr\|_{L^2}\,d\tau\\ &\le
  C\int_0^t\tau^{1-\f{p_2(m+2)}{2}}\|f(\tau,\cdot)\|_{H^s}\,d\tau \le
  C_{p_2}t^{2-\f{p_2(m+2)}{2} -1/p}\|f\|_{L^{p}((0,T),H^s)}.
\end{align*}
The estimates of $I_3$ and $I_4$ yield \eqref{2-6}.
\end{proof}

Now we start to establish the $L^{\infty}$ (or
$L_{\text{loc}}^{\infty}$) property of solutions $u$ of the linear
problem $\p_t^2u-t^m\Delta_x u=0$ with piecewise smooth or
single-point singular initial data.

\begin{lemma}\label{lem2-4} 
Assume that $u(t,x)\in C([0, T], H^{1/2-}(\mathbb R^n))$ is a solution
of the linear equation
\begin{equation}\label{2-11}
\left\{ \enspace
\begin{aligned} 
&\p_t^2u-t^m\Delta_x u=0,  \quad (t,x)\in (0,T)\times\mathbb R^n,\\
&u(0,x)=\varphi_1(x),\quad \p_tu(0,x)=\varphi_2(x).
\end{aligned}
\right.
\end{equation}

\textup{(i)} \ If $\vp_1(x)$ and $\varphi_2(x) $ satisfy assumption
\eqref{a1}, then there is a constant $C=C(n, T)>0$ such
that, for $(t,x) \in (0,T] \times \R^n$,
\[
  |u(t,x)| \le C\left(1+ |\ln t|\right). 
\]
In addition, if $n=1$, then $u(t,x)\in L^{\infty}((0, T)\times\mathbb
R)$.

\textup{(ii)} \ If $\vp_1(x)$ and $\varphi_2(x) $ satisfy assumption
\eqref{a2}, then there is a constant $C=C(n, T)>0$ such
that, for $(t,x) \in (0,T] \times \R^n$,
\[
|u(t,x)| \le C \left(1+|\ln t|^2\right).
\]
\end{lemma}

\begin{proof} 
(i) \ For $j=1,2$, one can write
\[
\varphi_j(x) = \varphi_{j2}(x)+\left(\varphi_{j1}(x)
-\varphi_{j2}(x)\right)E(x_1), 
\]
where $E(x_1)$ is the Heaviside function with $E(x_1)=\begin{cases}
1, &x_1>0,\\
0, &x_1<0.
\end{cases}$ \
Recall that the Fourier transform of $E(x_1)$ is
\[
\hat{E}(\xi_1)=\f1 2 \left(\delta(\xi_1) -\f i \pi \operatorname{p.v.} 
\f 1 \xi_1 \right),
\]
where $\delta$ is the Dirac delta function, $i=\sqrt{-1}$, and
$\operatorname{p.v.}$ denotes the principal value.
Then it follows from \eqref{2-11} that one has
\begin{multline*} 
\hat{u}(t,\xi) = \f 1 2\, V_1(t, |\xi|) \left(
\left(\hat{\varphi}_{11}(\xi)+\hat{\varphi}_{12}(\xi)\right) 
-i H\left(\hat{\varphi}_{11}(\zeta_1, \xi')-
\hat{\varphi}_{12}(\zeta_1, \xi')  \right)(\xi_1)\right)\\
+\, \f 1 2 V_2(t, |\xi|) \left(
\left(\hat{\varphi}_{21}(\xi)+\hat{\varphi}_{22}(\xi)\right) 
-i H\left(
\hat{\varphi}_{21}(\zeta_1, \xi')- \hat{\varphi}_{22}(\zeta_1, \xi')
\right)(\xi_1)\right),
\end{multline*}
where the functions $V_i$ ($i=1, 2$) have been defined in \eqref{2-1},
$\xi'=(\xi_2, \dots, \xi_n)$, and $H$ is the Hilbert
transformation. The latter means that, for $1 \le j, k \le 2$,
\[
H\left( \hat{\varphi}_{jk}(\zeta_1, \xi')\right)(\xi_1)
= \f 1 \pi \Bigl(\operatorname{p.v.} \f 1 \zeta_1
\ast \hat{\varphi}_{jk}(\zeta_1, \xi') \Bigr)(\xi_1) =\f 1 \pi
\lim_{\varepsilon\rightarrow 0^+} \int_{|\zeta_1|\ge\varepsilon}
\frac{ \hat{\varphi}_{jk}(\xi_1-\zeta_1, \xi')} {\zeta_1}
\, d \zeta_1.
\]
Therefore,
\begin{multline}\label{2-12} 
u(t,x) = \frac{1}{2}\int_{\mathbb R^n} e^{2 \pi i x \cdot \xi}
\sum_{j, k,\ell =1}^2
 V_\ell(t, |\xi|) \hat{\varphi}_{jk}(\xi)\, d\xi \\
- \, \frac{i}{2}\int_{\mathbb R^n} e^{2 \pi i x \cdot \xi}
\sum_{j,\ell =1}^2  V_\ell(t, |\xi|) H\bigl(
\hat{\varphi}_{j1}(\zeta_1, \xi')-\hat{\varphi}_{j2}(\zeta_1,
\xi')\bigr)(\xi_1)\, d\xi.
\end{multline}

To show that $u(t,x)\in L^{\infty}((0,T)\times\mathbb R^n)$, we will
show that each term in \eqref{2-12} is bounded. By
\eqref{2-2}, one has that, for $ 1\le j, k, \ell \le 2$,
\[
\Bigl|\int_{\R^n} e^{2 \pi i x \cdot \xi} V_\ell(t,
|\xi|)\hat{\varphi}_{jk}(\xi) \ d\xi \Bigr| \le \|V_\ell(t,
|\xi|)\|_{L^\infty((0,T) \times \R^n)}
\|\hat{\varphi}_{jk}\|_{L^1(\R^n)} \le C.
\]
Next we treat the terms $V_l(t, |\xi|)H\bigl(
\hat{\varphi}_{jk}(\zeta_1, \xi')\bigr)(\xi_1)$ ($ 1\le j, k, \ell \le
2$) in \eqref{2-12}. Due to the Schwartz inequality and the $L^2(\R)$
boundedness of the Hilbert transformation $H$, one has that, for
$(\xi_1, \xi')\in (0, M)\times\mathbb R^{n-1}$ for some large $M>0$,
\begin{multline*}\label{2-14}
\Big|\int_{ \xi' \in \R^{n-1}} \int_{|\xi_1| < M} e^{2
\pi i x \cdot \xi} \ H( \hat{\varphi}_{jk}(\zeta_1, \xi'))(\xi_1) 
V_\ell(t, |\xi|) \ d\xi_1 d\xi' \Big| \\ 
\begin{aligned} 
&\le C_M\|V_\ell(t,
|\xi|)\|_{L^\infty((0,T)\times \R^n)} \int_{\xi' \in \R^{n-1}}\Big(
\int_{|\xi_1| \le M} |H( \hat{\varphi}_{jk}(\zeta_1,
\xi'))(\xi_1)|^2\,d\xi_1 \Big)^{1/2} d\xi' \\
&\le C_M\,
 \|\hat{\varphi}_{jk}\|_{L^1(\R_{\xi'}^{n-1}, L^2(\R_{\xi_1}))}\le C_M. 
\end{aligned}
\end{multline*}
For $(\xi_1, \xi')\in (M, \infty)\times\mathbb R^{n-1}$, we
set
\[
I\equiv\int_{ \xi' \in \R^{n-1}} \int_{|\xi_1|> M} e^{2 \pi i x \cdot \xi} 
H( \hat{\varphi}_{jk}(\zeta_1,\xi'))(\xi_1)  V_\ell(t, |\xi|) \  d\xi_1 d\xi'.
\]
Denote by $\Phi_{jk}(x_1,\xi')$ the Fourier transform of $\varphi_{jk}$ with
respect to the variables $x_2, \dots, x_n$. (If $n=1$, then
$\Phi_{jk}=\varphi_{jk}$.) Notice that, for any fixed $\xi_1 \neq 0$,
\[
\xi_1 H\bigl(\hat{\varphi}_{jk}(\zeta_1,\xi')\bigr)(\xi_1)= 
-\frac{ 1}{ \pi^2}\,\Phi_{jk}(0,\xi') + H \bigl(\zeta_1 \ 
\hat{\varphi}_{jk}(\zeta_1 ,\xi')\bigr)(\xi_1)
\] 
which gives
\[
\begin{aligned}
I&\le C \int_{ \xi' \in \R^{n-1}} \int_{|\xi_1| > M}
\left|H \bigl(\zeta_1
\hat{\varphi}_{jk}(\zeta_1,\xi') \bigr)(\xi_1) \right|   
\frac{|V_\ell(t, |\xi|)|}{|\xi_1|}  \  d\xi_1 d\xi' 
 + C \int_{ \xi' \in \R^{n-1}}  |\Phi_{jk}(0,\xi')|\, d\xi' \\
 &\qquad\qquad  \times \left|
\Bigl(\int_{|\xi_1| > M,  |z_1| >  M} + \int_{|\xi_1| > M,
1/M < |z_1| < M} + \int_{|\xi_1| > M, |z_1| <1/M} \Bigr) 
e^{2 \pi i x_1 \xi_1}\, \frac{V_\ell(t,|\xi|)}{\xi_1}\, d\xi_1 \right|
\\&= I_1 +I_2\equiv I_{1}+I_{2,1} +I_{2,2}+I_{2,3},
\end{aligned}
\]
where $z_1= \frac{4i}{m+2}\,t^{\frac{m+2}{2}} \xi_1$. (If $n=1$, then $z_1=z$.)
The $L^2$ boundedness of the Hilbert transformation further yields
\begin{multline*}
I_{1} \le  C \left\|V_\ell(t, |\xi|)\right\|_{L^\infty((0,T)\times
\R^n)}\int_{\xi' \in \R^{n-1}} \left\|H \left(\zeta_1
\hat{\varphi}_{jk}(\zeta_1,\xi') \right)\right\|_{L^2(\R)} \left(\int_{
|\xi_1| \ge M} \frac{ d\xi_1}{|\xi_1|^2}
\right)^{1/2} d\xi' \\
\le  C_{M} \left\|\xi_1\hat{\varphi}_{jk}(\xi_1,
\xi')\right\|_{L^1(\R_{\xi'}^{n-1}, L^2(\R_{\xi_1}))} \le  C_{M}.
\end{multline*}
We now estimate the term $I_2$. For $l=1$, one has from the estimates
\eqref{2-1}--\eqref{2-3} that
\begin{multline}\label{2-17}
I_{2,1} \le C \int_{\xi' \in \R^{n-1}} |\Phi_{jk}(0,\xi')|\,d\xi'
\int_{|\xi_1| > M, |z_1| > M} |z|^{-\frac{m}{2(m+2)}}\,
\frac{d\xi_1}{\xi_1}\\ \le C \int_{\xi' \in
  \R^{n-1}}|\Phi_{jk}(0,\xi')|\, d\xi' \Big(\int_{ M}^\infty
s^{-1-\frac{m}{2(m+2)}}\, ds\Big) \le C_M,
\end{multline}
and similarly for $l=2$, where $ I_{2,1}$ is also dominated by a
constant $C_M$. Using the fact that $V_\ell(t, |\xi|) \in
L^\infty((0,T)\times \R^n)$ once again, one has that
\begin{equation}\label{2-18}
I_{2,2}\le \left\|V_\ell(t, |\xi|)\right\|_{L^\infty((0,T)\times
\R^n)} \int_{\xi' \in
\R^{n-1}}\left|\Phi_{jk}(0,\xi')\right|d\xi' \int_{1/M}^M \frac{dr}{r}
\le C_M.
\end{equation}
In addition, one has, for $t \in (0,T]$,
\begin{equation}\label{2-19}
I_{2,3} \le  C\, \|V_\ell(t, |\xi|)\|_{L^\infty((0,T) \times
\R^n)} \|\hat{\varphi}_{jk}\|_{L^1(\R^n)} \Big(\int_M^{\frac{m+2}{4}
t^{-\frac{m+2}{2}}\bigm/ M}\, \frac{ds}{s} \Big)
\le  C\left(1+ |\ln t|\right).
\end{equation}
Collecting the estimates in \eqref{2-17}---\eqref{2-19} yields $I_2
\le C \left(1+ |\ln t|\right)$ for $t \in (0,T]$. Thus,
\[
|u(t,x)|=\Bigl| \int_{\R^n} e^{2 \pi i x \cdot \xi}\hat{u}(t,\xi)\, d\xi \Bigr|
\le C \left(1+ |\ln t|\right), \quad (t,x) \in (0,T) \times \R^n.
\]

Finally, we discuss uniformly boundedness of the solution $u$ in
case $n=1$. Following the arguments above, one only needs to show
that the term
\[
I_{2,3}= C\, |\varphi_{jk}(0)|\, \biggl| \int_{|\xi_1|> M, |z_1|>
1/M} e^{2 \pi i x_1 \xi_1} \frac{V_\ell(t, |\xi_1|)}{\xi_1}\,
d\xi_1\biggr|
\] 
is uniformly bounded for $(t,x) \in (0,T) \times \R$.
For $\ell=1$,
\[
\begin{aligned} 
I_{2,3}&=  C\, |\varphi_{jk}(0)| \ \biggl | \int_{|\xi_1| > M,
|z_1| < 1/M} e^{2 \pi i x_1 \xi_1} e^{i
|z_1|/2} \Phi\Bigl(\frac{m}{2(m+2)}, \frac{m}{m+2}; z_1\Bigr)
\,\frac{d\xi_1}{\xi_1}
\biggr| \\
&= C\,  |\varphi_{jk}(0)|\ \biggl | \int_{|\xi_1| > M, |z_1|<
1/M}  \sin(2 \pi x_1 \xi_1) \Bigl(\cos\Bigl(\frac{|z_1|}{2}\Bigr)+ i
\sin\Bigl(\frac{|z_1|}{2}\Bigr) \Bigr)
\Phi\Bigl(\frac{m}{2(m+2)}, \frac{m}{m+2};z_1\Bigr)\, \frac{d\xi_1}{\xi_1}
\biggr| \\
&\le  C\, \int_{\xi_1 > M, |z_1| < 1/M} |\sin(2 \pi  x_1
\xi_1)| \ \sin\Bigl(\frac{|z_1|}{2}\Bigr)\,   \frac{d\xi_1}{\xi_1}
 \\
 & \qquad\qquad +  C\,   \biggl|
\int_{ |\xi_1| > M,  |z_1| <1/M} \sin(2 \pi
 x_1 \xi_1) \ \cos\Bigl(\frac{|z_1|}{2}\Bigr) 
\ \Phi(\frac{m}{2(m+2)}, \frac{m}{m+2}; z_1)\,
 \frac{d\xi_1}{\xi_1} \biggr|  \\
&\equiv I_{2,3}^{(1)} + I_{2,3}^{(2)}.
\end{aligned}
\]
Obviously, the nonnegative term $I_{2,3}^{(1)}$ is dominated by a positive
constant $C$.

Next we show the uniform boundedness of $I_{2,3}^{(2)}$. Denote by
\[
a_1=\Bigl(2 \pi x_1 + \frac{2}{m+2}\,t^{\frac{m+2}{2}}\Bigr) M, \quad 
b_1=\Bigl(2 \pi x_1 + \frac{2}{m+2}\,t^{\frac{m+2}{2}}\Bigr) \frac{4}{m+2}\,
t^{-\frac{m+2}{2}} M^{-1},
\]
and
\[
a_2=\Bigl(2 \pi x_1 - \frac{2}{m+2}\,t^{\frac{m+2}{2}}\Bigr) M, \quad 
b_2=\Bigl(2 \pi x_1 - \frac{2}{m+2}\,t^{\frac{m+2}{2}}\Bigr) \frac{4}{m+2}\,
t^{-\frac{m+2}{2}} M^{-1}.
\]
Then the term $I_{2,3}^{(2)}$ is dominated by
\[
C\,\sum_{i=1}^2\,\biggl|\int_{ a_i< s< b_i, |z_1| < 1/M  }
\Phi\Bigl(\frac{m}{2(m+2)}, \frac{m}{m+2}, z_1\Bigr)\,
 \frac{\sin s} {s}\, ds\biggr|.
\]
Since the confluent hypergeometric function
$\ds\Phi\Bigl(\frac{m}{2(m+2)}, \frac{m}{m+2};z_1\Bigr)$ is an
analytic function of $z_1$ \linebreak and $\ds\Phi\Bigl(\frac{m}{2(m+2)},
\frac{m}{m+2};0\Bigr)=1$, one has, for $M >1$ large and $|z_1|<
1/M$,
\[
\Phi\Bigl(\frac{m}{2(m+2)}, \frac{m}{m+2};z_1\Bigr)=1+ O(z_1),
\]
which  yields
\begin{multline*} 
I_{2,3}^{(2)} \le C\,\sum_{i=1}^2\,\biggl|\int_{a_i}^{b_i} \frac{\sin
s} {s}\, ds\biggr| +  C\,\sum_{i=1}^2\,\biggl|\int_{a_i < s
< b_i, |z_1| < 1/M}  \frac{|z_1|} {s}\, ds \biggr| \\
\le C\, \sum_{i=1}^2\,\biggl|\int_{a_i}^{b_i} \frac{\sin s}{s}\,ds
\biggr| + \frac{C\,t^{\frac{m+2}{2}}}{|2\pi x_1+ t^{\frac{m+2}{2}}|}\,
|b_1-a_1|
 + \frac{C\,t^{\frac{m+2}{2}}}{|2\pi x_1 - t^{\frac{m+2}{2}}|}\,
|b_2-a_2| \le C.
\end{multline*}
Thus, $I_{2,3} \le C$ for $\ell =1$. Similarly, one shows that
$I_{2,3} \le C$ for $\ell =2$. Hence, one has that $u(t,x) \in L^\infty((0,T)
\times \R)$ in case $n=1$.

\smallskip

(ii) \ Suppose that $\vp_1(x)$ and $\varphi_2(x) $ satisfy assumption
\eqref{a2}. From the proof of \cite[Lemma~2.1\,(a)]{rwy12}, one has
that $|\hat{\varphi}_i(\xi)|\le C\,\ds\f{(1+\ln|\xi|)}{|\xi|^{n}}$ for
$|\xi|> M>1$ and $\hat{\varphi}_i \in L^1(\{|\xi| < M\})$. Thus,
for $\ell =1,2$,
\[
\int_{|\xi| < M} |V_\ell(t, |\xi|)| \ |\hat{\varphi}_i(\xi)|\, d\xi
\le  \|V_\ell(t, |\xi|)\|_{L^\infty ((0,T) \times \mathbb R^n)}
\|\hat{\varphi}_i\|_{L^1(\{|\xi| \le M\})} \le C
\]
and
\begin{align*} 
\int_{|\xi| > M, |z| > M} |V_1(t, |\xi|)|  \
|\hat{\varphi}_i(\xi)|\, d\xi &\le C\,  \int_{|\xi| > M, |z| >
M}|z|^{-\frac{m}{2(m+2)}}\,\frac{1+ \ln|\xi|}{|\xi|^n}\,d\xi \\
&\le C\, t^{-m/4}\int_{ \frac{4}{m+2}\, t^{\frac{m+2}{2} } r
> M}r^{-\frac{m}{2(m+2)}}\,
\frac{1+ \ln r}{r}\, dr \\
&\le C\, \int_{s> M} s^{-\frac{m}{2(m+2)}-1} \left( 1+ \ln s+ | \ln
t| \right)ds
\le C\left(1+ |\ln t|\right).
\end{align*}
Similarly, 
\[
\int_{|\xi| > M, |z| > M} |V_2(t, |\xi|)|  \
|\hat{\varphi}_i(\xi)|\, d\xi  \le C\left(1+ |\ln t|\right).
\]
Further,
\begin{multline*}
 \int_{|\xi| > M, |z| < M} |V_\ell(t, |\xi|)| \
|\hat{\varphi}_i(\xi)|\, d\xi \le C\, \|V_\ell(t, |\xi|)\|_{L^\infty
((0,T] \times \mathbb R^n)} \int_{|\xi| > M, |z|
< M} \frac{1+ \ln|\xi|}{|\xi|^n}\, d\xi \\
\le C\, \|V_\ell(t, |\xi|)\|_{L^\infty ((0,T)\times \mathbb R^n)}
\int_M^{\frac{m+2}{4} t^{-\frac{m+2}{2}}M}  \frac{1+ \ln r}{r}\, dr
\le C \left(1+ |\ln t|\right)^2,
\end{multline*}
which completes the proof of Lemma~\ref{lem2-4}.
\end{proof}

\begin{lemma}\label{lem2-5} 
Let $u(t,x)$ be a solution to the problem
\begin{equation}\label{2-23}
\left\{ \enspace
\begin{aligned} 
&\p_t^2u-t^m\Delta_x u=f(t,x), \quad (t,x)\in (0,
T)\times\mathbb R^n,\\
&u(0,x)=\p_tu(0,x)=0,
\end{aligned}
\right.
\end{equation}

\textup{(i)} \ If $f(t,x)\in C([0,T], H^{s}(\mathbb R^n))$ and
$\p_{x'}^{\al}f(t,x)\in L^{\infty}((0,T), H^s(\mathbb R^n))$, where
$s>\ds\frac{m-2}{2(m+2)}$, $x'=(x_2,\dots, x_n)$, and $|\al|\le
\left[n/2\right]+1$, then $u(t,x)\in L^{\infty}((0, T)\times\mathbb
R^n)$.

\textup{(ii)} \ If $f(t,x)\in L^p((0,T), H^{r}(\mathbb R^n))$ and
$\p_{x'}^\alpha f(t,x)\in L^p((0,T), H^{r}(\mathbb R^n))$, where $1 <p
<\infty$, $r>\ds\frac{m}{2(m+2)}$, and $|\al|\le \left[n/2\right]+1$,
then $u(t,x)\in L^{\infty}((0, T)\times\mathbb R^n)$.
\end{lemma}

\begin{proof} 
Because of $[\p_{x'}^\alpha, \p_t^2 -t^m \Delta_x]=0,$ one has from
\eqref{2-23} that, for $|\al|\le \left[n/2\right]+1$,
\[
\left\{ \enspace
\begin{aligned} 
&\left(\p_t^2 -t^m\Delta_x\right)(\p_{x'}^\alpha u)=(\p_{x'}^\alpha f)(t,x),\\
&(\p_{x'}^\alpha u)(0,x)=\p_t(\p_{x'}^\alpha u)(0,x)=0.
\end{aligned}
\right.
\]

(i) \ Applying Lemma~\ref{2-3}\,(ii) with $p_3=\frac{2}{m+2}$, one sees that,
for any $t \in [0,T]$, there is a constant $C>0$
independent of $t$ such that
\begin{equation}\label{2-24}
\|(\p_{x'}^\alpha u)(t, \cdot)\|_{H^{s+\frac{2}{m+2}}(\R^n)} \le C\,
\|(\p_{x'}^\alpha f)(t,x)\|_{L^\infty((0,T), H^s(\R^n))}.
\end{equation}
One further has that, for $s>\ds\frac{m-2}{2(m+2)}$ and $\beta \in
\mathbb N_0^n$ with $|\beta|=\left[n/2\right]+1>\frac{n-1}{2}$,
\begin{equation}\label{2-25}
\begin{aligned} 
\|u(t,x_1, x')\|_{L^\infty((0,T)\times \R^n)} &\le  C
\|u(t,x_1, x')\|_{L^\infty ((0,T), H^{s + \frac{2}{m+2} }(\R)
\hat\otimes H^{[n/2]+1}(\R^{n-1}))} \\
&\le C \|(\p_{x'}^\beta u)(t,x_1, x')\|_{L^\infty ((0,T), H^{s +
\frac{2}{m+2} }(\R) \hat\otimes L^2(\R^{n-1}))}\\
&\le C \|(\p_{x'}^\beta u)(t,x_1, x')\|_{L^\infty ((0,T), H^{s +
\frac{2}{m+2} }(\R^{n}))},
\end{aligned}
\end{equation}
where $\hat\otimes$ denotes the (completed) Hilbert space tensor
product. Combining \eqref{2-24} and \eqref{2-25} yields $u(t,x)\in
L^{\infty}((0, T)\times\mathbb R^n)$.

\smallskip

(ii) \ Applying Lemma~\ref{lem2-3}\,(i) with $p_1=\frac{1}{m+2}$ and
using an argument analogous to that in (i) one sees that, for $p>1$,
$u(t,x)\in L^{\infty}((0, T)\times\mathbb R^n)$.
\end{proof}

Finally, based on Lemmas~\ref{lem2-4} and \ref{lem2-5}, we are able to
show that $u\in L^{\infty}((0, T)\times\mathbb R^n)$ for the solution
of \eqref{2-11} under assumption \eqref{a1} on the initial data.

\begin{lemma}\label{lem2-6} 
Let $u(t,x)$ be a solution of \textup{Eq.~\eqref{2-11}}.  If the
initial data $\vp_1(x)$ and $\varphi_2(x) $ satisfies assumption
\eqref{a1}, then, for any $T>0$, there is a constant $C=C(n, T)>0$
such that $u(t,x)\in L^{\infty}((0, T)\times\mathbb R^n)$ and
\[
  |u(t,x)|\leq C, \q (t,x)\in (0,T)\times\R^n.
\]
\end{lemma}

\begin{proof}
Write Eq.~\eqref{2-11} as
\begin{equation}\label{2-26}
\left\{\enspace
\begin{aligned} 
& \left(\p_t^2-t^m\p_1^2\right) u= t^m \sum_{i=2}^n \p_i^2 u(t,x),
\quad (t,x)\in (0,T)\times\mathbb R^n,\\
&u(0,x)= \varphi_1(x),\quad \p_t u(0,x)=\varphi_2(x).
\end{aligned}
\right.
\end{equation}
The solution $u(t,x)$ of problem \eqref{2-26} decomposes as
\[
 u(t,x)=w(t,x) +v(t,x), 
\]
where $w(t,x)$ satisfies
\[
\left\{\enspace
\begin{aligned} 
& \left(\p_t^2-t^m\p_1^2\right) w= 0, \quad (t,x)\in (0,T)\times\mathbb R^n,\\
&w(0,x)= \varphi_1(x),\quad \p_t w(0,x)=\varphi_2(x).
\end{aligned}
\right.
\]
and $v(t,x)$ satisfies
\begin{equation}\label{2-29}
\left\{\enspace
\begin{aligned} 
& \left(\p_t^2-t^m\p_1^2 \right) v= t^m \sum_{i=2}^n \p_i^2 u(t,x),
\quad (t,x)\in (0,T)\times\mathbb R^n,\\
&v(0,x)= 0,\quad \p_t v(0,x)=0.
\end{aligned}
\right.
\end{equation}
The Fourier transform $\hat{w}(t,\xi_1,x')$ of $w$ with respect to the
variable $x_1$ can be written as
\begin{multline*}
\hat{w}(t,\xi_1,x') = \f 1 2\, V_1(t, |\xi_1|) \Bigl(
\bigl(\hat{\varphi}_{11}(\xi_1,x')+\hat{\varphi}_{12}(\xi_1,x')\bigr) 
-i H\bigl(\hat{\varphi}_{11}(\zeta_1, x')-
\hat{\varphi}_{12}(\zeta_1, x')  \bigr)(\xi_1)\Bigr)\\
+ \f 1 2\, V_2(t, |\xi_1|) \Bigl(
\bigl(\hat{\varphi}_{21}(\xi_1,x')+\hat{\varphi}_{22}(\xi_1,x')\bigr) 
-i H\bigl(\hat{\varphi}_{21}(\zeta_1, x')- \hat{\varphi}_{22}(\zeta_1, x')
\bigr)(\xi_1)\Bigr).
\end{multline*}
Following the arguments in the proof of Lemma~\ref{lem2-4}\,(i) in
case $n=1$, one deduces that
\begin{equation}\label{2-30}
w(t,x) \in L^\infty((0,T) \times \R^n).
\end{equation}
Moreover, for any $\alpha\in{\mathbb N}_0^{n-1}$, from
$[\p_{x'}^\alpha, \p_t^2 -t^m \Delta_x]=0$ and \eqref{2-11} one
obtains that $\p_{x'}^\alpha u $ satisfies
\begin{equation}\label{2-31}
\left\{\enspace
\begin{aligned} 
& \left(\p_t^2-t^m\Delta_x \right) ( \p_{x'}^\alpha u)= 0, \quad (t,x)\in (0,
T)\times\mathbb R^n,\\
&( \p_{x'}^\alpha u)(0,x)= (\p_{x'}^\alpha \varphi_1)(x),\quad \p_t
(\p_{x'}^\alpha u)(0,x)=(\p_{x'}^\alpha \varphi_2)(x).
\end{aligned}
\right.
\end{equation}
Note that $\p_{x'}^\alpha \varphi_1$ and $\p_{x'}^\alpha \varphi_2$
also satisfy assumption \eqref{a1}. Then $\p_{x'}^\alpha \varphi_1,
\p_{x'}^\alpha \varphi_2 \in H^{1/2-}(\R^n)$. Hence, it follows from
Lemma~\ref{2-1} and \eqref{2-31} that $\p_{x'}^\alpha u \in C([0, T],
H^{1/2-}(\mathbb R^n ))\cap C\bigl((0,T], H^{\f{m+1}{m+2}-}(\mathbb
  R^n)\bigr)\cap C^1\bigl([0, T],$ \linebreak
  $H^{-\f{1}{m+2}-}(\mathbb R^n)\bigr)$ satisfies, for any $\dl>0$ and
  all $0<t\leq T$,
\begin{multline*}
\|(\p_{x'}^\alpha u)(t,\cdot)\|_{H^{1/2-\dl}(\mathbb R^n)}+
t^{m/4}\|(\p_{x'}^\alpha
u)(t,\cdot)\|_{H^{\f{m+1}{m+2}-\dl}(\mathbb R^n)} \\ 
+\|\p_t(\p_{x'}^\alpha
u)(t,\cdot)\|_{H^{-\f{1}{m+2}-\dl}(\mathbb R^n)}\le C(n,T,
\dl).
\end{multline*}
It then follows that
\[
t^m (\p_{x'}^\alpha u)(t,x) \in 
L^{\infty}\bigl((0, T), H^{\f{m+1}{m+2}-}(\R^n)\bigr).
\]
Together with Lemma~\ref{lem2-3}\,(i) and Eq.~\eqref{2-29}, this
yields $v(t,x) \in L^\infty\bigl((0,T),
H^{\frac{m+3}{m+2}-}(\R^n)\bigr)$.

Using $\left[\p_{x'}^\alpha, \p_t^2 -t^m \Delta_x\right]=0$ once
again, one has that
\[
\left\{ \enspace
\begin{aligned}
& \left(\p_t^2-t^m\p_1^2 \right) (\p_{x'}^\alpha v)=t^m\, \sum_{i=2}^n
(\p_i^2 \p_{x'}^\alpha u)(t,x), \quad (t,x)\in (0,
T)\times\mathbb R^n,\\
&(\p_{x'}^\alpha v)(0,x)= 0,\quad \p_t (\p_{x'}^\alpha v)(0,x)=0.
\end{aligned}
\right.
\]
One then obtains as in Lemma~\ref{lem2-5}\,(i) that
\[
\|v(t,x)\|_{L^\infty((0,T)\times \R^n)} \le C.
\]
Hence, combining \eqref{2-30} 
with the latter yields $u(t,x) \in L^\infty((0,T) \times \R^n)$.
\end{proof}


\section{Conormal spaces and commutator relations}\label{sec3}

In order to study the singularity structure of solutions of
\eqref{1-1} and \eqref{1-2}, we are required to introduce spaces of
conormal functions which relate to these problems.

For $k\in\mathbb N,$ set
\[
\Gamma_k=\Bigl\{(t,x)\colon\, t\ge 0, \
|x|^2=\f{4t^{k+2}}{(k+2)^2}\Bigr\}, \q
\Gamma_k^{\pm}=\Bigl\{(t,x)\colon\, t\ge
0, \ x_1=\pm \f{2t^{(k+2)/2}}{k+2}\Bigr\} 
\]
and
\[ 
l_0=\left\{(t,x)\colon\,  t\ge 0, \ |x|=0\right\},\q
\Sigma_0=\left\{(t,x)\colon\, t=0\right\}.
\]
A basis of the $C^{\infty}$ vector fields tangent to $\Gamma_k$ is
given by (see \cite[Lemma 4.2]{rwy12})
\[
\begin{aligned}
 V_0^{(k)}&=2t\p_t+(k+2)(x_1\p_1+\cdots+x_n\p_n),\\
V_{\ell}^{(k)}&=2t^{k+1}\p_\ell+\left(k+2\right)x_\ell\p_t,&&1 \le \ell \le n,\\
L_{ij}&=x_i\p_j-x_j\p_i,&& 1\le i<j\le n.
\end{aligned}
\]
Moreover, a basis of the $C^{\infty}$ vector fields tangent to $l_0$
is given by
\[
t \p_t, \q  V_0^{(m)}, \q L_{ij},\quad 1\le i<j\le n.
\]
To prepare for the commutator argument handling the
degenerate equations whose characteristic cones and characteristic
surfaces exhibit cusp singularities, we shall use the
following slightly altered vector fields tangent to $\G_k$:
\[
\begin{aligned} 
V_0^{(k)}&=2t\p_t+(k+2)(x_1\p_1+\cdots+x_n\p_n),\\
\bar V_{\ell}^{(k)}&=2t^{k/2+1}\p_\ell+(k+2)\f{x_\ell}{t^{k/2}}\p_t,&&
1\le \ell \le n,\\
L_{ij}&=x_i\p_j-x_j\p_i,&&1\le i<j\le n.
\end{aligned}
\]
Set 
\[
  P_1=\p_t\left(\p_t^2-t^{m}\Delta_x\right), \q 
  Q_k=\p_t^2-t^{k}\Delta_x.
\]
Further, let $[A,B]=AB-BA$ denote the commutator of $A$ and $B$.

By direct verification, one then has:
\begin{lemma}\label{lem3-1} 
For $1\le i\le n$ and $1\le i<j\le n$,
\[
\begin{aligned}
& \bigl[Q_k, V_0^{(k)}\bigr]=4 Q_k, \q  \bigl[Q_k, \bar{V}_l^{(k)}\bigr]=-k(k+2)
 \f{x_l}{t^{k/2+1}} Q_k + \f{k(k+2)}{4
t^2}\bar{V}_l^{(k)},\quad 1\leq l\leq n,\\
&\bigl[Q_k, L_{ij}\bigr]=0, \quad 1 \le i < j \le n,\\
&\bigl[V_0^{(k)}, \bar{V}_i^{(k)}\bigr]=0, \q \bigl[V_0^{(k)},L_{ij}\bigr]=0,  \q
\bigl[\bar{V}_i^{(k)},
L_{ij}\bigr]=\bar{V}_j^{(k)}\bar{V}_i^{(k)},\\
&\bigl[\bar{V}_i^{(k)},
\bar{V}_j^{(k)}\bigr]=2(k+1)(k+2)L_{ij}+\f{k(k+2)}{2}
\left(\f{x_j}{t^{k/2+1}}\bar{V}_i^{(k)}-
\f{x_i}{t^{k/2+1}}\bar{V}_j^{(k)} \right),  \\
&\bigl[\bar{V}_\ell^{(k)},L_{ij}\bigr]=0,\q i\neq \ell\neq j,\\
&\bigl[L_{ij},L_{kl}\bigr]=0, \q 1\le k<l\le n, \ k\neq i, \ l\neq
j, \q \bigl[L_{ij},L_{ik}\bigr]=L_{kj}
\q k\neq j,\\
&\bigl[P_1, V_0^{(m)}\bigr]=6P_1, \q \bigl[P_1, L_{ij}\bigr]=0,  \\
&\bigl[P_1, \bar{V}_i^{(m)}\bigr]= -\,\frac{3m(m+2)}{2}\, t^{-
m/2-1} x_i P_1 +(m+2)t^{m/2}Q_m\p_i+\frac{3 m
(m+2)^2}{4}\, t^{-m/2-2} x_i \p_t^2 \\
& \qquad \qquad \qquad \q
-\frac{m(m+2)^2}{2} \, t^{m/2-2} x_i \Delta_x 
-\frac{m (m+2)}{2}\, t^{m/2-1} \p_t\p_i +
\frac{m(m^2-4)}{4}\, t^{m/2-2} \p_i \\
& \qquad \qquad \qquad \q
-\frac{m(m+2)^2 (m+4)}{8}\,t^{-m/2-3} x_i \p_t.
\end{aligned}
\]
In addition,
\[
\bigl[P_1, t \p_t\bigr]= 3 P_1 + (m+2)\p_t \Delta_x +m(m+2)t^{m-1}
\Delta_x, \q
\bigl[t \p_t, V_0^{(m)}\bigr]= 0, \q \bigl[t \p_t, L_{ij}\bigr]=0.
\]
\end{lemma}

\begin{remark}\label{rem3-1} 
The commutator $\bigl[P_1,\bar{V}_i^{(m)}\bigr]$ for $1\le i\le n$ in
Lemma~\ref{lem3-1} contains the two singular terms $\ds\frac{3 m (m+2)^2
}{4}\,t^{-m/2-2} x_i \p_t^2$ and $\ds-\,\frac{m(m+2)^2 (m+4)}{8}\,
t^{-m/2-3} x_i \p_t$, and further the term $\ds
\frac{m(m+2)^2}{2}$\linebreak $t^{m/2-2} x_i \Delta_x$ which violates
the Levi condition to be imposed on degenerate hyperbolic
equations. Both will cause a loss of regularity for $\bar{V}_i^{(m)}u$
and further for $\bar{V}_{i_1}^{(m)}\dots \bar{V}_{i_k}^{(m)}u$, where
$k, i_1, \dots , i_k\in\mathbb N$. To avoid such a loss, we shall look
for relations that provide extra regularity of
$\bar{V}_{i_1}^{(m)}\dots \bar{V}_{i_k}^{(m)}u$. Such relations are
realized by the operator itself and by some parts of the vector
fields. (See formulas \eqref{3-1}--\eqref{3-7} in
Proposition~\ref{prop3-3} below.)
\end{remark}

\begin{lemma}\label{lem3-2} 
The vector fields 
\[
V^{(k)}=2t\p_t+(k+2)x_1\p_1, \q
\bar{V}^{(k)}_1=2t^{k/2+1}\p_1+(k+2)\f{x_1}{t^{k/2}}\p_t,
\quad R_\ell=\p_\ell, \q  2 \le \ell \le n,
\]
are tangent to $\Gamma_k^{\pm}$. Moreover, one has
\[
\begin{aligned} 
&\bigl[V^{(k)},\bar{V}_{1}^{(k)}\bigr]=0,\quad 
\bigl[V^{(k)},R_\ell\bigr]=-(k+2)R_\ell, \quad
2 \le \ell \le n,\quad \bigl[\bar{V}^{(k)}_1, R_\ell\bigr]=0, \quad
2 \le \ell \le n,\\
&\bigl[P_1, V^{(m)}\bigr]=6P_1+2(m+2)t^m \sum_{i=2}^n\p_t R_i^2 +2m(m+2)
t^{m-1}  \sum_{i=2}^n R_i^2;
\quad  \bigl[P_1, R_\ell\bigr]=0, \quad 2 \le \ell \le n, \\
&\bigl[Q_k,  V^{(k)}\bigr]=4Q_k, \quad \bigl[Q_k, R_\ell\bigr]=0,
\quad 2 \le \ell \le n.
\end{aligned}
\]
\end{lemma}

Following \cite{bea88,bea92}, we now introduce the notion of an
admissible function.

\begin{definition}[Admissible functions]\label{def3-1}
A function $h(x)\in L^{\infty}(\O)\cap C^{\infty}(\O)$ is said to be
admissible with respect to the vector fields $N_1,\dots, N_k$ if
$N_1^{j_1}\dots N_k^{j_k}h\in L^{\infty}(\O)\cap C^{\infty}(\O)$ for
all $(j_1,\dots, j_k) \in \mathbb N_0^k$.
\end{definition}

The module over the algebra of admissible functions with basis
$N_1,\dots, N_k$ constitutes a Lie algebra of vector fields on $\O$
provided that 
\[
\begin{aligned}
  & \text{each commutator $[N_i,N_j]$ ($1\leq i<j\leq k$) is a linear 
  combination} \\
  & \text{of $N_1,\dots, N_k$ with admissible coefficients.}
\end{aligned} \tag{$\ast$}
\]

Next we define admissible tangent vector fields with respect to the
surface $\G_m$ and the ray $l_0$.

\begin{definition}[Admissible tangent vector fields for
$\G_m \cup l_0$]\label{def3-2} 

${ }$

\textup{(1)} 
(Exterior to $\G_m$) For $\O_1=\{(t,x)\colon\, 0<t<C|x|\le \ve\}$,
$\mathcal{S}_1$ denotes the Lie algebra of vector fields with admissible
coefficients in $\O_1$ generated by $|x|\p_t$, $t^{m/2}|x|\p_\ell$
($1 \le \ell \le n$), $\ L_{ij}$ ($1\le i<j \le n$).

\textup{(2)} (Near $\G_m$) For $\O_2=\{(t,x)\colon\,
0<|x|<Ct\le\ve\}\cap \left\{(t,x)\colon \,
\bigl||x|-\f{2}{m+2}\,t^{\f{m+2}{2}} \bigr|<Ct^{\f{m+2}{2}}\right\}$,
$\mathcal{S}_2$ denotes the Lie algebra of vector fields with admissible
coefficients in $\O_2$ generated by $V_0^{(m)}$, $\bar{V}_l^{(m)}$
($1 \le \ell \le n$), $L_{ij}$ ($1\le i<j \le n$).

\textup{(3)} (Near $l_0$) For $\O_3=\{(t,x)\colon\,
|x|<Ct\le\ve\}\cap\left\{(t,x)\colon\, t^{\f{m+2}{2}}<C
\bigl||x|-\f{2}{m+2}\,t^{\f{m+2}{2}}\bigr|\right\}$, $\mathcal{S}_3$
denotes the Lie algebra of vector fields with admissible coefficients
in $\O_3$ generated by $t\p_t$, $V_0^{(m)}$, $L_{ij}$ ($1 \le i<j \le
n$).

\textup{(4)} (Between $\G_m$ and $l_0$) For $\O_4=\{(t,x)\colon\,
0<t<C|x|\le\ve\}\cap\left\{(t,x)\colon\, t^{\f{m+2}{2}}<C
\bigl||x|-\f{2}{m+2}\,t^{\f{m+2}{2}}\bigr|\right\}$, $\mathcal{S}_4$
denotes the Lie algebra of vector fields with admissible coefficients
in $\O_4$ generated by $t\p_t$, $t^{m+1}\p_i$ ($1\leq i\leq n$),
$L_{ij}$ ($1 \le i<j \le n\}$.
\end{definition}

In all four cases (1)--(4), one checks that condition ($\ast$) is
fulfilled.

The conormal space $I^{\infty}H^{s}(\G_m \cup l_0)$ for $0\le s<n/2$
is defined as follows:

\begin{definition}[Conormal space $I^{\infty}H^{s}(\G_m \cup
l_0)$]\label{def3-3} Define $u(t,x)\in I^{\infty}H^{s}(\G_m \cup l_0)$
  in $\{(t,x)\colon 0< t< T,$ $x\in\mathbb R^n\}$ if, away from the
  origin $\{t=|x|=0\}$ and near $\G_m$, $Z_1\dots Z_ku\in
  L^{\infty}((0, T), H^{s}(\mathbb R^{n}))$ for all smooth vector
  fields $Z_1, \cdots, Z_j\in \{V_0^{(m)}, \bar{V}_l^{(m)}, 1 \le \ell
  \le n, L_{ij}, 1\le i<j \le n\}$, away from $\{t=|x|=0\}$ and near
  $l_0$, $Z_1\dots Z_ku\in L^{\infty}((0, T), H^{s}(\mathbb R^{n}))$
  for all smooth vector fields $Z_1, \dots, Z_j\in \{ t \p_t,
  \ V_0^{(m)}, L_{ij}, 1\le i<j\le n \}$. Near $\{t=|x|=0\}$, the
  following properties hold:

\textup{(1)} \ If $h_1(t,x)\in C^{\infty}(\mathbb
R^{n+1}\setminus\{0\})$ is homogeneous of degree zero and supported in
$\O_1$, then \linebreak $Z_1\dots Z_k(h_1(t,x)$ $u(t,x))\in
L^{\infty}((0, T), H^{s}(\mathbb R^{n}))$ for all $Z_1,\dots, Z_k\in
\mathcal{S}_1$.

\textup{(2)} \ If $h_2(t,x)\in C^{\infty}(\mathbb
R^{n+1}\setminus\{0\})$ is homogeneous of degree zero and supported in
$\{(t,x)\colon 0<|x|<Ct\le\ve\}$ and $\chi(\theta)\in C^{\infty}$ has
compact support near $\{\theta=1\}$, then $Z_1\dots
Z_k(h_2(t,x)\chi\left(\f{(m+2)|x|}{2t^{\f{m+2}{2}}})u\right)\in L^{\infty}((0,
  T), H^{s}(\mathbb R^{n}))$ for all $Z_1,\dots, Z_k\in \mathcal{S}_2$.

\textup{(3)} \ If $h_3(t,x)\in C^{\infty}(\mathbb
R^{n+1}\setminus\{0\})$ is homogeneous of degree zero and supported in
$\{(t,x)\colon |x|<Ct\le\ve\}$ and $\chi_1(\theta)\in C^{\infty}$ has
compact support away $\{\theta=1\}$, then $Z_1\dots
Z_k\big(h_3(t,x)\chi_1\left(\f{(m+2)|x|}{2t^{\f{m+2}{2}}}\right)u
\big)\in L^{\infty}((0, T), H^{s}(\mathbb R^{n}))$ for all
$Z_1,\dots, Z_j\in \mathcal{S}_3$.

\textup{(4)} \ If $h_4(t,x)\in C^{\infty}(\mathbb
R^{n+1}\setminus\{0\})$ is homogeneous of degree zero and supported in
$\{(t,x)\colon 0<t<C|x|\le\ve\}$ and $\chi_2(\theta)\in C^{\infty}$ has
compact support away $\{\theta=1\}$, then $Z_1\dots
Z_k\big(h_4(t,x)\chi_2\left(\f{(m+2)|x|}{2t^{\f{m+2}{2}}}\right)u \big)\in
L^{\infty}((0, T), H^{s}(\mathbb R^{n}))$ for all $Z_1,\dots, Z_j\in
\mathcal{S}_4$.
\end{definition}

Note that $h_1(t,x)$,
$h_2(t,x)\chi\left(\f{(m+2)|x|}{2t^{\f{m+2}{2}}}\right)$,
$h_3(t,x)\chi_1\left(\f{(m+2)|x|}{2t^{\f{m+2}{2}}}\right)$, and
$h_4(t,x)\chi_2\left(\f{(m+2)|x|}{2t^{\f{m+2}{2}}}\right)$ are
admissible functions in the regions $\O_1$, $\O_2$, $\O_3$, and
$\O_4$, respectively. Moreover, they belong to
$L^{\infty}((0, T), H^{n/2-}(\mathbb R^n))$.

Because some of the vector fields in Definition~\ref{def3-3} (e.g.,
$\bar{V}_i^{(m)}$, $1\le i\le n$) do not admit good commutator
relations with $P_1=\p_t \left(\p_t^2-t^m\Delta_x \right)$ (in the
sense that not all coefficients appearing in the commutators are
admissible), one has to look for additional relations among these
vector fields. From now on, we will often write $\bar {V_i}$ ($1\le
i\le n$) and $V_0$ instead of $\bar{V}_i^{(m)}$ and $V_0^{(m)}$,
respectively.

By a careful computation as in \cite{rwy12}, one obtains the following
result:

\begin{proposition}\label{prop3-3} 
Let $\O_i$ \textup{(}$1 \le i \le 4$\textup{)} be as given in\/
\textup{Definition~\ref{def3-2}}. Then\/\textup{:}

\textup{(1)} \ With $N_1^0=|x|\,\p_{t}$, $N_1^i=t^{m/2}|x|\,\p_i$
\textup{(}$1\le i\le n$\textup{)} in $\O_1$, one has, for $\nu=0,\dots,
n$,
\begin{multline}\label{3-1} 
(N_1^{\nu})^2=h_0^{1\nu} Q_m+h_1^{1\nu}V_0^2+\sum_{\substack{1\le i<k\le n\\
1\le m<l\le n}}a_{ikml}^{1\nu}L_{ik}L_{ml} + \sum_{1\le
i<k\le
n}b_{ik}^{1\nu}V_0L_{ik}+\sum_{0\le i\le n}r_i^{1\nu}N_1^iV_0\\
+\,\sum_{\substack{0\le i\le n\\1\le m<l\le
n}}b_{iml}^{1\nu}N_1^iL_{ml}+ \sum_{0\le i\le
n}c_{0i}^{1\nu}N_1^i+c_0^{1\nu}V_0,
\end{multline}
where all the coefficients $h_0^{1\nu}$, $h_1^{1\nu}$, $a_{ikml}^{1\nu}$,
$b_{ik}^{1\nu}$, $r_i^{1\nu}$, $b_{iml}^{1\nu}$, $c_{0i}^{1\nu}$, and
$c_0^{1\nu}$ are admissible in $\O_1$.

\textup{(2)} \ With
$N_2^i=\bigl(|x|-\f{2}{m+2}\,t^{\f{m+2}{2}}\bigr)\p_i$ \textup{(}$1\le
i\le n$\textup{)} in $\O_2$, one has, for $\nu=1,\dots, n$,
\begin{equation}\label{3-2}
\bar{V}_{\nu}
=\f{1}{2t^{\f{m+2}{2}}}\biggl((m+2)x_{\nu}V_0-(m+2)^2 \sum_{k\not=\nu}x_kL_{\nu k}
-(m+2)((m+2)|x|+2t^{\f{m+2}{2}})N_2^{\nu}\biggr)
\end{equation}
and
\begin{multline}\label{3-3} 
(N_2^{\nu})^2 =h_0^{2\nu} Q_m+h_1^{2\nu}V_0^2+ \sum_{\substack{1\le i<k\le
n\\1\le m<l\le n}}a_{ikml}^{2\nu}L_{ik}L_{ml} + \sum_{1\le
i<k\le
n}b_{ik}^{2\nu}V_0L_{ik}+ \sum_{1\le i\le n}r_i^{2\nu}N_2^iV_0\\
+\, \sum_{\substack{1\le i\le n\\ 1\le m<l\le
n}}b_{iml}^{2\nu}N_2^iL_{ml}+\sum_{1\le i\le
n}c_{0i}^{2\nu}N_2^i+c_0^{2\nu}V_0,
\end{multline}
where all the coefficients $h_0^{2\nu}$, $h_1^{2\nu}$,
$a_{ikml}^{2\nu}$, $b_{ik}^{2\nu}$, $r_i^{2\nu}$, $b_{iml}^{2\nu}$,
$c_{0i}^{2\nu}$, and $c_0^{2\nu}$ are admissible in $\O_2$.

\textup{(3)} \ With $N_3^0=t\p_t$ in $\O_3$, one has
\begin{equation}\label{3-4}
\bar{V}_i =  \frac{2 t^{\frac{m+2}{2}} x_i}{(m+2) |x|^2}\,V_0
+ \frac{x_i \Big((m+2)^2 |x|^2 -4 t^{m+2} \Big)}{(m+2)|x|^2
t^{\frac{m+2}{2}}}\, N_3^0 -\frac{2 t^{\frac{m+2}{2}}}{|x|^2} \sum_{k
\ne i}x_k L_{ik},\quad 1\le i\le n,
\end{equation} 
and
\begin{multline}\label{3-5} 
(N_3^0)^2 =h_0^{30} Q_m+h_1^{30}V_0^2+\sum_{\substack{1\le
i<k\le n\\1\le m<l\le n}}a_{ikml}^{30}L_{ik}L_{ml} +\sum_{1\le i<k\le
n}b_{ik}^{30}V_0L_{ik}+r_0^{30}N_3^0V_0\\
+\, \sum_{1\le m<l\le
n}b_{0ml}^{30}N_3^0L_{ml}+c_{00}^{30}N_3^0+c_0^{30}V_0,
\end{multline}
where all the coefficients $h_0^{30}$, $h_1^{30}$, $a_{ikml}^{30}$,
$b_{ik}^{30}$, $r_0^{30}$, $b_{0ml}^{30}$, $c_{00}^{30}$, and $c_0^{30}$ are
admissible in $\O_3$.

\textup{(4)} \ With $N_4^i=t^{\f{m+2}{2}}\p_i$ \textup{(}$1\le i\le
n$\textup{)} in $\O_4$, one has, for $\nu=1,\dots, n$,
\begin{equation}\label{3-6}
\bar{V}_{\nu}=\frac{(m+2)x_{\nu}}{2t^{\frac{m+2}{2}}}\, V_0 +
\frac{4t^{m+2}-(m+2)^2 |x|^2}{t^{\frac{m+2}{2}}} N_4^{\nu} -
\frac{(m+2)^2}{2 t^{\frac{m+2}{2}}} \sum_{ k \ne {\nu}} x_k L_{\nu k}
\end{equation}
and
\begin{multline}\label{3-7} 
(N_4^{\nu})^2 =h_0^{4\nu} Q_m+h_1^{4\nu}V_0^2+\sum_{\substack{1\le i<k\le
n \\ 1\le m<l\le n}}a_{ikml}^{4\nu}L_{ik}L_{ml} + \sum_{1\le
i<k\le
n}b_{ik}^{4\nu}V_0L_{ik}+ \sum_{1\le i\le n}r_i^{4\nu}N_4^iV_0\\
+\, \sum_{\substack{1\le i\le n \\ 1\le m<l\le
n}}b_{iml}^{4\nu}N_4^iL_{ml}+ \sum_{1\le i\le
n}c_{0i}^{4\nu}N_4^i+c_0^{4\nu}V_0,
\end{multline}
where all the coefficients $h_0^{4\nu}$, $h_1^{4\nu}$,
$a_{ikml}^{4\nu}$, $b_{ik}^{4\nu}$, $r_i^{4\nu}$, $b_{iml}^{4\nu}$,
$c_{0i}^{4\nu}$, and $c_0^{4\nu}$ are admissible in $\O_4$.
\end{proposition}

\begin{remark}\label{rem3-2} 
Admissibility of the coefficients in each of the regions $\O_i$ ($1
\le i \le 4$) refers to the vector fields in $\mathcal{S}_i$.
\end{remark}

\begin{proof} 
(1) \ It follows from a direct computation that 
\begin{multline}\label{3-8} 
(N_1^0)^2=\f{1}{4t^{m+2}-(m+2)^2|x|^2}
\biggl(-4|x|^2t^{m+2} Q_m -|x|^2t^m \sum_{j=1}^n \bar{V}_j^2+4|x|t^{m+1}N_1^0V_0
\\
 +\, (m+2)|x|^2t^mV_0 + \Bigl(\bigl(2(m+2)(n-1)-8\bigr)\,
t^{m+1}|x|-\f{m(m+2)^2 |x|^3}{2t} \Bigr)N_1^0\biggr),
\end{multline}
\begin{multline}\label{3-9}
(N_1^i)^2= a_1 Q_m+b_1 V_0^2 +c_1 x_i N_1^iV_0 +d_1
\sum_{j=1}^n\bar{V}_j^2 +e_1 \sum_{k=1}^n x_k V_0
\bar{V}_k \\
+\, f_1 \sum_{j=1}^n x_j N_1^i \bar{V}_j +g_1 \sum_{ k\ne i} x_k \bar{V}_i L_{ik} 
+ h_{11} V_0   + h_{12} N_1^i + \sum_{j \ne i} h_1^j \bar{V}_j,
\end{multline}
where
\begin{align*} 
&a_1=\frac{4 |x|^2 t^{m+3}}{(m+2)^2|x|^2-4t^{m+2}}, \q b_1=-\,
\frac{|x|^2 t^{m+1}\bigl(4t^{m+2}+(m+2)^2
|x|^2\bigr)}{\bigl(4t^{m+2}-(m+2)^2 |x|^2\bigr)^2},\q
 c_1=\frac{2n(m+2)t^{m/2}|x|}{(m+2)^2|x|^2-4t^{m+2}
},\\
&d_1= \frac{ |x|^2 t^{m+1}}{(m+2)^2|x|^2-4t^{m+2}}, \q
e_1=\frac{(m+2) |x|^2 t^{m/2} \bigl( 4t^{m+2}+(m+2)^2 |x|^2
\bigr)}{2t \bigl( 4t^{m+2}-(m+2)^2 |x|^2\bigr)^2},\\
&f_1=\frac{nx_i\bigl( 4t^{m+2}+(m+2)^2 |x|^2\bigr)}{2t |x| \bigl(
4t^{m+2}-(m+2)^2 |x|^2 \bigr)}, \q g_1=-\frac{n}{2}\,
t^{m/2-1},
\end{align*}
and $h_{11}, h_{12}, h_1^j$ are admissible in $\O_1$.  One further has
in $\O_1$ that
\begin{multline}\label{3-10} 
\bar{V}_i = \frac{2 t^{\frac{m+2}{2}} x_i}{(m+2) |x|^2}\,V_0
+ \frac{x_i \bigl((m+2)^2 |x|^2 -4t^{m+2} \bigr)}{(m+2)|x|^3
t^{m/2}} N_1^0 -\frac{2
t^{\frac{m+2}{2}}}{|x|^2} \sum_{k \ne i}x_k L_{ik}, \\
= \frac{(m+2)x_i}{2t^{\frac{m+2}{2}}}\, V_0 + \frac{4t^{m+2}-(m+2)^2
|x|^2}{t^{m/2}|x|}\, N_1^i - \frac{(m+2)^2}{2
t^{\frac{m+2}{2}}} \sum_{ k \ne i} x_k L_{ik}.
\end{multline}
Substituting \eqref{3-10} into \eqref{3-8} and \eqref{3-9} yields
\eqref{3-1}.

(2) \ It follows from a direct computation that
\begin{multline}\label{3-11} 
(N_2^i)^2 = a_2 Q_m+b_2 V_0^2 +c_2 x_i N_2^iV_0 +d_2
\sum_{j=1}^n\bar{V}_j^2 +e_2 \sum_{k=1}^n x_k V_0
\bar{V}_k \\ 
+\, f_2 x_i \sum_{j=1}^n x_j N_2^i \bar{V}_j +g_2 \sum_{ k\ne i} x_k 
\bar{V}_i L_{ik} + h_{21} V_0   + h_{22} N_2^i + \sum_{j \ne i} h_2^j \bar{V}_j,
\end{multline}
where
\begin{align*} 
a_2&= \frac{4 t^3
(|x|-\frac{2}{m+2}t^{\frac{m+2}{2}})^2}{(m+2)^2|x|^2-4t^{m+2}},
\q  b_2=-\,\frac{t(|x|-\frac{2}{m+2}\,t^{\frac{m+2}{2}})^2
((m+2)^2|x|^2+4t^{m+2}) }{\bigl((m+2)^2|x|^2-4t^{m+2}\bigr)^2} \\
c_2&=
\frac{2n(m+2)(|x|-\frac{2}{m+2}\,t^{\frac{m+2}{2}})}{(m+2)^2|x|^2-4t^{m+2}},
\quad d_2=
\frac{t(|x|-\frac{2}{m+2}\,t^{\frac{m+2}{2}})^2}{(m+2)^2|x|^2-4t^{m+2}},
\\
e_2&=
\frac{(m+2)(|x|-\frac{2}{m+2}\,t^{\frac{m+2}{2}})^2((m+2)^2|x|^2+4t^{m+2})
}{2t^{m/2}
\bigl((m+2)^2|x|^2-4t^{m+2}\bigr)^2}, \\
f_2&= \frac{n\bigl(|x|-\frac{2}{m+2}\,t^{\frac{m+2}{2}} \bigr)
\bigl((m+2)^2|x|^2+4t^{m+2} \bigr)   }{2 |x|^2
t^{\frac{m+2}{2}}\bigl(4t^{m+2}-(m+2)^2|x|^2\bigr)}, \quad g_2=-\,
\frac{n(|x|-\frac{2}{m+2}\,t^{\frac{m+2}{2}})^2}{2t^{\frac{m+2}{2}}|x|^2},
\end{align*}
and $h_{21}$, $h_{22}$, $h_2^j$ are admissible in $\O_2$. Note that,
for $1\le i\le n$,
\begin{equation}\label{3-12} 
\bar{V}_i=\f{1}{2t^{\f{m+2}{2}}}\Bigl((m+2)x_iV_0-(m+2)^2
\sum_{k\not=i}x_kL_{ik} -(m+2)((m+2)|x|+2t^{\f{m+2}{2}})N_2^i\Bigr).
\end{equation}
Then combining \eqref{3-11} and \eqref{3-12} yields \eqref{3-2} and
\eqref{3-3}.

\smallskip

(3) \ Since one has 
\begin{multline}\label{3-13}
(N_3^0)^2=\f{1}{4t^{m+2}-(m+2)^2|x|^2}\biggl(-4t^{m+4}Q_m
-t^{m+2}\sum_{j=1}^n \bar{V}_j^2+4t^{m+2}N_3^0V_0+(m+2) t^{m+2}V_0\\
+\,\Bigl( 2(n-1)(m+2) t^{m+2} -\f{(4+m)(m+2)^2
|x|^2}{2}\Bigr)N_3^0\biggr)
\end{multline}
and
\begin{equation}\label{3-14}
\bar{V}_i =  \frac{2 t^{\frac{m+2}{2}} x_i}{(m+2) |x|^2}V_0
+ \frac{x_i \bigl((m+2)^2 |x|^2 -4 t^{m+2} \bigr)}{(m+2)|x|^2
t^{\frac{m+2}{2}}} N_3^0 -\frac{2 t^{\frac{m+2}{2}}}{|x|^2} \sum_{k
\ne i}x_k L_{ik},\quad i=1,\dots, n,
\end{equation} 
it follows from \eqref{3-13} and \eqref{3-14} that \eqref{3-4} and
\eqref{3-5} hold.

\smallskip

(4) \ A direct computation yields for $1\le i\le n$
\begin{multline}\label{3-15} 
(N_4^i)^2 = a_4 Q_m+b_4 V_0^2 +c_4 x_i N_4^iV_0 +d_4
\sum_{j=1}^n\bar{V}_j^2 +e_4 \sum_{k=1}^n x_k V_0
\bar{V}_k \\
+\, f_4 x_i \sum_{j=1}^n x_j N_4^i \bar{V}_j 
+g_4 \sum_{ k\ne i} x_k \bar{V}_i L_{ik}
+ h_{41} V_0   + h_{42} N_4^i + \sum_{j \ne i} h_4^j \bar{V}_j,
\end{multline}
where
\begin{align*} 
&a_4=\frac{4t^{m+5}}{ (m+2)^2|x|^2-4t^{m+2}},  \q
b_4=-\frac{t^{m+3} \bigl( 4t^{m+2}+ (m+2)^2|x|^2\bigr)}{4t^{m+2}-
(m+2)^2|x|^2}, \q
c_4=\frac{2n(m+2)t^{\frac{m+2}{2}}}{(m+2)^2|x|^2-4t^{m+2}},\\
&d_4=\frac{t^{m+3}}{(m+2)^2|x|^2-4t^{m+2}}, \q e_4=\frac{(m+2)
t^{\frac{m+4}{2}}\bigl((m+2)^2|x|^2+4t^{m+2}
\bigr)}{2\bigl((m+2)^2|x|^2-4t^{m+2}\bigr)},\\
&f_4=\frac{n\bigl( (m+2)^2|x|^2+4t^{m+2}\bigr)}{2|x|^2 \bigl(
4t^{m+2}-(m+2)^2|x|^2\bigr)}, \q g_4=-\,\frac{n t^{\frac{m+2}{2}}}{2
|x|^2},
\end{align*}
and $h_{41}$, $h_{42}$, $h_4^j$ are admissible in $\O_4$.
In addition,
\begin{equation}\label{3-16}
\bar{V}_i=\frac{(m+2)x_i}{2t^{\frac{m+2}{2}}}\, V_0 +
\frac{4t^{m+2}-(m+2)^2 |x|^2}{t^{\frac{m+2}{2}}} N_4^i -
\frac{(m+2)^2}{2 t^{\frac{m+2}{2}}} \sum_{ k \ne i} x_k L_{ik},\quad
1\le i\le n.
\end{equation}
Then substituting \eqref{3-16} into \eqref{3-15} yields \eqref{3-6}
and \eqref{3-7}.
\end{proof}

Next we define admissible tangent vector fields with respect to the
hypersurfaces $\G_{m}^{\pm} \cup \Sigma_0$. As before, we denote
$V^{(m)}$ by $V$.

\begin{definition}[Admissible tangent vector fields for
$\G_m^{\pm} \cup \Sigma_0 $]\label{def3-4}

${ }$

(1) (Near $\{t=0\}$) \ For $W_1= \{(t,x)\colon 0\le t<C|x_1|\le\ve\}$,
  $\mathcal{M}_1$ denotes the Lie algebra of vector fields in $W_1$ with
  admissible coefficients generated by $x_1\p_1$, $x_1\p_t$,
  $R_\ell$ ($2 \le \ell \le n$).

(2) (Near $\G_m^{\pm}$) \ For $W_{2,\pm}=\{(t,x)\colon 0<|x_1|<Ct\le
  \ve\}\cap\{(t,x)\colon
  |x_1\mp\f{2}{m+2}t^{\f{m+2}{2}}|<Ct^{\f{m+2}{2}}\}$,
  $\mathcal{M}_{2,\pm}$ denotes the Lie algebra of vector fields in
  $W_{2,\pm}$ with admissible coefficients generated by $V$,
  $\bar{V}_1$, $R_\ell$ ($2 \le \ell \le n$).

(3) (Near $\Sigma_0$) \ Let $W_3=\{(t,x)\colon
  |x_1|<Ct\le\ve\}\cap\{(t,x)\colon
  t^{\f{m+2}{2}}<C|x_1\mp\f{2}{m+2}t^{\f{m+2}{2}}|\}$, $\mathcal{M}_3$
  denotes the Lie algebra of vector fields in $W_3$ with admissible
  coefficients generated by $t\p_t$, $V$, $R_\ell$ ($2 \le \ell \le
  n$).

(4) (Between $\G_m^{\pm}$ and $\Sigma_0$) \ For $W_4 \{(t,x)\colon
    \f{Ct}{1+C}<|x_1|<Ct\le\ve\}\cap\{(t,x)\colon
    t^{\f{m+2}{2}}<C|x_1\mp\f{2}{m+2}t^{\f{m+2}{2}}|\}$,
    $\mathcal{M}_4$ is the Lie algebra of vector fields in $W_4$ with
    admissible coefficients generated by $t\p_t$,
    $t^{\f{m+2}{2}}\p_1$, $R_\ell$ ($2 \le \ell \le n$).
\end{definition}

\begin{remark}\label{rem3-3} 
In $W_{2,\pm}$, for computations we will also use the equivalent
vector fields $V$, $N_{2,\pm}$, $R_2,\dots, R_n$ with $N_{2,\pm}=
\bigl(x_1\mp\f{2}{m+2}t^{\f{m+2}{2}}\bigr)\p_1$ instead of $V$, $\bar
V_1$, $R_2, \dots, R_n$. This equivalence stems from the fact that
\begin{align*}
N_{2,\pm}&=\f{t^{\f{m+2}2}}{(m+2)^2\bigl(x_1\pm\f{2}{m+2}\,t^{\f{m+2}{2}}
\bigr)}
\Bigl(\f{(m+2)x_1}{t^{\f{m+2}{2}}}\,V-2\bar{V}_1\Bigr),\\
\bar{V}_1&=\f{ m+2}{2 t^{\f{m+2}{2}}}\Bigl( x_1 V-
(m+2)\bigl(x_1\pm\f{2}{m+2}\,t^{\f{m+2}{2}}\bigr) N_{2,\pm}\Bigr),
\end{align*}
where all the coefficients are admissible in $W_{2,\pm}$.
\end{remark}

We similarly define the conormal space
$I^{\infty}H^s(\G_m^{\pm} \cup \Sigma_0)$ for $0\le s<1/2$.

\begin{definition}[Conormal space $I^{\infty}H^s(\G_m^{\pm}
\cup\Sigma_0 )$]\label{def3-5} 
Define $u(t,x)\in I^{\infty}H^s(\G_m^{\pm})$ in $t> 0$
if, away from $\{t=x_1=0\}$ and near $\G_m^{\pm}$, $Z_1\dots Z_ku\in
L^{\infty}((0, T), H^s(\mathbb R^n))$ for all smooth vector fields
$Z_1,\dots,Z_k\in \{V, \bar{V}_1, R_2, \dots, R_n\}$,
away from $\{t=x_1=0\}$ and near $\Sigma_0$, $Z_1\dots Z_ku\in
L^{\infty}((0, T), H^s(\mathbb R^n))$ for all smooth vector fields
$Z_1,\dots,Z_k\in \{t\p_t, V_0, L_{ij}, 1 \le i<j \le n\}$. Near
$\{t=x_1=0\}$, the following properties hold:

(1) \ If $h_1(t,x_1)\in C^{\infty}(\mathbb R^{2}\setminus\{0\})$ is
homogeneous of degree zero and supported in $W_1=\{(t,x_1)\colon
t<C|x_1|\le\ve\}$, then $Z_1\dots Z_k\big(h_1(t,x_1)u(t,x) \big)\in
L^{\infty}((0, T), H^{s}(\mathbb R^{n}))$ for all $Z_1,\dots, Z_k\in
\mathcal{M}_1$.

(2) \ If $h_2(t,x_1)\in C^{\infty}(\mathbb R^{2}\setminus\{0\})$ is
homogeneous of degree zero and supported in $\{(t,x_1)\colon
0<|x_1|<Ct\le\ve\}$ and $\eta_{1,\pm}(\theta)\in C^{\infty}$ has
compact support near $\{\theta=\pm 1\}$, then $Z_1\dots Z_k
\Bigl(h_2(t,x_1)\eta_{1,\pm}\Bigl(\f{(m+2)x_1}{2t^{\f{m+2}{2}}}\Bigr)$
\linebreak $ u(t,x) \Bigr)\in L^{\infty}((0, T), H^{s}(\mathbb
R^{n}))$ for all $Z_1,\dots, Z_k\in \mathcal{M}_{2,\pm}$.

(3) \ If $h_3(t,x_1)\in C^{\infty}(\mathbb R^{2}\setminus\{0\})$ is
homogeneous of degree zero and supported in $\{(t,x_1)\colon
|x_1|<Ct\le\ve\}$ and $ \eta_{2,\pm}(\theta)\in C^{\infty}$ has
compact support away from $\{\theta=\pm 1\}$, then $Z_1\dots
Z_k\Bigl(h_3(t,x_1)\eta_{2,\pm}\Bigl(\f{(m+2)x_1}{2t^{\f{m+2}{2}}}\Bigr)$
\linebreak $ u(t,x)\Bigr)\in L^{\infty}((0, T), H^{s}(\mathbb R^{n}))$
for all $Z_1,\dots, Z_j\in \mathcal{M}_3$.

(4) \ If $h_4(t,x_1)\in C^{\infty}(\mathbb R^{2}\setminus\{0\})$ is
homogeneous of degree zero and supported in $\{(t,x_1)\colon
\ds\f{Ct}{1+C}<|x_1|<Ct\le\ve\}$ and $ \eta_{3, \pm}(\theta)\in
C^{\infty}$ has compact support away from $\{\theta=\pm1\}$, then
$Z_1\dots Z_k\Bigl(h_4(t,x_1)$ \linebreak
$\eta_{3,\pm}\Bigl(\f{(m+2)x_1}{ 2t^{\f{m+2}{2}}}\Bigr) u(t,x)\Bigr)
\in L^{\infty}((0, T), H^{s}(\mathbb R^{n}))$ for all $Z_1,\dots,
Z_j\in \mathcal{M}_4$.
\end{definition}

Note that the cut-off functions $h_1$, $h_2\eta_{1,\pm}
\Bigl(\f{(m+2)x_1}{2t^{\f{m+2}{2}}}\Bigr)$,
$h_3\eta_{2,\pm}\Bigl(\f{(m+2)x_1}{2t^{\f{m+2}{2}}}\Bigr)$, and
$h_4\eta_{3,\pm}\Bigl(\f{(m+2)x_1}{2t^{\f{m+2}{2}}}\Bigr)$ are
admissible in the regions $W_1$, $W_{2,\pm}$, $W_3$, and $W_4$
respectively. Moreover, they belong to the space $L^{\infty}((0, T),$
\linebreak $ H^{n/2-}(\mathbb R^n))$.

Similar to Proposition \ref{prop3-3}, one has:

\begin{proposition}\label{prop3-4}  
Let $W_1, W_{2,\pm}$, $W_3$, and $W_4$ be given as in\/
\textup{Definition \ref{def3-4}}. Then one has\/\textup{:}

\textup{(1)} \ With $N_1=x_1\p_{t}$ in $W_1$,
\begin{multline*} 
N_1^2=\f{1}{(m+2)^2x_1^2-4t^{m+2}}
\biggl((m+2)^2x_1^4 Q_m +x_1^2t^m V^2-4x_1t^{m+1}N_1 V \\
+\, (m+2)^2 x_1^4t^m  \sum_{i=2}^n R_i^2 
 -(m+2)x_1^2t^m
V+2(m+4)x_1t^{m+1}N_1\biggr).
\end{multline*}

\textup{(2)} \ With $N_{2,\pm}=
\bigl(x_1\mp\f{2}{m+2}\,t^{\f{m+2}{2}}\bigr)\,\p_1$ in $W_{2,\pm}$,
\begin{multline*} 
N_{2,\pm}^2=\f{x_1\mp\f{2}{m+2}
t^{\f{m+2}{2}}}{(m+2)^2(x_1\pm\f{2}{m+2}t^{\f{m+2}{2}})} \biggl(4t^2
Q_m- V^2 +4t^{m+2} \sum_{i=2}^nR_i^2+2V\biggr)\\
+\, \f{2x_1}{(m+2)\bigl(x_1\pm\f{2}{m+2}\,t^{\f{m+2}{2}}\bigr)}\,N_{2,\pm}V
-\f{2\bigl(x_1 \pm
t^{\f{m+2}{2}}\bigr)}{(m+2)\bigl(x_1\pm\f{2}{m+2}\,t^{\f{m+2}{2}}\bigr)}
\,N_{2,\pm}.
\end{multline*}

\textup{(3)} \ With $N_3=t\p_{t}$ in $W_3$,
\begin{multline*} 
N_3^2=\f{1}{(m+2)^2x_1^2-4t^{m+2}}\biggl((m+2)^2x_1^2t^2 Q_m
+t^{m+2} V^2-4t^{m+2}N_3V+(m+2)^2x_1^2t^{m+2}\sum_{i=2}^nR_i^2\\
-\,(m+2)t^{m+2}V
+\bigl((m+2)^2x_1^2+2(m+2)t^{m+2}\bigr)N_3\biggr).
\end{multline*}

\textup{(4)} \ With $N_{4}=t^{\f{m+2}{2}}\p_1$ in $W_4$,
\begin{multline*} 
N_{4}^2 =\f{1}{(m+2)^2x_1^2-4t^{m+2}}
\biggl(4t^{m+4}Q_m-t^{m+2}V^2+2(m+2)x_1t^{\f{m+2}{2}}N_{4}V+4t^{2(m+2)}
\sum_{i=2}^nR_i^2\\
+\,2t^{m+2}V -3(m+2)^2 x_1t^{\f{m+2}{2}}N_{4}\biggr).
\end{multline*}
\end{proposition}

\begin{remark}\label{rem3-4} 
As in Remark \ref{rem3-2}, one verifies that, in
Proposition~\ref{prop3-4}, all coefficients are admissible in the
corresponding region with respect to the corresponding vector fields.
\end{remark}

Without loss of generality, one can assume that $m_1 >m_2$ and $0<t <T
\le 1$.

\begin{definition}[Admissible tangent vector fields for
$\G_{m_1} \cup \G_{m_2}$]\label{def3-6}

${ }$

(1) (Near $\{t=0\}$) \ For $D_1= \{(t,x)\colon t<C|x|\le \ve\}$, $X_1$
  denotes be the Lie algebra of vector fields in $D_1$ with admissible
  coefficients generated by $|x|\,\p_t$,
  $t^{m_2/2}|x|\,\p_\ell$ ($1 \le \ell \le n$), $L_{ij}$ ($1 \le i<j
  \le n$).

(2) (Near $\G_{m_2}$) \ For $D_2 =\{(t,x)\colon
  0<|x|<Ct\le\ve\}\cap\bigl\{(t,x)\colon \bigl||x|\mp
  \f{2}{m_2+2}t^{\f{m_2+2}{2}} \bigr|<Ct^{\f{m_2+2}{2}}\bigr\}$, $X_2$
  denotes the Lie algebra of vector fields in $D_2$ with admissible
  coefficients generated by $V_0^{(m_2)}$, $\bar{V}_\ell^{(m_2)}$ ($1
  \le \ell \le n$), $L_{ij}$ ($1\le i<j \le n$).

(3) (Between $\G_{m_2}$ and $\G_{m_1}$) \ Let $D_3= \{(t,x)\colon
  0<|x|<Ct\le\ve\}\cap \bigl\{(t,x)\colon t^{\f{m_2+2}{2}}<C
  \bigl||x|\mp\f{2}{m_2+2}t^{\f{m_2+2}{2}}\bigr|\bigr\} \cap
  \bigl\{(t,x)\colon \frac{2}{m_1+2} t^{\f{m_1+2}{2}}< |x| <
  \frac{2}{m_2+2} t^{\f{m_2+2}{2}}\bigr\} \cap \bigl\{(t,x)\colon
  t^{\f{m_1+2}{2}}<C
  \bigl||x|\mp\f{2}{m_1+2}t^{\f{m_1+2}{2}}\bigr|\bigr\}$, $X_3$
  denotes the Lie algebra of vector fields in $D_3$ with admissible
  coefficients generated by $t\p_t$, $t^{m_2+1}\, \p_\ell$ ($1 \le
  \ell \le n$), $L_{ij}$ ($1 \le i<j \le n$).

(4) (Near $\G_{m_1}$) \ For $D_4= \{(t,x)\colon 0<|x|<Ct\le\ve\}\cap
  \bigl\{(t,x)\colon t^{\f{m_2+2}{2}}<C
  \bigl||x|\mp\f{2}{m_2+2}t^{\f{m_2+2}{2}}\bigr|\bigr\}\cap
  \bigl\{(t,x)\colon \bigl||x|\mp\f{2}{m_1+2}t^{\f{m_1+2}{2}}\bigr| <
  C t^{\frac{m_1+2}{2}} \bigr\}$, $X_4$ is the Lie algebra of vector
  fields in $D_4$ with admissible coefficients generated by
  $V_0^{(m_1)}$, $\bar{V}_\ell^{(m_1)}$ ($1 \le \ell \le n$), $L_{ij}$
  ($1 \le i<j \le n$).

(5) (Inside $\G_{m_1}$) \ Let $D_5$ be the region $\{(t,x):
  t^{\f{m_2+2}{2}}<C
  \big||x|\mp\f{2}{m_2+2}t^{\f{m_2+2}{2}}\big|\}\cap\{(t,x):
  t^{\f{m_1+2}{2}}<C \big||x|\mp\f{2}{m_1+2}t^{\f{m_1+2}{2}}\big|\}$
  and $X_5$ be the Lie algebra of vector fields with admissible
  coefficients on $D_5$ generated by $\{t\p_t; \ t^{m_1+1}\p_\ell, 1
  \le \ell \le n; \ L_{ij}, 1 \le i<j \le n\}$.
\end{definition}

Then we define the conormal spaces $I^\infty H^s(\G_{m_1}
\cup \G_{m_2})$ with $s < n/2$ and $m_1>m_2$.

\begin{definition}[Conormal space $I^{\infty}H^s(\G_{m_1}
\cup \G_{m_2})$]\label{def3-7} 
Define $u(t,x)\in I^{\infty}H^s(\G_{m_1} \cup \G_{m_2})$
in $t > 0$ if, away from $\{t=|x|=0\}$ but near $\G_{m_i}$ ($i=1,2$),
$Z_1\dots Z_ku\in L_{\textup{loc}}^{\infty}((0, T], H^s(\mathbb R^n))$ for all
  smooth vector fields $Z_1,\dots,Z_k\in \{V_0^{(m_i)}, \bar{V_\ell^{(m_i)}}, 
1 \le \ell \le n, \ L_{kj}, 1\le k<j \le n\}$, and near $\{t=|x|=0\}$, 
the following properties hold:

(1) \ If $h_1(t,x)\in C^{\infty}(\mathbb R^{n+1}\setminus\{0\})$ is
homogeneous of degree zero and supported on $D_1=\{(t,x): 0<
t<C|x|\le\ve\}$, then $Z_1\dots Z_k\bigl(h_1(t,x)u(t,x)\bigr)\in
L^{\infty}((0, T), H^{s}(\mathbb R^{n}))$ for all $Z_1,\dots, Z_k\in
X_1$.

(2) \ If $h_2(t,x)\in C^{\infty}(\mathbb R^{n+1}\setminus\{0\})$ is
homogeneous of degree zero and supported on $\{(t,x):
0<|x|<Ct\le\ve\}$ and $\chi_0(\theta)\in C^{\infty}$ has compact
support near $\{\theta=1\}$, then $Z_1\dots
Z_k\Bigl(h_2(t,x)\chi_0\Bigl(\f{(m_2+2)|x|}{2t^{\f{m_2+2}{2}}}
\Bigr)u(t,x)\Bigr)\in L^{\infty}((0, T), H^{s}(\mathbb R^{n}))$ for
all $Z_1,\dots, Z_k\in X_2$.

(3) \ If $h_3(t,x)\in C^{\infty}(\mathbb R^{n+1}\setminus\{0\})$ is
homogeneous of degree zero and supported on $\{(t,x):
0<|x|<Ct\le\ve\}$, and $\chi_1(\theta)\in C^{\infty}$ has compact
support away $\{\theta=1\}$, $\chi_2(\theta)\in C^{\infty}$ has
compact support away $\{\theta=1\}$, then $Z_1\dots
Z_k\Bigl(h_3(t,x)\chi_1\Bigl(\f{(m_2+2)|x|}{2t^{\f{m_2+2}{2}}}\Bigr)
\chi_2\Bigl(\f{(m_1+2)|x|}{2t^{\f{m_1+2}{2}}}\Bigr)u(t,x)
\Bigr)\in L^{\infty}((0, T), H^{s}(\mathbb R^{n}))$ for all $Z_1,\dots,
Z_j\in X_3$.

(4) \ If $h_4(t,x)\in C^{\infty}(\mathbb R^{n+1}\setminus\{0\})$ is
homogeneous of degree zero and supported on $\{(t,x):
0<|x|<Ct\le\ve\}$, and $\chi_3(\theta)\in C^{\infty}$ has compact
support away $\{\theta=1\}$, $\chi_4(\theta)\in C^{\infty}$ has
compact support near  $\{\theta=1\}$, then $Z_1\dots
Z_k\Bigl(h_4(t,x)\chi_3\Bigl(\f{(m_2+2)|x|}{2t^{\f{m_2+2}{2}}}\Bigr)
\chi_4\Bigl(\f{(m_1+2)|x|}{2t^{\f{m_1+2}{2}}}\Bigr)u(t,x)
\Bigr)\in L^{\infty}((0, T), H^{s}(\mathbb R^{n}))$ for all $Z_1,\dots,
Z_j\in X_4$.

(5) \ If $h_5(t,x)\in C^{\infty}(\mathbb R^{n+1}\setminus\{0\})$ is
homogeneous of degree zero and supported on $\{(t,x):
0<|x|<Ct\le\ve\}$, and $\chi_5(\theta)\in C^{\infty}$ has compact
support away $\{\theta=1\}$, $\chi_6(\theta)\in C^{\infty}$ has
compact support away $\{\theta=1\}$, then $Z_1\dots
Z_k\Bigl(h_5(t,x)\chi_5\Bigl(\f{(m_2+2)|x|}{2t^{\f{m_2+2}{2}}}\Bigr)
\chi_6\Bigl(\f{(m_1+2)|x|}{2t^{\f{m_1+2}{2}}} \Bigr)u\Bigr)\in
L^{\infty}((0, T), H^{s}(\mathbb R^{n}))$ for all $Z_1,\dots, Z_j\in
X_5$.
\end{definition}

One similarly defines admissible tangent vector fields for
$\G_{m_1}^{\pm} \cup \G_{m_2}^{\pm}$ and the conormal spaces $I^\infty
H^s(\G_{m_1}^{\pm} \cup \G_{m_2}^{\pm})$ with $s < n/2$ and
$m_1>m_2$.

\begin{definition}[Admissible tangent vector fields for
$\G_{m_1}^{\pm} \cup \G_{m_2}^{\pm}$]\label{def3-8}

${ }$

(1) (Near $\{t=0\}$) \ For $E_1=\{(t,x)\colon t<C|x_1|\le \ve\}$,
  $Y_1$ is the Lie algebra of vector fields in $E_1$ with admissible
  coefficients generated by $x_1\p_t$, \ $x_1\p_1$, $\p_\ell$ \ ($2 \le
  \ell \le n$).

(2) (Near $\G_{m_2}^{\pm}$) For $E_2=\{(t,x)\colon
  0<|x_1|<Ct\le\ve\}\cap\{(t,x)\colon \big||x_1|\mp
  \f{2}{m_2+2}\,t^{\f{m_2+2}{2}} \big|<Ct^{\f{m_2+2}{2}}\}$, $Y_2$ is
  the Lie algebra of vector fields in $E_2$ with admissible
  coefficients generated by $V^{(m_2)}$, $\bar{V}_1^{(m_2)}$,
  $\p_\ell$ \ ($2 \le \ell \le n$).

(3) (Between $\G_{m_2}^{-}$ and $\G_{m_1}^{-}$ or $\G_{m_2}^{+}$ and
  $\G_{m_1}^{+}$) For $E_3=\{(t,x)\colon 0<|x_1|<Ct\le\ve\}\cap
    \{(t,x)\colon t^{\f{m_2+2}{2}}<C
    \big||x_1|\mp\f{2}{m_2+2}\,t^{\f{m_2+2}{2}}\big|\}\cap\{(t,x)\colon
    \frac{2}{m_1+2}\,t^{\f{m_1+2}{2}}< |x_1| < \frac{2}{m_2+2}\,
    t^{\f{m_2+2}{2}}\} \cap\{(t,x)\colon t^{\f{m_1+2}{2}}<C
    \big||x_1|\mp\f{2}{m_1+2}\,t^{\f{m_1+2}{2}}\big|\}$, $Y_3$ is the
    Lie algebra of vector fields in $E_3$ with admissible coefficients
    generated by $t\p_t$, $t^{\frac{m_2+2}{2}} \p_1$, $\p_\ell$ \ ($2
    \le \ell \le n$).

(4) (Near $\G_{m_1}^{\pm}$) For $E_4= \{(t,x)\colon
    0<|x_1|<Ct\le\ve\}\cap \{(t,x)\colon t^{\f{m_2+2}{2}}<C
    \big||x|\mp\f{2}{m_2+2}\,t^{\f{m_2+2}{2}}\big|\}\cap\{(t,x)\colon
    \big||x|\mp\f{2}{m_1+2}\,t^{\f{m_1+2}{2}}\big| < C
    t^{\frac{m_1+2}{2}}\}$, $Y_4$ is the Lie algebra of vector fields
    in $E_4$ with admissible coefficients generated by $V^{(m_1)}$,
    $\bar{V}_1^{(m_1)}$, $\p_\ell$ \ ($2 \le \ell \le n$).

(5) (Between $\G_{m_1}^{-}$ and $\G_{m_1}^{+}$) For $E_5=
    \{(t,x)\colon 0<|x_1|<Ct\le\ve\}\cap \{(t,x)\colon
    t^{\f{m_2+2}{2}}<C
    \big||x|\mp\f{2}{m_2+2}\,t^{\f{m_2+2}{2}}\big|\}\cap\{(t,x)\colon
    t^{\f{m_1+2}{2}}<C
    \big||x|\mp\f{2}{m_1+2}t^{\f{m_1+2}{2}}\big|\}$, $Y_5$ is the Lie
    algebra of vector fields in $E_5$ with admissible coefficients
    generated by $t\p_t$, $t^{\frac{m_1+2}{2}}\p_1$, $\p_\ell$ \ ($2
    \le \ell \le n$).
\end{definition}

Then one defines the conormal spaces $I^\infty H^s(\G_{m_1}^{\pm} \cup
\G_{m_2}^{\pm})$ with $-n/2<s < 1/2$ and $m_1>m_2$.

\begin{definition}[Conormal space $I^{\infty}H^s(\G_{m_1}^{\pm} \cup 
\G_{m_2}^{\pm})$]\label{def3-9} A function $u(t,x)$ defined for $t>0$
  belongs to $I^{\infty}H^s(\G_{m_1}^{\pm} \cup \G_{m_2}^{\pm})$ if,
  away from $\{t=|x_1|=0\}$, but near $\G_{m_i}^{\pm}$ ($i=1,2$),
  $Z_1\dots Z_ku\in L^{\infty}((0, T), H^s(\mathbb R^n))$ for all
  smooth vector fields $Z_1,\dots,Z_k\in \{V^{(m_i)},
  \bar{V}_1^{(m_i)}, \p_\ell, 2 \le \ell \le n\}$, and, near
    $\{t=|x_1|=0\}$, the following properties hold:

(1) \ If $h_1(t,x_1)\in C^{\infty}(\mathbb R^{2}\setminus\{0\})$ is
  homogeneous of degree zero and supported in $E_1=\{(t,x_1)\colon
  t<C|x_1|\le\ve\}$, then $Z_1\dots Z_k\bigl(h_1(t,x_1)u(t,x)\bigr)
  \in L^{\infty}((0, T), H^{s}(\mathbb R^{n}))$ for all $Z_1,\dots,
  Z_k\in Y_1$.

(2) \ If $h_2(t,x_1)\in C^{\infty}(\mathbb R^{2}\setminus\{0\})$ is
  homogeneous of degree zero and supported in $\{(t,x_1)\colon
  0<|x_1|<Ct\le\ve\}$ and $\chi_{\pm}(\theta)\in C^{\infty}$ has
  compact support near $\{\theta=\pm1\}$, then $Z_1\dots
  Z_k\Bigl(h_2(t,x_1)\chi_{\pm}\Bigl(\f{(m_2+2)x_1}{2t^{\f{m_2+2}{2}}}\Bigr)u
  \Bigr)\in L^{\infty}((0, T), H^{s}(\mathbb R^{n}))$ for all
  $Z_1,\dots, Z_k\in Y_2$.

(3) \ If $h_3(t,x_1)\in C^{\infty}(\mathbb R^{2}\setminus\{0\})$ is
  homogeneous of degree zero and supported in $\{(t,x_1)\colon
  0<|x_1|<Ct\le\ve\}$, $\chi_{1,\pm}(\theta)\in C^{\infty}$ has
  compact support away from $\{\theta=\pm1\}$, and $\chi_{2,
    \pm}(\theta)\in C^{\infty}$ has compact support away from
  $\{\theta=\pm 1\}$, then $Z_1\dots
  Z_k\Bigl(h_3(t,x_1)\chi_{1,\pm}\Bigl(\f{(m_2+2)x_1}{2t^{\f{m_2+2}{2}}}\Bigr)
  \chi_{2,\pm}\Bigl(\f{(m_1+2)x_1}{2t^{\f{m_1+2}{2}}}\Bigr)u \Bigr)\in
  L^{\infty}((0, T), H^{s}(\mathbb R^{n}))$ for all $Z_1,\dots, Z_j\in
  Y_3$.

(4) \ If $h_4(t,x_1)\in C^{\infty}(\mathbb R^{2}\setminus\{0\})$ is
  homogeneous of degree zero and supported in $\{(t,x_1)\colon
  0<|x_1|<Ct\le\ve\}$, $\chi_{3, \pm}(\theta)\in C^{\infty}$ has
  compact support away from $\{\theta=\pm 1\}$, and $\chi_{4,
    \pm}(\theta)\in C^{\infty}$ has compact support near $\{\theta=\pm
  1\}$, then $Z_1\dots Z_k\Bigl(h_4(t,x_1)\chi_{3,
    \pm}\Bigl(\f{(m_2+2)x_1}{2t^{\f{m_2+2}{2}}}\Bigr)\chi_{4,
    \pm}\Bigl(\f{(m_1+2)x_1}{2t^{\f{m_1+2}{2}}}\Bigr)u \Bigr)\in
  L^{\infty}((0, T), H^{s}(\mathbb R^{n}))$ for all $Z_1,\dots, Z_j\in
  Y_4$.

(5) \ If $h_5(t,x_1)\in C^{\infty}(\mathbb R^{2}\setminus\{0\})$ is
  homogeneous of degree zero and supported in $\{(t,x_1)\colon
  0<|x_1|<Ct\le\ve\}$, $\chi_{5, \pm}(\theta)\in C^{\infty}$ has
  compact support away from $\{\theta=\pm 1\}$, and $\chi_{6,
    \pm}(\theta)\in C^{\infty}$ has compact support away from
  $\{\theta=\pm 1\}$, then $Z_1\dots Z_k\Bigl(h_5(t,x_1)\chi_{5,
    \pm}\Bigl(\f{(m_2+2)x_1}{2t^{\f{m_2+2}{2}}}\Bigr)\chi_{6,
    \pm}\Bigl(\f{(m_1+2)x_1}{2t^{\f{m_1+2}{2}}}\Bigr)u\Bigr) \in
  L^{\infty}((0, T), H^{s}(\mathbb R^{n}))$ for all $Z_1,\dots, Z_j\in
  Y_5$.
\end{definition}


\section{Local existence of solutions of Eqs. \eqref{1-1} and 
\eqref{1-2}}\label{sec4}

In this section, we will use the Banach fixed point theorem to obtain
the local existence of low regularity solutions of \eqref{1-1} and
\eqref{1-2}. The method is to reduce both the third-order equation in
\eqref{1-1} and the fourth-order equation in \eqref{1-2} to the
corresponding Tricomi-type problem. Let us stress that the conditions
on the initial data are much weaker than those in \cite{rwy12}. Thanks
to Lemmas \ref{lem2-3}--\ref{lem2-6}, we are able to overcome the
difficulties induced by the low regularity.

It is readily seen that problem \eqref{1-1} is equivalent to the
second-order degenerate hyperbolic equation
\begin{equation}\label{4-1}
\left\{ \enspace
\begin{aligned}
&\p_t^2u-t^{m}\Delta_x u= \varphi_2(x)+ \int_0^t
f(s,x,u)\,ds, \quad (t,x)\in (0,
+\infty)\times\mathbb R^n,\\
&u(0,x)=\varphi_0(x),\quad \p_t{u}(0,x)=\vp_1(x).
\end{aligned}
\right.
\end{equation}
which contains a nonlocal nonlinear term.

Let us first consider problem (4.1) under assumption \eqref{a2} which
is easier to handle than assumption \eqref{a1}.

\begin{theorem}\label{thm4-1}
Let assumption \textup{\eqref{a2}} hold. If $f(t,x,u)$ satisfies the
assumptions of \textup{Theorem~\ref{thm1-2}}, then there is a constant
$0<T\le T_0$ such that \textup{(4.1)} has a local solution $u\in
L_{\textup{loc}}^{\infty}((0,T)\times \R^n) \cap C([0,T],
H^{n/2-}(\mathbb R^n))\cap C\bigl((0, T],
  H^{n/2+\f{m}{2(m+2)}-}(\mathbb R^n)\bigr)\cap C^1\bigl([0,T],
  H^{n/2-\f{m+4}{2(m+2)}-}(\mathbb R^n)\bigr)$.
\end{theorem}

\begin{proof} 
Observe that the $\varphi_j(x)$ ($0 \le j \le 2$) belong to
$H^{n/2-}(\R^n)$ under assumption \eqref{a2}.  Let
$u_1(t,x)$ satisfy
\begin{equation}\label{4-2}
\left\{ \enspace
\begin{aligned} &\p_t^2u_1-t^{m}\Delta_x u_1=0, \quad 
(t,x)\in (0,+\infty)\times\mathbb R^n, \\
&u_1(0,x)=\varphi_0(x),\quad
\p_t{u_1}(0,x)=\vp_1(x).
\end{aligned}
\right.
\end{equation}
For any small $\delta>0$, it follows from Lemmas~\ref{lem2-1}
and \ref{lem2-4} (choose $s=n/2 -\delta$) that
\begin{multline*}
u_1\in L_{\text{loc}}^\infty((0,T_0)\times \R^n) \\
\cap C([0,T_0],
H^{n/2-\dl}(\mathbb R^n))\cap C\bigl((0, T_0],
H^{n/2+\f{m}{2(m+2)}-\dl}(\mathbb R^n)\bigr)\cap C^1\bigl([0,T_0],
H^{n/2-\f{m+4}{2(m+2)}-\dl}(\mathbb R^n)\bigr)
\end{multline*}
which satisfies, for any $t\in (0, T_0]$,
\[
\left\{ \enspace
\begin{aligned} 
&\|u_1(t,
\cdot)\|_{L^\infty(\R^n)} \le C(\dl)\left(1+|\ln t|^2\right),\\
&\|u_1(t,\cdot)\|_{H^{n/2-\dl}(\mathbb R^n)}+
t^{m/4}\,\|u_1(t,\cdot)\|_{H^{n/2+\f{m}{2(m+2)}-\dl}(\mathbb
R^n)}\\
&\hspace{40mm}+\|\p_tu_1(t,\cdot)\|_{H^{n/2-\f{m+4}{2(m+2)}-\dl}(\mathbb
R^n)}\le C(\dl). 
\end{aligned}
\right.
\]

Let $u_2(t,x)$ be the solution of
\begin{equation}\label{4-4}
\left\{ \enspace
\begin{aligned} 
&\p_t^2u_2-t^m\Delta_x u_2=\varphi_2(x)+ \int_0^t f(s,x,0)\,ds,
\quad (t,x)\in (0,
+\infty)\times\mathbb R^n,\\
&u_2(0,x)=\p_t{u_2}(0,x)=0.
\end{aligned}
\right.
\end{equation}
From Lemmas~\ref{lem2-3}\,(ii) and \ref{lem2-5}\,(i) one has that
(choose $s=n/2 -\delta$ and $p_3=\frac{m}{2(m+2)}$)
\begin{multline*}
u_2\in L^\infty((0,T_0)\times \R^n) \\
\cap C([0,T_0],
H^{n/2-\dl}(\mathbb R^n))\cap C\bigl((0, T_0],
H^{n/2+\f{m}{2(m+2)}-\dl}(\mathbb R^n)\bigr)\cap C^1\bigl([0,T_0],
H^{n/2-\f{m}{2(m+2)}-\dl}(\mathbb R^n)\bigr)
\end{multline*}
which satisfies, for any $t\in (0, T_0]$,
\begin{multline*}
\|u_2(t, \cdot)\|_{L^\infty(\R^n)}+\|u_2(t,\cdot)\|_{H^{n/2-\dl}(\mathbb R^n)}+
t^{m/4}\,\|u_2(t,\cdot)\|_{H^{n/2+\f{m}{2(m+2)}-\dl}(\mathbb
R^n)}\\ +\|\p_tu_2(t,\cdot)\|_{H^{n/2-\f{m}{2(m+2)}-\dl}(\mathbb
R^n)}\le C(\dl).
\end{multline*}

Set $v(t,x)=u(t,x)-u_1(t,x)-u_2(t,x)$. It follows from \eqref{4-1},
\eqref{4-2}, and \eqref{4-4} that $v$ is a solution of
\[
\left\{ \enspace
\begin{aligned} 
&\p_t^2v-t^{m}\Delta_x v=\int_0^t
\left(f(s,x,u_1+u_2+v)-f(s,x,0)\right) ds, \quad (t,x)\in [0,
+\infty)\times\mathbb R^n, \\
&v(0,x)=\p_t v(0,x)=0.
\end{aligned}
\right.
\]
For $w\in C([0,T], H^{n/2+p_0(m)-\dl})\cap C\bigl((0,T],
H^{n/2+p_1(m)-\dl}\bigr) \cap C^1\bigl([0,T],$
$H^{n/2-\f{m}{2(m+2)}+p_2(m)-\dl}\bigr)$, where $0<T\le T_0$ and
$p>1$ is large, define
\begin{multline*}
\interleave w\interleave \equiv \sup_{0\le t \le
T}\|w(t,\cdot)\|_{H^{n/2+ p_0(m)-\dl}(\mathbb R^n)} +\sup_{0<t \le
T}t^{1/p+\f{(m+2)p_1(m)}{2}-2}\|w(t,\cdot)\|_{H^{n/2+p_1(m)-\dl}(\mathbb
R^n)}\\
+\sup_{0\le t \le T}\|\p_tw(t,\cdot)\|_{H^{n/2-\f{m}{2(m+2)}+
p_2(m)-\dl}(\mathbb R^n)},
\end{multline*}
where $p_0(m)=\min\Bigl\{\f{4p-2}{p(m+2)}, 1\Bigr\},
p_1(m)=\min\Bigl\{\f{p(m+8)-4}{2p(m+2)},1\Bigr\}$, and
$p_2(m)=\min\Bigl\{\frac{2(p-1)}{p(m+2)}, \f{m}{2(m+2)}\Bigr\}$. Let
the set $G$ be defined by
\begin{multline*}
G\equiv \Bigl\{w\in C([0,T], H^{n/2+ p_0(m)-\dl}) \\ \cap C\bigl((0,T],
H^{n/2+p_1(m)-\dl}\bigr)\cap C^1\bigl([0,T],
H^{n/2-\f{m}{2(m+2)}+p_2(m)-\dl}\bigr)\colon\, 
\interleave w\interleave  \leq1\Bigr\}. 
\end{multline*}
For $w \in G$, one has $u_1+u_2+w \in L_{\text{loc}}^\infty((0,T)
\times \R^n)\cap L^q((0,T) \times \R^n)$ for all $1<q<\infty$. Let
\begin{multline}\label{4-8} 
E\left(f(t,x,u)-f(t,x,0)\right) \\
\equiv \biggl(\int_0^t
(V_2(t,|\xi|)V_1(\tau,|\xi|)-V_1(t,|\xi|)V_2(\tau,|\xi|))
\int_0^\tau (f(s,x,u(s,x))-f(s,x,0))^\wedge(\xi) \, ds
d\tau\biggr)^\vee(t,x) 
\end{multline}
and define the nonlinear map $\mathcal F$ by
\[
\mathcal F(w)=E\left(f(t,x,u_1+u_2+w)-f(t,x,0)\right).
\]

We will show that $\mathcal F$ maps $G$ into itself, and that it is a
contraction for small $T>0$.

By Lemma~\ref{lem2-3}\,(i) (with $p_1=p_0(m)-\dl/2<p_1(m)$),
the H\"older inequality, and the polynomial increase of $f(t,x,u)$ with
respect to the variable $u$, for $w\in G$ and $T>0$ small, one has
\begin{multline*}
\|\mathcal F(w)(t,\cdot)\|_{H^{n/2+p_0(m)-\dl}} 
\le C(\dl)\, t^{2-\f{(m+2)p_0(m)}{2}+\f{(m+2)\dl}{4}-1/p} \\
\begin{aligned} 
& \hspace{30mm}  \times 
\left\|\int_0^t
\left(f(s,\cdot,u_1(s,\cdot)+u_2(s,\cdot)+w(s,\cdot))-f(s,\cdot,0)\right)
ds\right\|_{L^p((0, T),H^{n/2-\dl/2})}\\
& \le C(\dl)\,T^{3-\f{(m+2)p_0(m)}{2}+\f{(m+2)\dl}{4}-1/p} 
\\
&\hspace{30mm}\times \bigl\|\bigl(f(t,\cdot,u_1(t,\cdot)
+u_2(t,\cdot)+w(t,\cdot))-f(t,\cdot,0)
\bigr)\bigr\|_{L^p((0,T),H^{n/2-\dl/2})}\\
& \le C(\dl)\,T^{3-\f{(m+2)p_0(m)}{2}+\f{(m+2)\dl}{4}-1/p}\left\|
u_1(t,\cdot)+u_2(t,\cdot)
+w(t,\cdot)\right\|_{L^{\infty}((0,T),H^{n/2
-\dl/2})}\\
&\hspace{30mm}\times \left(\int_0^T
\left(1+|\ln t|^2\right)^{pK}dt\right)^{1/p} \\
& \le
C(\dl)\,T^{3-\f{(m+2)p_0(m)}{2}+\f{(m+2)\dl}{4}-1/p}.
\end{aligned}
\end{multline*}
For $T>0$ small, one obtains that
\begin{equation}\label{4-11}
\|\mathcal F(w)(t,\cdot)\|_{H^{n/2+p_0(m)-\dl}}\le\f13.
\end{equation}
Moreover, by Lemma~\ref{lem2-3}\,(i) (with
$p_1=p_1(m)-\dl/2$), one has
\begin{align*} 
\|\mathcal F(w)(t,\cdot)\|_{H^{n/2+ p_1(m)-\dl}} &\le
Ct^{2-\f{(m+2)p_1(m)}{2}+\f{(m+2)\dl}{4}-1/p} \\
& \qquad\qquad \times
\left\|f(t,\cdot,u_1(t,\cdot)+u_2(t,\cdot)+w(t,\cdot))-f(t,\cdot,0)
\right\|_{L^p((0,T),H^{n/2-\dl/2})}\\
&\le Ct^{2-\f{(m+2)p_1(m)}{2}+\f{(m+2)\dl}{4}-1/p},
\end{align*}
which yields, for $T>0$ small and $t\in (0, T]$,
\begin{equation}\label{4-12}
t^{1/p+\f{(m+2)p_1(m)}{2}-2}\left\|\mathcal F(w)(t,\cdot)
\right\|_{H^{n/2+p_1(m)-\dl}(\mathbb R^2)}
\le CT^{\f{(m+2)\dl}{4}}\le\f13.
\end{equation}
For $p_2=p_2(m)-\dl/2$ in Lemma~\ref{lem2-3}\,(i) and $T>0$ small, one
has
\begin{multline}\label{4-13} 
\|\p_t\mathcal
F(w)(t,\cdot)\|_{H^{n/2-\f{m}{2(m+2)}+p_2(m)-\dl}} \le C 
t^{1-\f{m+2}{2}p_2(m)+\frac{\dl(m+2)}{4}-1/p} \\
\times 
\|f(t,\cdot,u_1(t,\cdot)+u_2(t,\cdot)+w(t,\cdot))
-f(t,\cdot,0)\|_{L^p((0,T],H^{n/2-\frac{\delta}{2}})} \le CT^{\f{(m+2)\dl}{4}}\le\f13.
\end{multline}
Collecting \eqref{4-11}--\eqref{4-13} yields, for $T>0$ small,
\begin{equation}\label{4-14}
\interleave \mathcal F(w)\interleave \le 1,
\end{equation}
which shows  that $\mathcal F$ maps $G$ into $G$.

Next we prove that the map $\mathcal F$ is strongly contractible for
$T>0$ small. For $w_1, w_2\in G$, in view of
$f(\tau,x,u_1+u_2+w_1)-f(\tau,x,u_1+u_2+w_2)=\int_0^1f'(\tau,x,u_1+u_2+\th
w_1+(1-\th)w_2)(w_1-w_2)\,d\th$, by a direct computation as for
\eqref{4-11}--\eqref{4-13} one has that, for $T>0$ small,
\begin{multline}\label{4-15} 
\interleave \mathcal F(w_1)-\mathcal F(w_2)\interleave =\interleave
E(f(t,x,u_1+u_2+w_1(\tau,\cdot))-Ef(t,x,u_1+u_2+w_2(\tau,\cdot))\interleave
\\
\leq
C T^{\f{(m+2)\dl}{4}} \interleave w_1-w_2\interleave \le
\f12\interleave w_1-w_2\interleave.
\end{multline} 

Thus, by the Banach fixed-point theorem and
\eqref{4-14}--\eqref{4-15}, we have completed the proof of
Theorem~\ref{thm4-1}.
\end{proof}

Next, we prove the local existence of solutions of \eqref{1-1} under
assumption \eqref{a1}.

\begin{theorem}\label{thm4-2} 
Under assumption \eqref{a1}, there is a constant $T>0$ such that
\eqref{1-1} has a local bounded solution $u\in L^{\infty}((0,
    T)\times\mathbb R^n)\cap C([0,T], H^{1/2-}(\mathbb R^n))\cap
C\bigl((0, T],H^{\f{m+1}{m+2}-}(\mathbb R^n)\bigr)\cap C^1\bigl([0,T],
  H^{-\f{1}{m+2}-}(\mathbb R^n)\bigr)$.
\end{theorem}

\begin{proof}
Let $u_1(t,x)$ and $u_2(t,x)$ be defined as in \eqref{4-2} and
\eqref{4-4}, respectively. Then, for any fixed $\dl>0$ with
$\dl<\ds\f{1}{2(m+2)}$, one infers from $\varphi_j\in H^{1/2-}(\R^n)
(0\le j \le 2)$, Lemma~\ref{lem2-6}, and Lemma~\ref{lem2-1} (with
$s=1/2-\delta$) that
\begin{multline*}
u_1(t,x)\in L^{\infty}((0, T)\times\mathbb R^n) \\ \cap C([0,T],
H^{1/2-\dl}(\mathbb R^n))\cap C\bigl((0, T],
  H^{\f{m+1}{m+2}-\dl}(\mathbb R^n)\bigr)\cap C^1\bigl([0,T],
  H^{-\f{1}{m+2}-\dl}(\mathbb R^n)\bigr)
\end{multline*}
which satisfies, for $t\in (0, T]$,
\begin{equation}\label{4-16}
\left\{ \enspace
\begin{aligned} 
& \|u_1(t,\cdot)\|_{L^{\infty}(\mathbb R^n)} \le C,\\
&\|u_1(t,\cdot)\|_{H^{1/2-\dl}(\mathbb R^n)}+
t^{m/4}\,\|u_1(t,\cdot)\|_{H^{\f{m+1}{m+2}-\dl}(\mathbb
R^n)}+\|\p_tu_1(t,\cdot)\|_{H^{-\f{1}{m+2}-\dl}(\mathbb R^n)}\le C(
\dl),
\end{aligned}
\right.
\end{equation}
and from Lemma~\ref{lem2-5} and Lemma~\ref{lem2-2} (with
$s=1/2-\delta$ and $p_1=\frac{m}{2(m+2)}$),
\begin{multline*}
u_2(t,x)\in L^{\infty}((0, T)\times\mathbb R^n) \\
\cap C([0,T],H^{1/2-\dl}(\mathbb R^n))\cap 
C\bigl((0, T], H^{\f{m+1}{m+2}-\dl}(\mathbb
R^n)\bigr)\cap C^1\bigl([0,T], H^{\f{1}{m+2}-\dl}(\mathbb R^n)\bigr)
\end{multline*}
which satisfies, for $t\in (0, T]$,
\begin{multline}\label{4-17}
\|u_2(t,\cdot)\|_{L^{\infty}(\mathbb R^n)}+\|u_2(t,\cdot)\|_{H^{1/2-\dl}(\mathbb R^n)}+
t^{m/4}\,\|u_2(t,\cdot)\|_{H^{\f{m+1}{m+2}-\dl}(\mathbb
R^n)} \\+\, \|\p_tu_2(t,\cdot)\|_{H^{\f{1}{m+2}-\dl}(\mathbb R^n)}\le C(
\dl).
\end{multline}

Because of $[\p_{x'}^\alpha, \p_t^2 -t^m \Delta_x]=0$, one obtains
higher regularity of $u_1(t,x)$ and $u_2(t,x)$ in the $x'=(x_2,\dots,
x_n)$ directions. In fact, for any $|\al|\ge 1$,
\[
\left\{ \enspace
\begin{aligned} 
& (\p_t^2-t^m\Delta_x) ( \p_{x'}^{\al}u_1)=0,\quad (t,x)\in
(0,+\infty)\times\mathbb R^n,\\
&\p_{x'}^{\al}u_1(0,x)=\p_{x'}^{\al} \varphi_0(x),\quad
\p_t\p_{x'}^{\al}u_1(0,x)=\p_{x'}^{\al}\vp_1(x)
\end{aligned}
\right.
\]
which gives
\begin{multline*}
\p_{x'}^{\al}u_1(t,x)\in L^{\infty}((0, T)\times\mathbb R^n)\\
\cap C([0,T],H^{1/2-\dl}(\mathbb R^n))\cap 
C\bigl((0, T], H^{\f{m+1}{m+2}-\dl}(\mathbb
R^n)\bigr)\cap C^1\bigl([0,T], H^{-\f{1}{m+2}-\dl}(\mathbb R^n)\bigr),
\end{multline*}
and, for any $t\in(0, T]$,
\begin{equation}\label{4-19}
\left\{ \enspace
\begin{aligned}
&\|\p_{x'}^{\al}u_1(t,\cdot)\|_{L^{\infty}(\mathbb
R^n)} \le C_{\al}\\
&\|\p_{x'}^{\al}u_1(t,\cdot)\|_{H^{1/2-\dl}(\mathbb R^n)}+
t^{m/4}\,\|\p_{x'}^{\al}u_1(t,\cdot)\|_{H^{\f{m+1}{m+2}-\dl}(\mathbb
R^n)}\\
& \hspace{50mm}+\,\|\p_t\p_{x'}^{\al}u_1(t,\cdot)\|_{H^{-\f{1}{m+2}-\dl}(\mathbb
R^n)}\le C_{\al}(\dl). 
\end{aligned}
\right.
\end{equation}
Furthermore, for $|\al|\ge 1$,
\begin{equation}\label{4-20}
\left\{ \enspace
\begin{aligned} 
& (\p_t^2 -t^m\Delta_x )
(\p_{x'}^{\al}u_2)=\p_{x'}^{\al}\varphi_2(x)+ \int_0^t
(\p_{x'}^{\al}f)(s,x,0)\,ds,\quad (t,x)\in (0,
+\infty)\times\mathbb R^n,\\
&\p_{x'}^{\al}u_2(0,x)=0,\quad \p_t\p_{x'}^{\al}u_1(0,x)=0
\end{aligned}
\right.
\end{equation}
which gives
\begin{multline*}
\p_{x'}^{\al}u_2(t,x)\in L^{\infty}((0, T)\times\mathbb R^n)\\ \cap \,
C([0,T],H^{1/2-\dl}(\mathbb R^n))\cap C\bigl((0, T], H^{\f{m+1}{m+2}-\dl}(\mathbb
R^n)\bigr)\cap C^1\bigl([0,T], H^{\f{1}{m+2}-\dl}(\mathbb R^n)\bigr)
\end{multline*}
and, for any $t\in (0, T]$,
\begin{multline}\label{4-21}
\|\p_{x'}^{\al}u_2(t,\cdot)\|_{L^{\infty}(\mathbb R^n)}
+\|\p_{x'}^{\al}u_2(t,\cdot)\|_{H^{1/2-\dl}(\mathbb R^n)}+
t^{m/4}\,\|\p_{x'}^{\al}u_2(t,\cdot)\|_{H^{\f{m+1}{m+2}-\dl}(\mathbb
R^n)}\\
+\|\p_t\p_{x'}^{\al}u_2(t,\cdot)\|_{H^{\f{1}{m+2}-\dl}(\mathbb
R^n)}\le C( \dl).
\end{multline}

Set $v(t,x)=u(t,x)-u_1(t,x)-u_2(t,x)$. Then one has from \eqref{4-1}
that
\begin{equation}\label{4-22}
\left\{ \enspace
\begin{aligned} 
&\p_t^2v-t^{m}\Delta_x v=\int_0^t \big(f(s,x,u_1+u_2+v)-f(s,x,0)\big)\, ds,\\
&v(0,x)=\p_t v(0,x)=0.
\end{aligned}
\right.
\end{equation}
In order to solve \eqref{4-1}, it suffices to solve \eqref{4-22}. This
requires to establish an \textit{a priori\/} $L^{\infty}$ bound on
$\p_{x'}^{\al}v$ in \eqref{4-22} for $|\al|\le \left[n/2\right]+1$. To
this end, motivated by Lemma~\ref{lem2-3} and Lemma~\ref{2-5}, one
should establish $\p_{x'}^{\al+\beta}v\in L^{\infty}((0, T),
H^s(\mathbb R^n))$ with some constant $s>1/2$ and $|\beta|\le
\left[n/2\right]+1$.

Applying $\p_{x'}^{\g}$ $(|\g|\le 2\left[n/2\right]+2)$ on both sides of
\eqref{4-22} yields
\begin{equation}\label{4-23}
\left\{ \enspace
\begin{aligned}
&\p_t^2\p_{x'}^{\g}v-t^{m}\Delta_x \p_{x'}^{\g}v=
F_{\g}(t,x, \p_{x'}^{\al}v)_{|\al|\le|\g|}\\
&\qquad\qquad \equiv
\sum_{|\beta|+l\le |\g|}C_{\beta l}\int_0^t
\bigl((\p_{x'}^{\beta}f)(s,x,u_1+u_2+v)-(\p_{x'}^{\beta}f)(s,x,0)\bigr)\\
&\hspace{40mm} \times \p_u^{l}f(s,x,u_1+u_2+v)\prod_{\substack{1\le k\le l\\
\beta_1+\dots+\beta_l=l}}\p_{x'}^{\beta_k}(u_1+u_2+v)\,ds,\\
&\p_{x'}^{\g}v(0,x)=\p_t \p_{x'}^{\g}v(0,x)=0.
\end{aligned}
\right.
\end{equation}
If 
\[
  \sum_{|\al|\le [n/2]+1}\|\p_{x'}^{\al}v\|_{L^{\infty}((0,
    T)\times\mathbb R^n)}+\ds\sum_{|\g|\le
  2[n/2]+2}\|\p_{x'}^{\g}v\|_{L^{\infty}((0, T), H^s(\mathbb
  R^n))}\le 2, 
\]
where $s=\frac{2}{m+2}+1/2-\dl>1/2$ (for $\dl>0$ small) and $T\le 1$,
then by Lemma~\ref{lem2-5}, \eqref{4-19}, and \eqref{4-21}, one has
from \eqref{4-23} that, for $T>0$ small,
\begin{equation}\label{4-24}
\sum_{|\al|\le [n/2]+1}\|\p_{x'}^{\al}v\|_{L^{\infty}((0,
T)\times\mathbb R^n)}\le 1.
\end{equation}

Relying the preparations above, we will now use the Banach fixed-point
theorem to establish Theorem~\ref{thm4-2}. For $w\in L^{\infty}((0,
T)\times\mathbb R^n)\cap C([0,T], H^{1/2+p_0(m)-\dl})\cap
C\bigl((0,T], H^{1/2+p_1(m)-\dl}\bigr) \cap C^1\bigl([0,T],$
  \linebreak $H^{\f{1}{m+2}+p_2(m)-\dl}\bigr)$ with $\p_{x'}^{\al}w\in
  L^{\infty}((0, T)\times\mathbb R^n)$ ($|\al|\le \left[n/2\right]+1$)
  and $\p_{x'}^{\g}w\in C([0,T], H^{1/2+p_0(m)-\dl})\cap C\bigl((0,T],
    H^{1/2+p_1(m)-\dl}\bigr)\cap C^1\bigl([0,T],
    H^{\f{1}{m+2}+p_2(m)-\dl}\bigr)$ ($|\g|\le 2\left[n/2\right]+2$),
    define
\begin{multline*} 
\interleave w\interleave \equiv
\sum_{|\al|=0}^{[n/2]+1}\|\p_{x'}^{\al}w\|_{L^{\infty}((0,T)\times
\mathbb R^n)}+ 
 \sup_{0 \le t \le T }\sum_{|\g|=0}^{2[n/2]+2}\|\p_{x'}^{\g}w(t,\cdot)\|_{H^{1/2+
p_0(m)-\dl}(\mathbb R^n)}\\
  +\, \sup_{0 < t \le T }t^{\f{(m+2)p_1(m)}{2}-2}
 \sum_{|\g|=0}^{2[n/2]+2}\|\p_{x'}^{\g}w(t,\cdot)\|_{H^{1/2+
p_1(m)-\dl}(\mathbb R^n)}\\
+\,\sup_{0 \le t \le T
}
\sum_{|\g|=0}^{2[n/2]+2}\|\p_t\p_{x'}^{\g}w(t,\cdot)
\|_{H^{\f{1}{m+2}+p_2(m)-\dl}(\mathbb
R^n)}.
\end{multline*}
The set $Q$ is defined by
\[
Q\equiv \Bigl\{w\in L^{\infty}((0, T)\times\mathbb R^n)\cap C([0,T],
H^{1/2-\dl}(\mathbb R^n))\cap C^1\bigl([0,T], H^{-\f{1}{m+2}-\dl}(\mathbb
R^n)\bigr)\colon \interleave w\interleave \leq 2\Bigr\}.
\]
Further define a nonlinear map $\mathcal F$ by 
\[
\mathcal F(w)=E\left(f(t,x,u_1+u_2+w)-f(t,x,0)\right),
\]
where the operator $E$ has been introduced in \eqref{4-8}.

As in the proof of Theorem~\ref{thm4-1}, we now show that $\mathcal F$
maps $Q$ into itself and that it is strongly contractible for $T>0$
small. Indeed, $\mathcal F(w)$ for $w\in Q$ solves the problem
\[
\left\{ \enspace
\begin{aligned} 
&(\p_t^2-t^m\Delta_x)\mathcal F(w)=\int_0^t
\bigl(f(s,x,u_1+u_2+w)-f(s,x,0)\bigr)\, ds,\\
&\mathcal F(w)|_{t=0}=\p_t\mathcal F(w)|_{t=0}=0.
\end{aligned}
\right.
\]
From \eqref{4-24} one concludes that, for $T>0$ small,
\begin{equation}\label{4-26}
\sum_{|\al|=0}^{[n/2]+1}\|\p_{x'}^{\al}\mathcal Fw(t,x)
\|_{L^{\infty}((0,T)\times\mathbb R^n)}\le  1.
\end{equation}
Similar to the proof of Theorem~\ref{thm4-1}, one has, for $T>0$
small,
\begin{multline}\label{4-27}
\sup_{0 \le t \le T
}\sum_{|\g|=0}^{2[n/2]+2}\|\p_{x'}^{\g}w(t,\cdot)\|_{H^{\f{m+6}{2(m+2)}+
    p_0(m)-\dl}(\mathbb R^n)} \\ +\, \sup_{0 < t \le T
}t^{\f{(m+2)p_1(m)}{2}-2}
\sum_{|\g|=0}^{2[n/2]+2}\|\p_{x'}^{\g}w(t,\cdot)\|_{H^{\f{m+6}{2(m+2)}+
    p_1(m)-\dl}(\mathbb R^n)} \\ +\, \sup_{0 \le t \le T
}\sum_{|\g|=0}^{2[n/2]+2}\|\p_t\p_{x'}^{\g}w(t,\cdot)\|_{H^{\f{3}{m+2}+p_2(m)-\dl}(\mathbb
  R^n)}\le 1
\end{multline}
and
\[
\interleave \mathcal F(w_1)-\mathcal F(w_2)\interleave
\le\f12\interleave w_1-w_2\interleave,
\]
where $w_1, w_2\in Q$.
Combining \eqref{4-26} with \eqref{4-27} yields
\[
\interleave \mathcal F(w)\interleave \le 2,
\]
i.e., $\mathcal F$ maps $Q$ into itself. 

Invoking the Banach fixed-point theorem completes the proof of
Theorem~\ref{thm4-2}.
\end{proof}

Next we study the local existence of solutions of \eqref{1-2}.

\begin{theorem}\label{thm4-3} 

\textup{(i)} \ Under assumption \eqref{a2}, if $f(t,x,u)$ satisfies
the assumption in Theorem~\ref{thm1-2}, then there is a constant $T>0$
such that \eqref{1-2} has a local solution $u\in
L_{\textup{loc}}^{\infty}((0, T]\times\mathbb R^n)\cap C([0,T],
  H^{n/2-}(\mathbb R^n))\cap C\bigl((0, T],
    H^{n/2+\f{m_2}{2(m_2+2)}-}(\mathbb R^n)\bigr)\cap C^1\bigl([0,T],
    H^{n/2-\f{m_2+4}{2(m_2+2)}-}(\mathbb R^n)\bigr)$.

\textup{(ii)} \ Under assumption \eqref{a1}, \eqref{1-2} has a unique
local solution $u\in L^{\infty}((0, T)\times\mathbb R^n)\cap C([0,
  T],H^{1/2-}(\mathbb R^n))$ \linebreak $\cap\, C \bigl((0, T],
  H^{\f{m_2+1}{m_2+2}-}(\mathbb R^n)\bigr)\cap C^1\bigl([0, T],
  H^{-\f{1}{m_2 +2}-}(\mathbb R^n)\bigr)$.
\end{theorem}

\begin{proof}  
(i) \  Let $u_1$ satisfy
\begin{equation}\label{4-29}
\left\{ \enspace
\begin{aligned} 
& \left(\p_t^2 -t^{m_1}\Delta_x \right)\left(\p_t^2 -t^{m_2}\Delta_x \right)
u_1(t,x)=0, \quad (t,x)\in (0,
+\infty)\times\mathbb R^n, \\
&u_1(0,x)=\varphi_0(x),\quad
\p_t{u}_1(0,x)=\vp_1(x),\q
\p_t^2 u_1(0,x)=\p_t^3 u_1(0,x)=0
\end{aligned}
\right.
\end{equation} 
and $u_2(t,x)$ satisfy
\begin{equation}\label{4-30}
\left\{ \enspace
\begin{aligned} 
& \left(\p_t^2 -t^{m_1}\Delta_x \right)\left(\p_t^2 -t^{m_2}\Delta_x \right)
u_2(t,x)=f(t,x,0), \quad (t,x)\in (0,
+\infty)\times\mathbb R^n, \\
&u_2(0,x)=
\p_t{u}_2(0,x)=0,\q
\p_t^2 u_2(0,x)=\varphi_2(x), \qquad \p_t^3 u_2(0,x)=\varphi_3(x).
\end{aligned}
\right.
\end{equation} 
Then $u_1$ is the solution of
\begin{equation}\label{4-31}
\left\{ \enspace
\begin{aligned} 
& \left(\p_t^2 -t^{m_2}\Delta_x \right) u_1(t,x)=0, \quad (t,x)\in
(0,+\infty)\times\mathbb R^n, \\
&u_1(0,x)=\varphi_0(x),\quad \p_t{u}_1(0,x)=\vp_1(x),
\end{aligned}
\right.
\end{equation} 
while $u_2$ is the solution of
\begin{equation}\label{4-32}
\left\{ \enspace
\begin{aligned} 
& \left(\p_t^2 -t^{m_2}\Delta_x \right) u_2(t,x)=g_1(t,x)+g_2(t,x),
\quad (t,x)\in [0,
+\infty)\times\mathbb R^n, \\
&u_2(0,x)=\p_t{u}_2(0,x)=0,
\end{aligned}
\right.
\end{equation} 
where $g_1(t,x)$ satisfies
\[
\left\{ \enspace
\begin{aligned} 
& \left(\p_t^2 -t^{m_1}\Delta_x \right) g_1(t,x)=0, \quad (t,x)\in
(0,+\infty)\times\mathbb R^n, \\
&g_1(0,x)=\varphi_2(x), \quad \p_t{g}_1(0,x)=\vp_3(x)
\end{aligned}
\right.
\]
and $g_2(t,x)$ satisfies
\[
\left\{ \enspace
\begin{aligned} 
& \left(\p_t^2 -t^{m_1}\Delta_x \right) g_2(t,x)=f(t,x,0), \quad
(t,x)\in (0,
+\infty)\times\mathbb R^n, \\
&g_2(0,x)=\p_t{g}_2(0,x)=0.
\end{aligned}
\right.
\]

From \eqref{4-31}, one has by Lemma~\ref{lem2-4} and
Lemma~\ref{lem2-1} that
\begin{multline*}
u_1\in L_{\textup{loc}}^\infty((0,T]\times \R^n) \\ \cap\, C([0,T],
H^{n/2-\dl}(\mathbb R^n)) \cap C\bigl((0, T],
H^{n/2+\f{m_2}{2(m_2+2)}-\dl}(\mathbb R^n)\bigr)\cap C^1\bigl([0,T],
H^{n/2-\f{m_2+4}{2(m_2+2)}-\dl}(\mathbb R^n)\bigr)
\end{multline*}
which satisfies, for any $t\in (0, T]$,
\[
\left\{ \enspace
\begin{aligned} &
\|u_1(t, \cdot)\|_{L^\infty(\R^n)}\le C(1+|\ln t|)^2,\\
&\|u_1(t,\cdot)\|_{H^{n/2-\dl}(\mathbb R^n)}+
t^{m_2/4}\,\|u_1(t,\cdot)\|_{H^{n/2+\f{m_2}{2(m_2+2)}-\dl}(\mathbb
R^n)}+\|\p_tu_1(t,\cdot)\|_{H^{n/2-\f{m_2+4}{2(m_2+2)}-\dl}(\mathbb
R^n)}\le C(\dl). 
\end{aligned}
\right.
\]
Further Lemma~\ref{lem2-1} implies that $g_1 \in C([0,T],
H^{n/2-\dl}(\mathbb R^n))$. In addition, one has from Lemma~\ref{lem2-3}
that $g_2 \in C([0,T], H^{n/2-\dl}(\mathbb R^n))$. Therefore, in view of
Lemma~\ref{lem2-5} and Lemma~\ref{lem2-3}, one obtains that
\begin{multline*}
u_2(t,x)\in L^{\infty}((0, T)\times\mathbb R^n) \\ \cap\, C([0,T],
H^{1/2-\dl}(\mathbb R^n))\cap C\bigl((0, T],
H^{\f{m_2+1}{m_2+2}-\dl}(\mathbb R^n)\bigr)\cap C^1\bigl([0,T],
H^{\f{1}{m_2+2}-\dl}(\mathbb R^n)\bigr)
\end{multline*}
which satisfies, for any $t\in [0, T]$,
\begin{multline*}
\|u_2(t,\cdot)\|_{L^{\infty}(\mathbb R^n)}+\|u_2(t,\cdot)\|_{H^{1/2-\dl}(\mathbb R^n)}+
t^{m_2/4}\|u_2(t,\cdot)\|_{H^{\f{m_2+1}{m_2+2}-\dl}(\mathbb
R^n)}\\
+\, \|\p_tu_2(t,\cdot)\|_{H^{\f{1}{m_2+2}-\dl}(\mathbb R^n)}\le C(\dl).
\end{multline*}

Set $w(t,x)=u(t,x)-u_1(t,x)-u_2(t,x)$. Then one has from \eqref{1-2},
\eqref{4-29}, and \eqref{4-30} that
\begin{equation}\label{4-37}
\left\{\enspace
\begin{aligned} 
& \left(\p_t^2 -t^{m_1}\Delta_x\right)\left(\p_t^2 -t^{m_2}\Delta_x\right)
w(t,x)\\
&\qquad\qquad =\,f(t,x, u_1+u_2+w)-f(t,x,0), \quad (t,x)\in (0,
+\infty)\times\mathbb R^n, \\
&\p_t^j{w}(0,x)=0,\quad  0\le j\le 3,
\end{aligned}
\right.
\end{equation}
which implies that $w$ satisfies
\begin{equation}\label{4-38}
\left\{\enspace
\begin{aligned} 
& (\p_t^2 -t^{m_2}\Delta_x ) w(t,x)=\int_0^t F(s, t, x; u)\,
ds, \quad (t,x)\in (0,
+\infty)\times\mathbb R^n, \\
&w(0,x)= \p_t{w}(0,x)=0,
\end{aligned}
\right.
\end{equation}
where
\begin{multline*}
F(s, t, x; u) \\
= \, \mathcal F_{\xi}^{-1} \bigl(\{ V_2(t, |\xi|)V_1(s, |\xi|) -
V_1(t, |\xi|)V_2(s, |\xi|) \} \mathcal F_x (f(t,x,
u(t,x))-f(t,x,0))(\xi)\bigr)(s, t, x).
\end{multline*}
Then following the arguments as in the proof of Theorem~\ref{thm4-1},
one obtains by the Banach fixed-point theorem that there is a unique
solution $w$ of \eqref{4-38} such that $ w\in C([0,T], H^{n/2+
  p_0(m_2)-\dl}(\R^n))\cap C\bigl((0,T],$ \linebreak
$H^{n/2+p_1(m_2)-\dl}(\R^n)\bigr) \cap C^1\bigl([0,T],
  H^{n/2-\f{m_2}{2(m_2+2)}+p_2(m_2)-\dl}(\R^n)\bigr)$. This shows
  that \eqref{1-2} has a local solution $u\in C([0,T],
  H^{n/2-}(\mathbb R^n))\cap C\bigl((0, T],
    H^{n/2+\f{m_2}{2(m_2+2)}-}(\mathbb R^n)\bigr)\cap C^1\bigl([0,T],
    H^{n/2-\f{m_2+4}{2(m_2+2)}-}(\mathbb R^n)\bigr)$.

\smallskip

(ii)\ Let $u_1$ and $u_2$ be defined as in \eqref{4-29} and
\eqref{4-30}, respectively. Then one infers from \eqref{4-31},
Lemma~\ref{lem2-6}, and Lemma~\ref{lem2-3} that
\begin{multline*}
u_1(t,x)\in L^{\infty}((0, T)\times\mathbb R^n)\\
\cap\, C([0,T],
H^{1/2-\dl}(\mathbb R^n))\cap C\bigl((0, T],
H^{\f{m_2+1}{m_2+2}-\dl}(\mathbb R^n)\bigr)\cap C^1\bigl([0,T],
H^{-\f{1}{m_2+2}-\dl}(\mathbb R^n)\bigr)
\end{multline*}
which satisfies, for any $t\in (0, T]$,
\[
\left\{ \enspace
\begin{aligned}
& \|u_1(t,\cdot)\|_{L^{\infty}(\mathbb R^n)}\le C,\\
&\|u_1(t,\cdot)\|_{H^{1/2-\dl}(\mathbb R^n)}+
t^{\f{m_2}{4}}\|u_1(t,\cdot)\|_{H^{\f{m_2+1}{m_2+2}-\dl}(\mathbb
R^n)}+\|\p_tu_1(t,\cdot)\|_{H^{-\f{1}{m_2+2}-\dl}(\mathbb R^n)}\le C(
\dl),
\end{aligned}
\right.
\]
and from \eqref{4-32}, Lemma~\ref{lem2-5}, and Lemma~\ref{lem2-3} that
\begin{multline*}
u_2(t,x)\in L^{\infty}((0, T)\times\mathbb R^n) \\ 
\cap\, C([0,T],
H^{1/2-\dl}(\mathbb R^n))\cap C\bigl((0, T],
H^{\f{m_2+1}{m_2+2}-\dl}(\mathbb R^n)\bigr)\cap C^1\bigl([0,T],
H^{\f{1}{m_2+2}-\dl}(\mathbb R^n)\bigr)
\end{multline*}
which satisfies, for any $t\in [0, T]$,
\[
\|u_2(t,\cdot)\|_{L^{\infty}(\mathbb R^n)}+\|u_2(t,\cdot)\|_{H^{1/2-\dl}(\mathbb R^n)}+
t^{\f{m_2}{4}}\|u_2(t,\cdot)\|_{H^{\f{m_2+1}{m_2+2}-\dl}(\mathbb
R^n)}+\|\p_tu_2(t,\cdot)\|_{H^{\f{1}{m_2+2}-\dl}(\mathbb R^n)}\le C(
\dl).
\]
Because of $[\p_{x'}^\alpha, (\p_t^2 -t^{m_1}\Delta_x )(\p_t^2
  -t^{m_2}\Delta_x )]=0$ and $[\p_{x'}^\alpha, \p_t^2
  -t^{m_j}\Delta_x]=0$ ($j=1,2$), one also has that, for any
$|\al|\ge 1$,
\[
\left\{ \enspace
\begin{aligned}
& \left(\p_t^2-t^{m_2}\Delta_x\right) ( \p_{x'}^{\al}u_1)=0,\quad
(t,x)\in (0,
+\infty)\times\mathbb R^n,\\
&\p_{x'}^{\al}u_1(0,x)=\p_{x'}^{\al} \varphi_0(x),\quad
\p_t\p_{x'}^{\al}u_1(0,x)=\p_{x'}^{\al}\vp_1(x)
\end{aligned}
\right.
\]
and
\[
\left\{ \enspace
\begin{aligned}
& \left(\p_t^2 -t^{m_2}\Delta_x \right) (
\p_{x'}^{\al}u_2)(t,x)=\p_{x'}^{\al} g_1(t,x)+\p_{x'}^{\al}
g_2(t,x), \quad (t,x)\in (0,
+\infty)\times\mathbb R^n, \\
&\p_{x'}^{\al} u_2(0,x)=\p_t \p_{x'}^{\al}{u}_2(0,x)=0,
\end{aligned}
\right.
\]
where $\p_{x'}^{\al}g_1(t,x)$  satisfies
\[
\left\{ \enspace
\begin{aligned}
& \left(\p_t^2 -t^{m_1}\Delta_x \right) \p_{x'}^{\al}g_1(t,x)=0,
\quad (t,x)\in (0,
+\infty)\times\mathbb R^n, \\
&\p_{x'}^{\al}g_1(0,x)=\p_{x'}^{\al}\varphi_2(x), \quad
\p_t\p_{x'}^{\al}{g}_1(0,x)=\p_{x'}^{\al}\vp_3(x)
\end{aligned}
\right.
\]
while $\p_{x'}^{\al}g_2(t,x)$ satisfies
\[
\left\{ \enspace
\begin{aligned}
& \left(\p_t^2 -t^{m_1}\Delta_x \right) \p_{x'}^{\al}g_2(t,x)=f(t,x),
\quad (t,x)\in (0,
+\infty)\times\mathbb R^n, \\
&\p_{x'}^{\al}g_2(0,x)=\p_t{\p_{x'}^{\al}g}_2(0,x)=0.
\end{aligned}
\right.
\]
Then
\begin{multline*}
\p_{x'}^\alpha u_1(t,x)\in L^{\infty}((0, T)\times\mathbb R^n)\\
\cap\, C([0,T],
H^{1/2-\dl}(\mathbb R^n))\cap C\bigl((0, T],
H^{\f{m_2+1}{m_2+2}-\dl}(\mathbb R^n)\bigr)\cap C^1\bigl([0,T],
H^{-\f{1}{m_2+2}-\dl}(\mathbb R^n)\bigr)
\end{multline*}
which satisfies, for any $t\in [0,
T]$,
\[
\left\{ \enspace
\begin{aligned} 
&\|\p_{x'}^\alpha u_1(t,\cdot)\|_{L^{\infty}(\mathbb
R^n)} \le C,\\
&\|\p_{x'}^\alpha u_1(t,\cdot)\|_{H^{1/2-\dl}(\mathbb R^n)}+
t^{\f{m_2}{4}}\|\p_{x'}^\alpha
u_1(t,\cdot)\|_{H^{\f{m_2+1}{m_2+2}-\dl}(\mathbb R^n)}+\|\p_t
\p_{x'}^\alpha u_1(t,\cdot)\|_{H^{-\f{1}{m_2+2}-\dl}(\mathbb R^n)}\le
C( \dl), 
\end{aligned}
\right.
\]
while
\begin{multline*}
\p_{x'}^\alpha u_2(t,x)\in L^{\infty}((0,
T)\times\mathbb R^n)\\ \cap\, 
C([0,T], H^{1/2-\dl}(\mathbb R^n))\cap C\bigl((0,
T], H^{\f{m_2+1}{m_2+2}-\dl}(\mathbb R^n)\bigr)\cap C^1\bigl([0,T],
H^{\f{1}{m+2}-\dl}(\mathbb R^n)\bigr)
\end{multline*}
which satisfies, for any $t\in [0, T]$,
\begin{multline*}
\|\p_{x'}^\alpha u_2(t,\cdot)\|_{L^{\infty}(\mathbb R^n)}+\|\p_{x'}^\alpha 
u_2(t,\cdot)\|_{H^{1/2-\dl}(\mathbb R^n)}+
t^{m/4}\,\|\p_{x'}^\alpha
u_2(t,\cdot)\|_{H^{\f{m+1}{m+2}-\dl}(\mathbb R^n)}\\ +\, \|\p_t
\p_{x'}^\alpha u_2(t,\cdot)\|_{H^{\f{1}{m+2}-\dl}(\mathbb R^n)}\le C(
\dl).
\end{multline*}

Set $w(t,x)=u(t,x)-u_1(t,x)-u_2(t,x)$. Then $w$
satisfies \eqref{4-37}  and also \eqref{4-38}. Thus $\p_{x'}^\g w$ satisfies
\begin{equation}\label{4-47}
\left\{ \enspace
\begin{aligned} 
&\p_t^2\p_{x'}^{\g}w-t^{m}\Delta_x \p_{x'}^{\g}w =
G_{\g}(t,x, \p_{x'}^{\al}v)_{|\al|\le|\g|} \\
&\qquad \qquad \equiv\,
 \sum_{|\beta|+l\le |\g|}C_{\beta l}\int_0^t
(\p_{x'}^{\beta}F)(s,t,x;u)\p_u^{l}F(s,t,x;u) \prod_{\substack{1\le k\le l\\
\beta_1+\dots+\beta_l=l}}\p_{x'}^{\beta_k}u(t,x)\,ds,\\
&\p_{x'}^{\g}w(0,x)=\p_t \p_{x'}^{\g}w(0,x)=0.
\end{aligned}
\right.
\end{equation}
Then following the same argument as in the proof of
Theorem~\ref{thm4-2}, one obtains by applying the Banach fixed-point
theorem that there is a unique bounded solution $w$ of \eqref{4-47}
such that $w\in L^{\infty}((0, T)\times\mathbb R^n)\cap C([0,T],
H^{1/2+p_0(m_2)-\dl})\cap C((0,T], H^{1/2+p_1(m_2)-\dl}) \cap
  C^1\bigl([0,T],H^{\f{1}{m_2+2}+p_2(m_2)-\dl}\bigr)$ with
  $\p_{x'}^{\al}w\in$ \linebreak $L^{\infty}((0, T)\times\mathbb R^n)$
  ($|\al|\le [n/2]+1$) and $\p_{x'}^{\g}w\in C([0,T],
  H^{1/2+p_0(m_2)-\dl})\cap C((0,T], H^{1/2+p_1(m_2)-\dl})$ $\cap\,
    C^1\bigl([0,T], H^{\f{1}{m_2+2}+p_2(m_2)-\dl}\bigr)$ ($|\g|\le
    2[n/2]+2$), and, therefore, \eqref{1-2} has a unique local solution
    $u\in$ \linebreak $L^{\infty}((0, T)\times\mathbb R^n)\cap
    C([0,T], H^{1/2-}(\mathbb R^n))\cap C\bigl((0, T],$
      $H^{\f{m_2+1}{m_2+2}-}(\mathbb R^n)\bigr)\cap C^1\bigl([0,T],
      H^{-\f{1}{m_2+2}-}(\mathbb R^n)\bigr)$.
\end{proof}


\section{Proof of Theorem~\ref{thm1-1}}\label{sec5}
Based on the results of Sections~\ref{sec2} to \ref{sec4}, we now
prove Theorem~\ref{thm1-1}. To this end, we first establish conormal
regularity of the local solutions $u(t,x)$ obtained in
Theorems~\ref{thm4-2} and \ref{thm4-3}\,(ii), respectively.

\begin{theorem}\label{thm5-1}
Under assumption \eqref{a1}, one has

\textup{(i)} $u(t,x)\in I^{\infty}H^{-\f{1}{m+2}-}(\G_{m}^{\pm}\cup
\Sigma_0)$ for the local solution $u$ of \textup{\eqref{1-1}},

\textup{(ii)} $u(t,x)\in
I^{\infty}H^{-\f{1}{m_2+2}-}(\G_{m_1}^{\pm}\cup\G_{m_2}^{\pm})$ for
the local solution $u$ of \textup{\eqref{1-2}}.
\end{theorem}

\begin{proof} 
(i) \ Note that for $\varphi_l$ ($l=0,1,2$) satisfying
assumption \eqref{a1}, one has $(x_1\p_1)^{k_1}\prod_{2\le i\le
n}\p_i^{k_i}\vp_l(x)\in L^\infty(\mathbb R^n) \cap H^{1/2-}(\mathbb R^n)$
for any $k_i\in\mathbb N_0$ ($1\le i=1\le n$). Thus, by the commutator
relations of Lemma 3.2, one has from Eq.~(4.1) that, for $k\ge 2$
and $j\ge 1$, there exists a $\Phi(x)\in L^\infty(\mathbb R^n) \cap
H^{1/2-}(\mathbb R^n)$ such that
\begin{equation}\label{5-1}
\left\{ \enspace
\begin{aligned} 
&\p_t^2U_k-t^m\Delta_x U_k=\sum_{\substack{\beta_0+l_0 \le k_0\\
\beta_{i}+l_{i}=k_{i}\\
\sum l_0^s+\sum l_i^s=l\le k}} C_{\beta l}  \int_0^t \bigl(
V^{\beta_0}\prod_{2\le i\le n}R_i^{\beta_i}\p_u^lf\bigr)(s,x,
u)\\
&\hspace{50mm} \times \prod_{1\le s\le l}\bigl(V^{l_0^s}\prod_{2\le i\le
n}R_i^{l_i^s}u\bigr)\,ds + \Phi(x),\\
&U_k(0,x)\in L^{\infty}(\mathbb R^n)\cap H^{1/2-}(\mathbb R^n),\quad
\p_tU_k(0,x)\in L^{\infty}(\mathbb R^n)\cap H^{1/2-}(\mathbb R^n),
\end{aligned}
\right.
\end{equation}
where $U_k=\Bigl\{V^{k_0}\prod_{2\le i\le n}R_i^{k_i}u\Bigr\}_{k_0+\sum k_i=k}$
for $k\in\mathbb N_0$.

By Lemma~\ref{lem2-1} together with Lemmas~\ref{lem2-2}--\ref{lem2-3}
and Lemma~\ref{lem2-6}, one has from \eqref{5-1} that
\begin{multline*}
U_k(t,x)\in  L^{\infty}((0, T)\times\mathbb R^n)\\
\cap C([0, T],H^{1/2-}(\mathbb
R^n)) \cap C((0, T], H^{\f{m+1}{m+2}-}(\mathbb R^n))\cap C^1([0, T],
H^{-\f{1}{m+2}-}(\mathbb R^n))
\end{multline*} 
which satisfies, for $t\in (0, T]$ and
any $\delta$ with $0<\delta \le \frac{1}{2(m+2)}$,
\begin{equation}\label{5-2}
\left\{ \enspace
\begin{aligned}
&  \|U_k(t,\cdot)\|_{L^{\infty}(\mathbb R^n)}\le
C_k,\\
&\|U_k(t,\cdot)\|_{H^{1/2-\dl}(\mathbb R^n)}+
t^{m/4}\|U_k(t,\cdot)\|_{H^{\f{m+1}{m+2}-\dl}(\mathbb R^n)}
+\|\p_tU_k(t,\cdot)\|_{H^{-\f{1}{m+2}-\dl}(\mathbb R^n)}\le C_k(\dl).
\end{aligned}
\right.
\end{equation}
The latter yields 
\begin{multline*}
N_{2,\pm}\biggl(h_2(t,x_1)\chi_{\pm}(\f{(m+2)x_1}{2t^{\f{m+2}{2}}})u\biggr)
=N_{2,\pm}(h_2\chi)u+\biggl( \f{(m+2)x_1}{2t^{\f{m+2}{2}}}\mp
1\biggr)h_2\chi t^{\f{m+2}{2}}\p_1u \\
\in L^{\infty}((0, T),
H^{-\f{1}{m+2}-\dl}),
\end{multline*}
where the functions $h_2$ and $\chi_{\pm}$ have been defined in
Definition~\ref{def3-5}. Furthermore, applying (2) of
Proposition~\ref{prop3-4} together with \eqref{5-2} yields
\begin{equation}\label{5-3}
N_{2,\pm}^{k_1}V^{k_0}\prod_{2\le i\le n}R_i^{k_i}(h_2\chi u)\in L^{\infty}((0, T),
H^{-\f{1}{m+2}-\dl}).
\end{equation}

Analogously, by (1), (3), and (4) in Proposition~\eqref{prop3-4} and
the same proof as for \eqref{5-3}, one obtains
\[
Z_1\dots Z_ku(t,x)\in L^{\infty}((0, T), H^{-\f{1}{m+2}-\dl})\quad
\text {on $W_i$ for $Z_1,\dots, Z_k\in \mathcal{M}_i$}, \enspace i=1,
3, 4.
\]
Therefore,
\[
  u(t,x)\in I^{\infty}H^{-\f{1}{m+2}-}(\G_m^{\pm} \cup \Gamma_0).
\]

\medskip

(ii) \ The solution $u$ of \eqref{1-2}
satisfies
\begin{equation}\label{5-4}
\left\{\enspace
\begin{aligned} 
& (\p_t^2 -t^{m_2}\Delta_x ) u(t,x)=v_1(t,x) + v_2(t,x,u),
\quad (t,x)\in (0,
+\infty)\times\mathbb R^n, \\
&u(0,x)=\varphi_0(x),\quad \p_t{u}(0,x)=\vp_1(x),
\end{aligned}
\right.
\end{equation}
where $v_1(t,x)$ satisfies
\[
\left\{\enspace
\begin{aligned} 
& (\p_t^2 -t^{m_1}\Delta_x ) v_1(t,x)=0, \quad (t,x)\in
(0,+\infty)\times\mathbb R^n, \\
&v_1(0,x)=\varphi_2(x), \quad \p_t{v}_1(0,x)=\vp_3(x)
\end{aligned}
\right.
\]
and $v_2(t,x)$ satisfies
\[
\left\{\enspace
\begin{aligned} 
& (\p_t^2 -t^{m_1}\Delta_x ) v_2(t,x)=f(t,x,u), \quad
(t,x)\in (0,+\infty)\times\mathbb R^n, \\
&v_2(0,x)=\p_t{v}_2(0,x)=0.
\end{aligned}
\right.
\]
By the commutator relations of Lemma~\ref{lem3-2}, one has from
Eq.~\eqref{5-4} that, for $k\ge 2$ and $j\ge 1$,
\begin{equation}\label{5-5}
\left\{ \enspace
\begin{aligned} 
& \left(\p_t^2-t^{m_2}\Delta_x \right) U_k=  \sum_{ \alpha\le
k_0} C_{\alpha}\Bigl((V^{(m_2)})^{\alpha}\prod_{1\le i<j\le
n}L_{ij}^{k_{ij}}
v_1\Bigr)(t,x)\\
& \hspace{30mm}  +
\sum_{\substack{\beta_0+l_0\le k_0\\
\beta_{ij}+l_{ij}=k_{ij}\\
\sum l_0^s+\sum l_{ij}^s=l\le k}} C_{\beta
l}\Bigl((V^{(m_2)})^{\beta_0}\prod_{1\le i<j\le n}L_{ij}^{\beta_{ij}}\p_u^l
v_2\Bigr)(t,x, u) \\
&\hspace{70mm}\times \prod_{1\le s\le l}\Bigl((V^{(m_2)})^{l_0^s}\prod_{1\le
i<j\le n}L_{ij}^{l_{ij}^s}u\Bigr),
 \\
&U_k(0,x)\in L^{\infty}(\mathbb R^n)\cap  H^{1/2-}(\mathbb R^n),\quad
\p_t U_k(0,x)\in L^{\infty}(\mathbb R^n)\cap H^{1/2-}(\mathbb R^n),
\end{aligned}
\right.
\end{equation}
where $U_k=\Bigl\{(V^{(m_2)})^{k_0}\prod_{2\le i\le n}R_i^{k_i}u\Bigr\}_{k_0+\sum
k_i=k}$ for  $k\in\mathbb N_0$.

By Lemma~\ref{lem2-1} together with Lemma~\ref{lem2-6} and
Lemma~\ref{lem2-2} together with Lemma~\ref{lem2-5}, one has from
\eqref{5-5} that
\begin{multline*}
U_k(t,x)\in  L^{\infty}([0, T]\times\mathbb R^n)\\
\cap C([0, T],H^{1/2-}(\mathbb
R^n)) \cap C\bigl((0, T], H^{\f{m_2+1}{m_2+2}-}(\mathbb R^n)\bigr)\cap C^1([0,
T], H^{-\f{1}{m_2+2}-}(\mathbb R^n))
\end{multline*}
which satisfies, for $t\in [0, T]$ and any fixed $\delta$ with
$0<\delta \le \frac{1}{2(m_2+2)}$,
\begin{equation}\label{5-6}
\left\{ \enspace
\begin{aligned}
& \|U_k(t,\cdot)\|_{L^{\infty}(\mathbb R^n)}\le
C_k,\\
& \|U_k(t,\cdot)\|_{H^{1/2-\dl}(\mathbb R^n)}+
t^{\f{m_2}{4}}\|U_k(t,\cdot)\|_{H^{\f{m_2+1}{m_2+2}-\dl}(\mathbb R^n)}
+\|\p_tU_k(t,\cdot)\|_{H^{-\f{1}{m_2+2}-\dl}(\mathbb R^n)}\le C_k(\dl).
\end{aligned}
\right.
\end{equation}
This yields that for $Z_{2,\pm}\equiv \left(x_1 \mp
\f{2}{m_2+2}\,t^{\f{m_2+2}{2}}\right)\p_1$,
\begin{multline*}
Z_{2,\pm}\biggl(h_2(t,x_1)\chi_{\pm}(\f{(m_2+2)x_1}{2t^{\f{m_2+2}{2}}})u\biggr)
\\
=Z_{2,\pm}(h_2\chi)u+\biggl( \f{(m_2+2)x_1}{2t^{\f{m_2+2}{2}}}\mp
1\biggr)h_2\chi_{\pm} t^{\f{m_2+2}{2}}\p_1u \in L^{\infty}([0, T],
H^{-\f{1}{m_2+2}-\dl}),
\end{multline*}
where the functions $h_2$ and $\chi_{\pm}$ have been defined in
Definition~\ref{def3-7} and we have used that $w_1(x)w_2(x)\in
H^{s-}(\mathbb R^n)$ for $w_1(x)\in H^{n/2-}(\mathbb R^n)$ and
$w_2(x)\in H^{s}(\mathbb R^n)$ with $-n/2<s<n/2$.  Moreover, applying
(2) of Proposition~\ref{prop3-4} together with \eqref{5-6} yields
\begin{equation}\label{5-7}
Z_{2,\pm}^{k_1}(V^{(m_2)})^{k_0}
\prod_{2\le i\le n}R_i^{k_i}(h_2\chi u)\in L^{\infty}((0, T),
H^{-\f{1}{m_2+2}-\dl}).
\end{equation}

Analogously, by (1), (3), and (4) in Proposition~\ref{prop3-4} and the
same proof as for \eqref{5-7}, one obtains 
\[
Z_1\dots Z_ku(t,x)\in
L^{\infty}((0, T), H^{-\f{1}{m_2+2}-\dl})\quad \text {on $E_i$ for
$Z_1,\dots, Z_k\in Y_i$}, \quad i=1, 3, 4, 5.
\]
Therefore, $u(t,x)\in I^{\infty}H^{-\f{1}{m_2+2}-}(\G_{m_1}^{\pm} \cup \G_{m_2}^{\pm})$.
\end{proof}

Finally, we complete the proof of Theorem~\ref{thm1-1}.

\begin{proof}[End of proof of Theorem~\ref{thm1-1}]
(i) \ Theorem~\ref{thm4-1} and Theorem~\ref{thm4-2}\,(ii) show the
  local existence of solutions of \eqref{1-1} and \eqref{1-2},
  respectively, under assumption \eqref{a1}.

(ii) \ From Theorem~\ref{thm5-1}, one then obtains for these local
  solutions that $u\in C^{\infty}\left(((0, T]\times\mathbb
    R^n)\setminus(\G_m^{\pm} \cup \G_0)\right)$ and $u\in
    C^{\infty}\left(((0, T]\times\mathbb
      R^n)\setminus(\G_{m_1}^{\pm}\cup\G_{m_2}^{\pm})\right)$,
      respectively.
\end{proof}


\section{Proof of Theorem~\ref{1-2}}\label{sec6}

Under assumption \eqref{a2}, we first establish conormal regularity
for local solutions $u(t,x)$ of \eqref{1-1} and \eqref{1-2}.

\begin{theorem}\label{thm6-1} 
Under assumption \eqref{a2}, one has 

\textup{(a)} \ $u(t,x)\in I^{\infty}H^{n/2-\f{m+4}{2(m+2)}-}(\G_m
\cup l_0)$ for the local solution $u$ of \eqref{1-1},

\textup{(b)} \ $u(t,x)\in I^{\infty}
H^{n/2-\f{m_2+4}{2(m_2+2)}-}(\G_{m_1}\cup \G_{m_2})$ for the
local solution $u$ of \eqref{1-2}.
\end{theorem}

\begin{proof}
(a) \ By the commutator relations of Lemma~\ref{lem3-1} and a direct
  computation, one has from \eqref{1-1} that
\begin{equation}\label{6-1}
\left\{ \enspace
\begin{aligned}
&\p_t \left(\p_t^2-t^m\Delta_x \right) U_k
= \sum_{\substack{\beta_0+\alpha_0\le k_0\\
\beta_{ij}+\alpha_{ij}=k_{ij}\\
\sum \alpha_0^s+\sum \alpha_{ij}^s=l\le k}} C_{\beta
l}\,\Bigl(L_0^{\beta_0}\prod_{1\le i<j\le n}L_{ij}^{\beta_{ij}}\p_u^lf\Bigr)(t,x,
u) \\
&\hspace{80mm} \times
\prod_{1\le s\le l}\Bigl(L_0^{\alpha_0^s}\prod_{1\le i<j\le
n}L_{ij}^{\alpha_{ij}^s}u\Bigr),\\
&U_k(0,x)\in L^\infty(\mathbb R^n) \cap H^{n/2-}(\mathbb R^n),\quad
\p_t^iU_k(0,x)\in L^\infty(\mathbb R^n) \cap H^{n/2-}(\mathbb R^n), \q
i=1,2,
\end{aligned}
\right.
\end{equation}
where $U_k=\Bigl\{L_0^{k_0}\prod_{1\le i<j\le
  n}L_{ij}^{k_{ij}}u\Bigr\}_{k_0+\sum k_{ij}=k}$ for $k\in\mathbb
N_0$. Note that in the process of deriving the regularity of
$U_k(0,x)$ and $\p_tU_k(0,x)$ we have used that $\prod_{1\le i, j\le
  n}(x_i\p_j)^{k_{ij}}\vp(x)$ $\in H^{n/2-}(\mathbb R^n)$ and
$w_1(x)w_2(x)\in H^{n/2-}(\mathbb R^n)$ for $w_1(x),\, w_2(x)\in
L^\infty(\mathbb R^n) \cap H^{n/2-}(\mathbb R^n)$.

We next prove by induction on $k$ that
\begin{multline}\label{6-2}
U_k(t,x)\in  L_{\textup{loc}}^{\infty}((0, T]\times\mathbb R^n)\\
\cap C([0,T], H^{n/2-})\cap C((0, T],
H^{n/2+\f{m}{2(m+2)}-}) \cap C^1([0,T],H^{n/2-\f{m+4}{2(m+2)}-}),
\end{multline} 
which satisfies, for any fixed small $\dl>0$,
\begin{equation}\label{6-3}
\left\{\enspace
\begin{aligned} 
& \|U_k(t, \cdot)\|_{L^\infty(\mathbb R^n)} \le C_k(\dl) (1+
|\ln t|^2),
\\
& \|U_k(t, \cdot)\|_{C([0, T], H^{n/2-\dl})}+t^{m/4}
\|U_k(t,\cdot)\|_{H^{n/2+\f{m}{2(m+2)}-\dl}} \\
& \hspace{60mm}+\|\p_tU_k(t,
\cdot)\|_{C([0, T], H^{n/2-\f{m+4}{2(m+2)}-\dl})}\le C_k(\dl).
\end{aligned}
\right.
\end{equation}
Note that \eqref{6-2}--\eqref{6-3} has been shown in
Theorem~\ref{thm4-1} in case of $k=0$. Assume that
\eqref{6-2}--\eqref{6-3} holds up to the order $k-1$. Then one has by
\eqref{6-1} that
\begin{equation}\label{6-4}
\left\{\enspace
\begin{aligned} 
&\p_t \left(\p_t^2-t^m\Delta_x \right)U_k-(\p_uf)(t,x,u)\,U_k=F_k(t,x),\\
&U_k(0,x)\in L^\infty(\mathbb R^n) \cap H^{n/2-}(\mathbb R^n),\quad
\p_t^i U_k(0,x)\in L^\infty(\mathbb R^n) \cap H^{n/2-}(\mathbb R^n),
\ i=1,2,
\end{aligned}
\right.
\end{equation}
where $F_k(t,x)\in C([0, T], H^{n/2-})$. From \eqref{6-4}, one sees
that $U_k$ satisfies
\[
\left\{\enspace
\begin{aligned}
&\left(\p_t^2-t^m\Delta_x \right)U_k- \int_0^t(\p_uf)(s,x,u)\,U_k\, ds
=G_k(t,x),\\
&U_k(0,x)\in L^\infty(\mathbb R^n) \cap  H^{n/2-}(\mathbb R^n),\quad
\p_tU_k(0,x)\in L^\infty(\mathbb R^n) \cap H^{n/2-}(\mathbb R^n),
\end{aligned}
\right.
\]
where $G_k(t,x)\in C([0, T], H^{n/2-})$. Then Lemma~\ref{lem2-1} and
Lemma~\ref{lem2-3}\,(i) yield \eqref{6-2}--\eqref{6-3} (for $k$).

We now prove that $u(t,x)\in I^{\infty}H^{n/2-}(\G_m \cup l_0)$. Note
that, for $1\le i\le n$, one has by \eqref{6-3} that
\begin{multline*}
N_2^i\left(h_2(t,x)\chi\left(\f{(m+2)|x|}{2t^{\f{m+2}{2}}}\right)u\right)
\\
=N_2^i\left(h_2\chi\right)u
+\f{2}{m+2}\left(\f{(m+2)|x|}{2t^{\f{m+2}{2}}}-1\right)h_2\chi
t^{\f{m+2}{2}}\p_iu\in L^{\infty}\bigl((0, T),H^{n/2-\f{m+4}{2(m+2)}-}\bigr),
\end{multline*}
where the definitions of $h_2(t,x)$ and
$\chi\left(\f{(m+2)|x|}{2t^{\f{m+2}{2}}}\right)$ have been given in
Definition~\ref{def3-3}. Furthermore, by Proposition~\ref{prop3-3} and
\eqref{6-3}, one obtains that, for any $k_i, k_0, k_{ij}\in\mathbb N_0$,
\begin{equation}\label{6-5}
(N_2^i)^{k_i}V_0^{k_0}\prod_{1\le i<j\le n}L_{ij}^{k_{ij}}(h_2\chi u)\in
L^{\infty}\bigl((0, T), H^{n/2-\f{m+4}{2(m+2)}-}\bigr).
\end{equation} 
Together with  Proposition~\ref{prop3-3}, this yields
\begin{equation}\label{6-6}
\bar{V}_i^{k_i}V_0^{k_0}\prod_{1\le i<j\le n}L_{ij}^{k_{ij}}(h_2\chi u)\in
L^{\infty}\bigl((0, T), H^{n/2-\f{m+4}{2(m+2)}-}\bigr).
\end{equation}

In order to show $u(t,x)\in I^{\infty}H^{n/2-}(\G_m \cup l_0)$,
we need to prove that
\[
\prod_{1\le i\le n}\bar{V}_i^{k_i}V_0^{k_0}\prod_{1\le i<j\le
n}L_{ij}^{k_{ij}}(h_2\chi u)\in L^{\infty}\bigl((0, T),
H^{n/2-\f{m+4}{2(m+2)}-}\bigr)
\] 
or equivalently
\begin{equation}\label{6-7}
\prod_{1\le i\le n}(N_2^i)^{k_i}V_0^{k_0}\prod_{1\le i<j\le
n}L_{ij}^{k_{ij}}(h_2\chi u)\in L^{\infty}\bigl((0, T),
H^{n/2-\f{m+4}{2(m+2)}-}\bigr).
\end{equation}
To this end, by the commutator relations of Lemma~\ref{lem3-1} and
Proposition~\ref{prop3-3}, it suffices to prove that
\begin{equation}\label{6-8}
N_2^{i_1}N_2^{i_2}\cdot\cdot\cdot N_2^{i_k}(h_2\chi u)\in
L^{\infty}\bigl((0, T), H^{n/2-\f{m+4}{2(m+2)}-}\bigr),\quad
1\le i_1<i_2<\cdot\cdot\cdot<i_k\le n, \ 2\le k\le n,
\end{equation}
because the proof on $N_2^{i_1}N_2^{i_2}\dots N_2^{i_k}V_0^{k_0}\prod_{1\le
  i<j\le n}L_{ij}^{k_{ij}}(h_2\chi u)\in L^{\infty}\bigl((0, T),
H^{n/2-\f{m+4}{2(m+2)}-}\bigr)$ is completely analogous. 

In fact, by $N_2^i\equiv a(t,x)\p_i$ with
$a(t,x)=|x|-\f{2}{m+2}\,t^{\f{m+2}{2}}$ and \eqref{6-5}, one has, for
$1\le i\le n$,
\begin{align*} 
\p_i^{2}\bigl(a^{2}(t,x)h_2\chi u \bigr)&= (a\p_i)^2(h_2\chi
u)+ \p_ia ( a \p_i)(h_2\chi u) +2a(\p_i^2a) h_2\chi u + 2 (\p_i a)^2 h_2\chi u \\
&= (N_2^i)^2 (h_2 \chi u) +(\p_i a) N_2^i(h_2 \chi u) +2a(\p_i^2a)
h_2\chi u + 2 (\p_i a)^2 h_2\chi u \\
&\hspace{60mm} \in L^{\infty}\bigl((0, T), H^{n/2-\f{m+4}{2(m+2)}-}\bigr),
\end{align*}
where we have used that $x_i/|x|\in H_{\textup{loc}}^{n/2-} (\mathbb
R^n)$ and $w_1(x)w_2(x)\in H^{\min\{s_1, s_2, s_1+s_2-n/2\}-}(\mathbb
R^n)$ for $w_1(x)\in H^{s_1}(\mathbb R^n)$ and $w_2(x)\in
H^{s_2}(\mathbb R^n)$ when $s_1,\, s_2\ge 0$. 
It follows that
\[
\Delta \bigl(a^{2}(t,x)h_2\chi u \bigr)\in 
L^{\infty}\bigl((0, T), H^{n/2-\f{m+4}{2(m+2)}-}\bigr)
\]
which gives by the regularity theory of second-order elliptic
equations
\[
\p_{ij}\left(a^2(t,x)h_2\chi u\right)\in 
L^{\infty}\bigl((0, T), H^{n/2-\f{m+4}{2(m+2)}-}\bigr), \q 1\le i<j\le n,
\]
or equivalently
\begin{equation}\label{6-12}
N_2^iN_2^j(h_2\chi u)\in L^{\infty}\bigl((0, T), H^{n/2-\f{m+4}{2(m+2)}-}\bigr),
\q 1\le i<j\le n.
\end{equation}

Analogously, one obtains, for any $1\le i, k\le n$,
\[
\p_i^2 \bigl(a^3\p_k(h_2\chi  u)\bigr)\in L^{\infty}\bigl((0, T), 
H^{n/2-\f{m+4}{2(m+2)}-}\bigr)
\] 
and
\[ 
\Delta \bigl(a^3\p_k(h_2\chi u)\bigr)\in L^{\infty}\bigl((0, T),
H^{n/2-\f{m+4}{2(m+2)}-}\bigr)
\]
which gives
\[
\p_{ij}\bigl(a^3\p_k(h_2\chi  u)\bigr)\in 
L^{\infty}\bigl((0, T), H^{n/2-\f{m+4}{2(m+2)}-}\bigr)
\] 
and further by \eqref{6-12}
\begin{equation}\label{6-13}
N_2^iN_2^jN_2^k(h_2\chi  u)\in 
L^{\infty}\bigl((0, T), H^{n/2-\f{m+4}{2(m+2)}-}\bigr).
\end{equation}

By induction, we have completed the proof of \eqref{6-8}.
Consequently, one has
\[
V_0^{k_0}\prod_{1\le i\le n}\bar{V}_i^{k_i} \prod_{1\le i<j\le
n}L_{ij}^{k_{ij}}(h_2\chi  u)  \in L^{\infty}\bigl((0, T),
H^{n/2-\f{m+4}{2(m+2)}-}\bigr).
\]
Similarly, by (1), (3), and (4) of Proposition~\ref{prop3-3} (note
that $\bar V_i$ can be expressed as a linear combination of $V_0$ and
$L_{jk}$ with admissible coefficients in $\O_1$, $\O_3$, and $\O_4$,
respectively), one arrives at
\begin{gather*}
Z_1\dots Z_k \bigl( h_1(t,x) u(t,x) \bigr)\in
L^{\infty}\bigl((0, T), H^{n/2-\f{m+4}{2(m+2)}-}\bigr),\quad 
Z_1,\dots, Z_k\in \mathcal{S}_1, \\
Z_1\dots Z_k \left( h_3(t,x)\chi_1\left(\f{(m+2)|x|}{2t^{\f{m+2}{2}}}
\right)u\right) \in
L^{\infty}\bigl((0, T), H^{n/2-\f{m+4}{2(m+2)}-}\bigr),\quad 
Z_1,\dots, Z_k\in \mathcal{S}_3, \\
\intertext{and}
Z_1\dots Z_k \left( h_4(t,x)\chi_2\left(\f{(m+2)|x|}{2t^{\f{m+2}{2}}}
\right)u\right) \in
L^{\infty}\bigl((0, T), H^{n/2-\f{m+4}{2(m+2)}-}\bigr),\quad
Z_1,\dots, Z_k\in \mathcal{S}_4,
\end{gather*}
where the functions $h_1, \, h_3, \, h_4$, and $\chi_i$ ($1\le i \le
2$) have been given in Definition~\ref{def3-3}.
Therefore,
\[
u(t,x)\in I^{\infty}H^{n/2-\f{m+4}{2(m+2)}-}(\G_m \cup l_0),
\]
as required.

\smallskip

(b) \ By the commutator relations of Lemma~\ref{lem3-1} and a direct
computation, one from \eqref{5-4} that
\begin{equation}\label{6-15}
\left\{\enspace
\begin{aligned} 
& \left(\p_t^2-t^{m_2}\Delta_x \right) U_k=  \sum_{ \alpha\le
k_0} C_{\alpha}\,\Bigl((V_0^{(m_2)})^{\alpha}\prod_{1\le i<j\le
n}L_{ij}^{k_{ij}}v_1\Bigr)(t,x)\\
& \hspace{35mm} +
\sum_{\substack{\beta_0+\alpha_0\le k_0\\
\beta_{ij}+\alpha_{ij}=k_{ij}\\
\sum \alpha_0^s+\sum \alpha_{ij}^s=l\le k}} C_{\beta
l}\,\Bigl((V_0^{(m_2)})^{\beta_0}\prod_{1\le i<j\le
n}L_{ij}^{\beta_{ij}}\p_u^l v_2\Bigr)(t,x, u)\\
&\hspace{70mm} \times \prod_{1\le s\le
l}\Bigl((V_0^{(m_2)})^{\alpha_0^s}\prod_{1\le i<j\le
n}L_{ij}^{\alpha_{ij}^s}u\Bigr),
 \\
&U_k(0,x)\in L^\infty(\mathbb R^n)\cap H^{n/2-}(\mathbb R^n),\quad
\p_t U_k(0,x)\in L^\infty(\mathbb R^n)\cap H^{n/2-}(\mathbb R^n),
\end{aligned}
\right.
\end{equation}
where $U_k=\Bigl\{(V_0^{(m_2)})^{k_0}\prod_{1\le i<j\le
  n}L_{ij}^{k_{ij}}u\Bigr\}_{k_0+\sum k_{ij}=k}$ for $k\in\mathbb
N_0$; $v_1$ and $v_2$ have been defined in \eqref{5-4}.

We next prove by induction on $k$ that
\begin{multline}\label{6-16}
U_k(t,x)\in L^\infty_{\textup{loc}}((0,T] \times \R^n) \\ \cap
  C([0,T], H^{n/2-})\cap C\bigl((0, T], H^{n/2+\f{m_2}{2(m_2+2)}-}\bigr)
    \cap C^1\bigl([0,T], H^{n/2-\f{m_2+4}{2(m_2+2)}-}\bigr)
\end{multline}
which satisfies, for any small $\dl$ with $0< \delta <
\frac{1}{2(m_2+2)}$,
\begin{equation}\label{6-17}
\left\{ \enspace
\begin{aligned}
& \|U_k(t, \cdot)\|_{L^\infty(\R^n)} \le C_k(\dl)(1+|\ln t|)^2,
\\
&\|U_k\|_{C([0, T], H^{n/2-\dl})}+t^{m_2/4}
\|U_k(t,\cdot)\|_{H^{n/2+\f{m_2}{2(m_2+2)}-\dl}} \\
&\hspace{50mm}
+\|\p_tU_k\|_{C([0, T], H^{n/2-\f{m_2+4}{2(m_2+2)}-\dl})}\le
C_k(\dl).  
\end{aligned}
\right.
\end{equation}
Note that \eqref{6-16}--\eqref{6-17} have been shown in
Theorem~\ref{thm4-3}\,(i) in case $k=0$. Assume that
\eqref{6-16}--\eqref{6-17} holds up to the order $k-1$. Then one has
from \eqref{6-15} that
\begin{equation}\label{6-18}
\left\{\enspace
\begin{aligned} & \left(\p_t^2-t^{m_2}\Delta_x \right)U_k=F_k(t,x),\\
&U_k(0,x)\in L^\infty(\mathbb R^n)\cap H^{n/2-}(\mathbb R^n),\quad
\p_t U_k(0,x)\in L^\infty(\mathbb R^n)\cap H^{n/2-}(\mathbb R^n),
\end{aligned}
\right.
\end{equation}
where $F_k(t,x)\in L^p((0, T), H^{n/2-})$ with any $1<p<\infty$. Then
Lemma~\ref{lem2-1} and Lemma~\ref{lem2-3}\,(i) yield
\eqref{6-16}--\eqref{6-17} (for $k$).

We now prove that $u(t,x)\in I^{\infty}H^{n/2-}(\G_{m_1} \cup
\G_{m_2})$. On $D_2$, set
$Z_2^i=\left(|x|-\f{2}{m_2+2}\,t^{\f{m_2+2}{2}}\right)\p_i$ for $1 \le
i \le n$. Note that by \eqref{6-17}, for $1\le i\le n$,
\begin{multline*}
Z_2^i\left(h_2(t,x)\chi\left(\f{(m_2+2)|x|}{2t^{\f{m_2+2}{2}}}\right)u\right)
\\
=Z_2^i(h_2\chi)u
+\f{2}{m_2+2}\Bigl(\f{(m_2+2)|x|}{2t^{\f{m_2+2}{2}}}-1\Bigr)h_2\chi
t^{\f{m_2}{2}+1}\p_iu\in L^{\infty}\bigl((0, T),
H^{n/2-\f{m_2+4}{2(m_2+2)}-}\bigr),
\end{multline*}
where the definitions of $h_2(t,x)$ and
$\chi(\f{(m_2+2)|x|}{2t^{\f{m_2+2}{2}}})$ have been given in
Definition~\ref{def3-7}. Furthermore, by Proposition~\ref{prop3-3} and
\eqref{6-17}, one obtains that, for any $k_i, k_0, k_{ij}\in\mathbb N_0$,
\begin{equation}\label{6-19}
(Z_2^i)^{k_i}(V_0^{(m_2)})^{k_0}\prod_{1\le i<j\le n}L_{ij}^{k_{ij}}(h_2\chi u)\in
 L^{\infty}\bigl((0, T), H^{n/2-\f{m_2+4}{2(m_2+2)}-}\bigr).
\end{equation}
Together with Proposition~\ref{prop3-3}, this yields
\[
(\bar{V}_i^{(m_2)})^{k_i}(V_0^{(m_2)})^{k_0}\prod_{1\le i<j\le
    n}L_{ij}^{k_{ij}}(h_2\chi u)\in L^{\infty}\bigl((0, T),
  H^{n/2-\f{m_2+4}{2(m_2+2)}-}\bigr).
\]

In order to show that $u(t,x)\in I^{\infty}H^{n/2-}(\G_{m_1} \cup
\G_{m_2})$, we need to prove that
\[
\prod_{1\le
i\le n}(\bar{V}_i^{(m_2)})^{k_i}(V_0^{(m_2)})^{k_0}\prod_{1\le
i<j\le n}L_{ij}^{k_{ij}}(h_2\chi u)\in L^{\infty}\bigl((0, T),
H^{n/2-\f{m_2+4}{2(m_2+2)}-}\bigr)
\] 
or equivalently
\[
\prod_{1\le
i\le n}(Z_2^i)^{k_i}(V_0^{(m_2)})^{k_0}\prod_{1\le i<j\le
n}L_{ij}^{k_{ij}}(h_2\chi u)\in  L^{\infty}\bigl((0, T),
H^{n/2-\f{m_2+4}{2(m_2+2)}-}\bigr).
\]
To this end, by the commutator relations of Lemma~\ref{lem3-1} and
Proposition~\ref{prop3-3}, it suffices to prove that
\begin{equation}\label{6-22}
Z_2^{i_1}Z_2^{i_2}\dots Z_2^{i_k}(h_2\chi u)\in
L^{\infty}\bigl((0, T), H^{n/2-\f{m_2+4}{2(m_2+2)}-}\bigr),\quad
1\le i_1<i_2<\cdot\cdot\cdot<i_k\le n, \ 2\le k\le n,
\end{equation}
because the proof on $Z_2^{i_1}Z_2^{i_2}\dots
Z_2^{i_k}(V_0^{(m_2)})^{k_0}\prod_{1\le i<j\le
  n}L_{ij}^{k_{ij}}(h_2\chi u)\in L^{\infty}\bigl((0, T),
H^{n/2-\f{m_2+4}{2(m_2+2)}-}\bigr)$ is completely analogous.

In fact, by $Z_2^i\equiv b(t,x)\p_i$ with
$b(t,x)=|x|-\f{2}{m_2+2}\,t^{\f{m_2+2}{2}}$ and \eqref{6-19}, one has,
for $1\le i\le n$,
\begin{equation}\label{6-23}
\begin{aligned} 
\p_i^{2}\bigl(b^{2}(t,x)h_2\chi u \bigr)&= (b\p_i)^2(h_2\chi
u)+ \p_ib ( b \p_i)(h_2\chi u) +2b(\p_i^2b) h_2\chi u + 2 (\p_i b)^2 h_2\chi u \\
&= (Z_2^i)^2 (h_2 \chi u) +(\p_i b) Z_2^i(h_2 \chi u) +2b(\p_i^2b)
h_2\chi u + 2 (\p_i b)^2 h_2\chi u \\
& \hspace{70mm}\in  L^{\infty}\bigl((0, T), H^{n/2-\f{m_2+4}{2(m_2+2)}-}\bigr).
\end{aligned}
\end{equation}
From \eqref{6-23}, one has 
\[
\Delta \bigl(b^{2}(t,x)h_2\chi u \bigr)\in  
L^{\infty}\bigl((0, T), H^{n/2-\f{m_2+4}{2(m_2+2)}-}\bigr)
\]
which gives  by the regularity theory of second-order elliptic
equations
\[
\p_{ij}\bigl(b^2(t,x)h_2\chi u\bigr)\in  
L^{\infty}\bigl((0, T), H^{n/2-\f{m_2+4}{2(m_2+2)}-}\bigr), \quad 1\le i<j\le n,
\]
or equivalently
\begin{equation}\label{6-26}
Z_2^i Z_2^j(h_2\chi u)\in L^{\infty}\bigl((0, T), H^{n/2-\f{m_2+4}{2(m_2+2)}-}\bigr)
\quad\text{for any $1\le i<j\le n$}.
\end{equation}

Analogously, we can get, for any $1\le i, k\le n$,
\[
\p_i^2 \bigl(b^3\p_k(h_2\chi  u)\bigr)\in L^{\infty}\bigl((0, T), 
H^{n/2-\f{m_2+4}{2(m_2+2)}-}\bigr)
\]
and
\[
\Delta \bigl(b^3\p_k(h_2\chi  u)\bigr)\in 
L^{\infty}\bigl((0, T), H^{n/2-\f{m_2+4}{2(m_2+2)}-}\bigr)
\] 
which gives
\[
\p_{ij}\bigl(b^3\p_k(h_2\chi  u)\bigr)\in 
L^{\infty}\bigl((0, T), H^{n/2-\f{m_2+4}{2(m_2+2)}-}\bigr)
\] 
and further by \eqref{6-26},
\begin{equation}\label{6-27}
Z_2^i Z_2^j Z_2^k(h_2\chi  u)\in 
L^{\infty}\bigl((0, T), H^{n/2-\f{m_2+4}{2(m_2+2)}-}\bigr).
\end{equation}

By induction, we have completed the proof of \eqref{6-22}.
Consequently, one has
\begin{equation}\label{6-28}
(V_0^{(m_2)})^{k_0} \prod_{1\le i\le n}({\bar{V}_i^{(m_2)}})^{k_i}
\prod_{1\le i<j\le
n}L_{ij}^{k_{ij}}(h_2\chi  u)  \in L^{\infty}\bigl((0, T),
H^{n/2-\f{m_2+4}{2(m_2+2)}-}\bigr).
\end{equation}
Similarly, by (1), (3) and (4) of Proposition~\ref{prop3-3} (note that
$M_i$ can be expressed as a linear combination of $M$ and $L_{jk}$
with admissible coefficients in $\O_1$, $\O_3$, and $\O_4$,
respectively), one arrives at
\begin{gather*}
Z_1\dots Z_k \bigl( h_1(t,x) u(t,x) \bigr)\in L^{\infty}\bigl((0, T),
H^{n/2-\f{m_2+4}{2(m_2+2)}-}\bigr),\quad Z_1,\dots, Z_k\in X_1,
\\ Z_1\dots Z_k \left(
h_3(t,x)\chi_1\left(\f{(m_2+2)|x|}{2t^{\f{m_2+2}{2}}}\right)
\chi_2\left(\f{(m_1+2)|x|}{2t^{\f{m_1+2}{2}}}\right)u\right) \in
L^{\infty}\bigl((0, T), H^{n/2-\f{m_2+4}{2(m_2+2)}-}\bigr),
\\ \intertext{for all $Z_1,\dots, Z_k\in X_3$, and since $t^{m_1 +1
  }\p_\ell = t^{m_1-m_2} t^{m_2 +1 }\p_\ell$, one also has, for all
  $Z_1,\dots, Z_k\in X_5$,} Z_1\dots Z_k
\left(h_5(t,x)\chi_5\left(\f{(m_2+2)|x|}{2t^{\f{m_2+2}{2}}}\right)
\chi_6\left(\f{(m_1+2)|x|}{2t^{\f{m_1+2}{2}}}\right)u\right) \in
L^{\infty}\bigl((0, T), H^{n/2-\f{m_2+4}{2(m_2+2)}-}\bigr),
\end{gather*}
where the functions $h_1, \, h_3, \, h_5$ ,and $\chi_i$ ($i =1, 2, 4,
5$) have been given in Definition~\ref{def3-7}.

Because 
\[
  V_0^{(m_1)}= 2V_0^{(m_2)} -\frac{(m_2+2)(m_1-2m_2-2)
  }{(m_2+2)^2 |x|^2 -4 t^{m_2+2}}\,|x|^2 V_0^{(m_2)}+
  \frac{2\left(m_1-2m_2-2\right)t^{\frac{m_2}{2}+1}}{(m_2+2)^2 |x|^2 -4
  t^{m_2+2}}\,\sum_{k=1}^n x_k \bar{V}_k^{(m_2)},
\]
one has from \eqref{6-16} and \eqref{6-28} by an argument similar to
the one dealing with $\bar{V}_j^{(m_2)}$ ($1 \le j \le n$) that, for
all $Z_1,\dots, Z_k\in X_4$,
\[
Z_1\dots
Z_k\left(h_4(t,x)\chi_3\left(\f{(m_2+2)|x|}{2t^{\f{m_2+2}{2}}}\right)
\chi_4\left(\f{(m_1+2)|x|}{2t^{\f{m_1+2}{2}}}\right)u(t,x)
\right)\in L^{\infty}\bigl((0, T), H^{n/2-\f{m_2+4}{2(m_2+2)}-}(\mathbb R^{n}).
\]
Therefore, $u(t,x)\in I^{\infty}H^{n/2-\f{m_2+4}{2(m_2+2)}-}(\G_{m_1}
\cup \G_{m_2})$, as required.
\end{proof}

\smallskip

Finally, we finish the proof of Theorem~\ref{thm1-2}.

\begin{proof}[End of proof of Theorem~\ref{thm1-2}]
(i) \ From Theorem~\ref{thm4-1} and Theorem~\ref{thm4-3}\,(ii), one
  obtains the local existence of solutions of \eqref{1-1} and
  \eqref{1-2} under assumption \eqref{a2}.

(ii) \ Based on Theorem~\ref{thm6-1}, one then sees that $u(t,x)\in
  I^{\infty}H^{n/2-\f{m+4}{2(m+2)}-}(\G_m \cup l_0)$ for the solution
  $u$ of \eqref{1-1} and $u(t,x)\in
  I^{\infty}H^{-\f{1}{m+2}-}(\G_{m_1}\cup \G_{m_2})$ for the
  solution $u$ of \eqref{1-2}.

Thus, the proof of Theorem~\ref{thm1-2} is finished.
\end{proof}


\nocite{*}
\bibliographystyle{plain}
\bibliography{RWY}


\end{document}